\documentclass[final]{amsart}
\usepackage{amsfonts,amssymb,amsmath,amsgen,amsthm,datetime}
\usepackage{hyperref}
\usepackage{mathrsfs}
\usepackage{placeins}
\usepackage[top=0.8in, bottom=0.9in, scale=0.7, centering]{geometry}
\usepackage{enumitem}
\usepackage{latexsym,amssymb}
\usepackage{graphicx}
\usepackage{tikz}      % Pour créer des diagrammes
\usetikzlibrary{positioning,calc}
\usepackage{xcolor}
\usepackage{etoolbox}
\hypersetup{
    colorlinks=true,      % Enable colored links
    linkcolor=blue,       % Color for internal links
    citecolor=red        % Color for citation links
}

\theoremstyle{plain}
\newtheorem{theorem}{Theorem}[section]
\newtheorem{defi}[theorem]{Definition}
\newtheorem{assumption}[theorem]{Assumption}
\newtheorem{lemma}[theorem]{Lemma}
\newtheorem{corollary}[theorem]{Corollary}
\newtheorem{prop}[theorem]{Proposition}
\theoremstyle{remark}
\newtheorem{remark}[theorem]{Remark}
\newtheorem{notation}[theorem]{Notation}

\newcommand{\ds}{\displaystyle}
\newcommand{\N}{\mathbb{N}}
\newcommand{\R}{\mathbb{R}}

\newcommand{\C}{\mathbb{C}}
\newcommand{\Lp}[3]{L^{#1}(#2,#3)}
\newcommand{\schwartz}[1]{\mathscr{S}(#1)}

\newcommand{\sigmakeps}[3]{\Sigma^{#1}_{\varepsilon}(#2,#3)}

\newcommand{\norm}[2]{\| #1 \|_{#2}}
\newcommand{\normLp}[3]{\| #1 \|_{L^{#2}(#3)}}
\newcommand{\transpose}[1]{#1^{\top}}
\newcommand{\Argsh}[1]{\mathrm{Argsh}\Big(#1\Big)}
\newcommand{\argsh}[1]{\mathrm{Argsh}(#1)}
\newcommand{\dint}{\mathrm{d}}
\newcommand{\diff}{\mathrm{d}}

\newcommand{\ps}[2]{#1 \cdot #2 }
\newcommand{\sgn}{\mathrm{sgn}}
\newcommand{\abs}[1]{\lvert #1 \lvert}

\definecolor{couleurcorrec}{RGB}{43, 135, 85}
\newcommand{\review}[1]{\textcolor{black}{#1}}

\begin{document}
\title{Propagation of wave packets close to conical intersections}

\author[M. Curely]{Marianne Curely}
\address[M. Curely]{Univ Angers, CNRS, LAREMA, SFR MATHSTIC, F-49000 Angers, France}
\email{marianne.curely@univ-angers.fr}

\begin{abstract} In this paper, we study the propagation of wave packets close to conical intersections with respect to a system of two Schrödinger equations presenting a codimension $2$ crossing. We focus on the dynamics that occur when the wave packets pass through an area close to the crossing, and our main results provide an explicit formula for the outgoing wave packet in terms of the incoming one, with a complete description of its phase and of the classical trajectories it follows, including a drift.
\end{abstract}

\maketitle


\noindent \textit{packets propagation.}\\


\tableofcontents
\addtocontents{toc}{\protect\setcounter{tocdepth}{2}}

%%%%%%%%%%%%%%%%%%%%%%%%%%%%%%%%%%%%%%%%%%%%%%%%%%%%%%%%%%%
\section{Introduction}
\subsection{The equation of molecular dynamics}
Let $V : \R^d \longrightarrow \C^{2 \times 2}$ be a matrix-valued function, referred to as the \textit{potential} and $t_0$ be a real number. We consider a system of two Schrödinger equations as follows
\begin{equation}
\label{equation}
    \left\{
        \begin{array}{r c l l}
        i \varepsilon \partial_t \psi^{\varepsilon}(t,x) + \ds\frac{\varepsilon^2}{2} \Delta \psi^{\varepsilon} (t,x) & = & V(x) \psi^{\varepsilon} (t,x) & \text{for all} \, \, (t,x) \in \R \times \R^d\\
        \psi^{\varepsilon} (t_0,x) & = & \psi_{0}^{\varepsilon}(x) & \text{for all} \, \, x \in \R^d
        \end{array}
    \right.
\end{equation}

\noindent where $0 < \varepsilon \ll 1$ is called the \textit{semiclassical parameter}, $\psi^{\varepsilon} : \R \times \R^d \longrightarrow \C^2$ is the \textit{wave function} we are looking for and $\psi_{0}^{\varepsilon} : \R^d \longrightarrow \C^2$ is a bounded family of $\Lp{2}{\R^d}{\C^2}$.

\subsubsection{On the existence of solutions}

\noindent In our study, we assume that $V$ belongs to $\mathcal{C}^{\infty}(\R^d)$ and that for all $x \in \R^d$, $V(x)$ is a real-valued symmetric matrix. Then, there exist three smooth functions $v, w_1, w_2$ such that for all $x \in \R^d$
\begin{equation*}
\label{matV}
    V(x) = v(x) I_2 + \begin{pmatrix}
        w_1(x) & w_2(x) \\
        w_2(x) & -w_1(x)
    \end{pmatrix} = v(x) I_2 + A(w(x))
\end{equation*}

\noindent using the notations introduced in the Appendix \ref{appendixA} and the vector $w = (w_1 , w_2)$. For all $x \in \R^d$, the two eigenvalues of $V(x)$ are
\begin{equation*}
\label{eigenvalues}
	\lambda_{\pm}(x) = v(x) \pm \lvert w(x) \lvert, \quad \text{where} \, \, \lvert w(x) \lvert \, = \sqrt{w_1^2(x) + w_2^2(x)}.
\end{equation*}

\noindent We make the following assumptions on the growth of the potential and the gap between its two eigenvalues.

\begin{assumption}[Growth conditions]
\label{hypsurV2}
\noindent
\begin{enumerate}
	\item The function $v$ is at most quadratic. For all $\gamma \in \N^d$ such that $\lvert \gamma \lvert \, \geqslant 2$, $\partial_x^{\gamma} v \in \Lp{\infty}{\R^d}{\R}$, that is 
\begin{equation*}
\exists \, c_{\gamma} > 0, \, \underset{x \in \R^d}{\mathrm{sup}} \lvert \partial_x^{\gamma} v(x) \lvert \, \leqslant c_{\gamma}.
\end{equation*}	
    \item For $i \in \lbrace 1,2 \rbrace$, for all $\gamma \in \N^d$ such that $\lvert \gamma \lvert \, \geqslant 1$, $\partial_x^{\gamma} w_i \in \Lp{\infty}{\R^d}{\R}$, that is
\begin{equation*}
\exists \, d_{\gamma} > 0, \underset{x \in \R^d}{\mathrm{sup}} \lvert \partial_x^{\gamma} w_i(x) \lvert \, \leqslant d_{\gamma}.
\end{equation*}
    \item The difference $\lambda_{+} - \lambda_{-}$ is supposed to satisfy a polynomial gap condition at infinity: there exist constants $c_0, n_0, r_0 > 0$ such that
    \begin{equation*}
        | \lambda_{+}(x) - \lambda_{-}(x) | \, \review{\geqslant} \,  c_0 \langle x \rangle^{-n_0}, \quad \text{when } |w(x)| \geqslant r_0.
    \end{equation*}
\end{enumerate}
\end{assumption}

\noindent \review{The gap condition~$(3)$ in Assumptions~\ref{hypsurV2} ensures that the eigenvalues of \(V\) do not cross at infinity. This is useful for proving the adiabatic propagation, which is recalled later in Theorem \ref{theo:adiabprop}.} \review{Regarding} our assumptions \review{$(1)$ and $(2)$} on the potential $V$, \review{we are led to} the same setting as in~\cite[Chapter~$4$]{livreRobertComb2012}, implying the existence of solutions to Equation \eqref{equation} in $\Lp{2}{\R^d}{\C^2}$ and more generally, in the functional spaces \(\Sigma_{\varepsilon}^{k}:=\Sigma_{\varepsilon}^{k}(\R^d,\C^2)\) for $k \in \N$. \review{Here and throughout the paper, we follow the French convention and denote by \(\N\) the set of non-negative integers.} The space \(\Sigma_{\varepsilon}^{k}\) consists of all functions \(f\in\Lp{2}{\R^d}{\C^2}\) such that
\begin{equation*}
\forall \, (\alpha, \beta) \in \N^{2d}, \, \lvert \alpha \lvert + \lvert \beta \lvert \,  \leqslant k, \, x \in \R^d \mapsto x^{\alpha} \partial_x^{\beta} f(x) \in \Lp{2}{\R^d}{\C^2}
\end{equation*}
\noindent and such that the norms
\begin{equation*}
\norm{f}{\Sigma^{k}_{\varepsilon}} \, := \underset{\lvert \alpha \lvert + \lvert \beta \lvert \leqslant k}{\text{sup}} \normLp{x^{\alpha} (\varepsilon \partial_x)^{\beta} f}{2}{\R^d}
\end{equation*}

\noindent are uniformly bounded in $\varepsilon$. For all $k \in \N$, we denote by $\Sigma^k$ the sets corresponding to $\varepsilon = 1$. We note that if the initial data belongs to one of these spaces, then the solution to \eqref{equation} remains in the same space.

\subsubsection{The Born-Oppenheimer approximation}

\noindent Equation \eqref{equation} arises from molecular dynamics and is obtained under the Born-Oppenheimer approximation, first introduced in \cite{BO1927} (further details can be found in \cite[Chapter II]{bookLubich} and references therein). In summary, this approximation allows to reduce the dimension of the space variable of the studied equation. Specifically, the dimension is reduced from $3N_n + 3N_e$ to $d = 3N_n$, by decoupling the motion of the $N_n$ nuclei and the $N_e$ electrons of the particle, due to the significant mass difference between them. This mass difference is used to define the semiclassical parameter $\varepsilon$, which appears in Equation \eqref{equation}, as the square root of the ratio of the electron mass to the smaller mass of nuclei in the studied molecule.\\

\noindent A system such as \eqref{equation} results from the Born-Oppenheimer approximation in the case where the two electronic potential energy surfaces, here denoted by $\lambda_{-}(x) \leqslant \lambda_{+}(x)$, are not sufficiently separated to allow the application of an adiabatic theorem, such as those presented in \cite{TeuSpAdiab} and \cite{martinezsordoni2002}: if a sufficient separation exists, one can accurately approximate the solution of \eqref{equation} by solving two independent scalar Schrödinger equation. This paper focuses on the case where the two electronic energy levels cross, as explained in the following subsection.

\subsubsection{The problem setting} We consider the case where the electronic energy levels possess a codimension $2$ crossing. The different types of crossings for the Schrödinger equation have been classified by Hagedorn in \cite{hagedornclassification} and the one that interests us is referred to as a type ~\rm{I} crossing in Hagedorn's classification. Therefore, we study the case where the \textit{crossing set}, defined by
\begin{equation*}
\Upsilon = \lbrace (x,\xi) \in \R^{2d}, \, \lvert w(x) \lvert \, = 0 \rbrace
\end{equation*}
\noindent satisfies the following assumption.
\begin{assumption}[Assumption on $\Upsilon$]
\label{hypcrossing} For all $(x,\xi) \in \Upsilon$, 
\begin{equation*}    
    \mathrm{Rank}(\diff w(x)) = 2.
\end{equation*}  
\end{assumption}

\noindent A vector $(x,\xi) \in \Upsilon$ is called a \textit{crossing point}. The previous assumption implies that $\Upsilon$ is a smooth submanifold of $\R^{2d}$ and $w_1(x) = w_2(x) = 0$ is a system of equation of $\Upsilon$. The type of crossing described by Assumption~\ref{hypcrossing} is also known as a \textit{conical crossing}. This terminology can be justified if we look at the following example with $d = 2$. For all $x \in \mathbb{R}^2$, we define the functions $v$ and $w$ by \(v(x) = 0 \) and \(w(x) = x.\) For $x \in \mathbb{R}^2$, we have \(\lambda_{\pm}(x) = \pm |x|.\) The eigenvalue surfaces form a double cone with the apex located at \( x = (0, 0) \). In this case, the set \( \Upsilon \) is given by \[\Upsilon = \{ (0,0) \} \times \mathbb{R}^2.\]

\noindent In what follows, we denote by $\Pi_{\pm}$ the matrices of the eigenprojectors associated with $\lambda_{\pm}$, defined for all $x \in \R^{2d} \setminus \Upsilon$ by
\begin{equation}
\label{projspectraux}
\Pi_{\pm}(x) = \dfrac{1}{2} \Big(  I_2  \pm \dfrac{1}{\lvert w(x) \lvert} A(w(x)) \Big). 
\end{equation}

\subsection{Wave packets}
\label{subsec:WP}

\noindent Here, we introduce the wave packets which are functions used to construct both our initial data and the approximate solution of system \eqref{equation}. Wave packets, also called coherent states, are highly localized functions, both in terms of position and impulsion, and are outlined in the following definition.

\begin{defi}[Wave packets] Let $\varepsilon$ be a positive real number, $z = (q,p)$ be a point of $\R^{2d}$ and $\varphi$ belong to $\schwartz{\R^d}$. The wave packet associated with the point $z$ and the function $\varphi$ is the following function
\begin{equation*}
\mathrm{WP}_{z}^{\varepsilon} \varphi : x \in \R^d \mapsto \varepsilon^{-\frac{d}{4}} e^{\frac{i}{\varepsilon} p \cdot (x - q)} \varphi \Big( \dfrac{x - q}{\sqrt{\varepsilon}} \Big).
\end{equation*}
 $z$ is called the center and $\varphi$ the profile of the wave packet.
\end{defi}

\noindent For all $\varepsilon > 0$, $z \in \R^{2d}$ and $\varphi \in \schwartz{\R^d}$, $\mathrm{WP}_{z}^{\varepsilon} \varphi$ belongs to all spaces $\Sigma_{\varepsilon}^{k}$ for any $k \in \N$ and therefore to ~$\schwartz{\R^d}$. Moreover, the family $(\mathrm{WP}_{z}^{\varepsilon}\varphi)_{\varepsilon > 0}$ is uniformly bounded in all spaces $\Sigma_{\varepsilon}^{k}$. These functions have been extensively studied, in particular in \cite{livreRobertComb2012} and \cite{Hagedorn1980}.\\

\noindent We recall that the initial data and the solution of the system \eqref{equation} are vectors. Since wave packets are scalar-valued functions, they must be combined with a vector to produce vector-valued functions. This vector is referred to as the \textit{direction} of the wave packet. As we will see in the following, the wave packets we construct are focalised along eigenvectors of the matrix $V$. Let us explain how we construct the initial wave packet considered in our study.

\begin{assumption}[Assumption on the initial data] \label{hypsurID}
    For all $x \in \R^d$, the initial data $\psi_{0}^{\varepsilon}$ of the system \eqref{equation} is given by
    \begin{equation*}
        \psi_{0}^{\varepsilon}(x) = \Vec{Y}_{0} \, \mathrm{WP}_{z_0}^{\varepsilon} \varphi(x)
    \end{equation*}
where $\mathrm{WP}_{z_0}^{\varepsilon}$ is a wave packet associated with a point $z_0 = (q_0,p_0) \in \R^{2d} \setminus \Upsilon$, a function $\varphi \in \schwartz{\R^d}$ and a normalized eigenvector $\Vec{Y}_{0} \in \R^2$ of $V(q_0)$ for the minus-mode ($V(q_0) \Vec{Y}_{0} = \lambda_{-}(q_0) \Vec{Y}_{0}$).
\end{assumption}

\begin{remark} The choice of minus-mode for the initial data is arbitrary. The corresponding details for the case where we start from the plus-mode are provided in Appendix~\ref{appendix:autremode}. If we choose to start from the two modes, solutions can be superposed by linearity of Equation \eqref{equation}.
\end{remark}

\noindent In this paper, our objective is to study the behavior of wave packets with respect to Equation~\eqref{equation}. Specifically, this means the following: if the initial data is defined using a wave packet (as in Assumption \ref{hypsurID}), is it possible to find a good approximate solution that remains dependent on wave packets in the same manner? Can we link the parameters of this wave packet to those of the initial one? The main purpose of this article is to study how wave packets behave with respect to Equation ~\eqref{equation}, when their centers ``reach''  a point close to the crossing set, in a sense that will be clarified in Subsection \ref{subsec:neighCS}.

\subsection{Classical quantities}
\label{subsec:parameters} First, we introduce all the quantities required to construct the approximate solution. The evolution of the parameters of this approximation is described by classical mechanics and a non-semiclassical Schrödinger equation.

\subsubsection{Classical trajectories}

Let $t_0$ be in $\mathbb{R}$ and $z_0 = (q_0,p_0)$ be a point of $\mathbb{R}^{2d} \setminus \Upsilon$. We consider the classical trajectories at time $t \in \R$, $z_{\pm}(t) = (q_{\pm}(t), p_{\pm}(t)) \in \R^{2d}$, issued from $z_0$ at time $t_0$. The maps $z_{\pm}$ are solutions to the following ordinary differential equations
\begin{equation}
    \label{eqclassicaltraj}
    \left\lbrace 
    \begin{array}{cll}
        \dot{q}_{\pm} & = & p_{\pm} \\
        \dot{p}_{\pm} & = & -\nabla \lambda_{\pm}(q_{\pm})\\
        (q_{\pm}(t_0),p_{\pm}(t_0)) & = & z_0.
    \end{array}
    \right.
\end{equation}

\noindent We consider $\mathrm{I}$ an interval of $\R$ that contains $t_0$. Thanks to \review{the} Cauchy-Lipschitz theorem and Assumptions \ref{hypsurV2}, as soon as $\mathrm{I}$ is such that for all $t \in \mathrm{I}$, $(q_{\pm}(t), p_{\pm}(t)) \in \mathbb{R}^{2d} \setminus \Upsilon$, the system \eqref{eqclassicaltraj} has a unique solution $(q_{\pm}, p_{\pm}) \in \mathcal{C}^{\infty}(\mathrm{I})$. The solutions of \eqref{eqclassicaltraj} are associated with the flow maps $\Phi_{\pm}^{t,t_0}(z_0) = z_{\pm}(t)$, defined for $t \in \mathrm{I}$. Thus, system \eqref{eqclassicaltraj} can be rewritten as follows
\begin{equation}
\label{eqflot}
    \left\lbrace 
    \begin{array}{cll}
        \partial_t  \Phi_{\pm}^{t,t_0}(z_0) & = & J \nabla_z h_{\pm} (\Phi_{\pm}^{t,t_0}(z_0)), \quad t \in \mathrm{I}\\
        \Phi_{\pm}^{t_0,t_0}(z_0) & = & z_0,
    \end{array}
    \right.
\end{equation}

\noindent where $J = \begin{pmatrix}
            0 & I_d \\
            - I_d & 0
        \end{pmatrix} \in \R^{2d \times 2d}$, $\nabla_z = \begin{pmatrix}
            \nabla_x \\
            \nabla_\xi 
        \end{pmatrix} \in \R^{2d}$ and $h_{\pm}$ are the scalar Hamiltonians associated with the eigenvalues $\lambda_{\pm}$, defined by
\begin{equation}
\label{eq:scalarhamiltonien}
        h_{\pm}(z) = \dfrac{\lvert \xi \lvert^2}{2} + \lambda_{\pm}(x), \quad z = (x, \xi).
\end{equation}
        
\subsubsection{Actions}
\noindent We associate with the classical trajectories the actions defined for all $t \in \mathrm{I}$ by
\begin{equation}
\label{Actions}
S_{\pm}(t;t_0,z_0) = \ds\int_{t_0}^{t} \lvert p_{\pm}(s) \lvert^2 - h_{\pm}(z_{\pm}(s)) \,  \diff s.
\end{equation}

\subsubsection{Real-valued time-dependent eigenvectors along the trajectory}

\noindent In order to determine the time-dependent eigenvectors $\Vec{Y}_{\pm}(t)$ guiding the approximate solution, we introduce the matrix-valued function $B_{\pm} \in \mathcal{C}^{\infty}(\R^{2d} \setminus \Upsilon)$, defined for all $(x,\xi) \in \R^{2d} \setminus \Upsilon$ by
\begin{equation}
\label{matB}
    B_{\pm}(x,\xi) = \pm \Pi_{\mp}(x) \xi \cdot \nabla_x \Pi_{+}(x) \Pi_{\pm}(x).
\end{equation}

\noindent We consider the normalized vectors $\Vec{Y}_{\pm}$ which are solutions to the differential equation
\begin{equation}
\label{eqVec}
    \partial_t  \Vec{Y}_{\pm}(t) =  B_{\pm}(q_{\pm}(t), p_{\pm}(t)) \Vec{Y}_{\pm}(t), \quad t \in \mathrm{I}
\end{equation}

\noindent where $\Vec{Y}_{\pm}(t_0) =  \Vec{Y}_{0}$ with $\Vec{Y}_{0}$ is defined \review{as} in Assumption \ref{hypsurID}. Under the same assumptions as above on the interval $\mathrm{I}$, this system admits a unique solution. More precise results related to these vectors and the analysis of their behavior when the classical trajectories approach the crossing set are detailed in Subsection \ref{subsec:studyVec}. In particular, we will prove that \(\vec{Y}_{\pm}\) is an eigenvector of the matrix \(V(q_{\pm})\). 

\subsubsection{Profile equations}

\noindent The profiles of the approximate solutions are constructed from the classical trajectories $(q_{\pm}, p_{\pm})$ on the time interval where they are defined and the solutions of the Schrödinger equations with a time-dependent harmonic potential
\begin{equation}
\label{eqProfil}
    i \partial_t  u_{\pm} = - \dfrac{1}{2} \Delta u_{\pm} + \dfrac{1}{2} \text{Hess} \, \lambda_{\pm} (q_{\pm}) y \cdot y u_{\pm} 
\end{equation}

\noindent with an initial data in $\schwartz{\R^d}$. According to \cite{MasperoRobert17}, these equations admit solutions in $\Sigma^k$ for any $k \in \review{\N}$ as long as the classical trajectories are far enough from the crossing. The behavior of these solutions when the trajectories $(q_{\pm}, p_{\pm})$ approach the crossing set $\Upsilon$ is analyzed in Subsection \ref{subsec:consProfil}.

\subsection{Adiabatic propagation and aim of the study}

\noindent The question of the propagation of coherent states has been extensively analyzed and the construction of a leading-order approximate solution using wave packets has been thoroughly studied. The case of the scalar equation is covered in \cite{Hagedorn1980} and \cite{livreRobertComb2012}, while \cite{carles2021semi} provides details for the nonlinear case. Concerning the vector case, the propagation far enough from the crossing (in the sense that $\lvert w(q) \lvert \, > \delta$, with $\delta$ a positive real number) is well known (see \cite{livreRobertComb2012} for example). Outside the gap region, a wave packet remains approximately in the same eigenspace as the initial one and propagates along the classical trajectories associated with the same mode. In ~\cite[Theorem 1.10]{FGH2021}, we have an error estimate between the solution $\psi^{\varepsilon}$ of Equation ~\eqref{equation} and the coherent state that approaches it to leading order, which depends on the parameter ~$\delta$. Let us recall this adiabatic theorem, that is the starting point of this article.

\begin{theorem}[Adiabatic propagation] \label{theo:adiabprop} Let Assumption\review{s} \ref{hypsurV2} and \ref{hypcrossing} hold. We consider $\varepsilon \in (0,1)$, $\delta > 0$ and $(t_0,T) \in \R^2$. We assume that $\psi^{\varepsilon}_0$ satisfies Assumption \ref{hypsurID} at $t_0$ and that the trajectory $t \mapsto \Phi_{-}^{t,t_0}(z_0) \in \lbrace \lvert w \lvert \, > \delta \rbrace$ for all $t \in [t_0, t_0 + T]$. For all $k \in \N$, there exists a constant $C_k > 0$ (uniform in $\delta$ and $\varepsilon$) such that
\begin{equation}
\label{eq:erreurFGH}
    \big\lvert \big\lvert \psi^{\varepsilon}(t) - \Vec{Y}_{-}(t) w_{-}^{\varepsilon}(t) \big\lvert \big\lvert_{\sigmakeps{k}{\R^d}{\C^2}} \leqslant C_k (1 + \lvert \ln \delta \lvert) \Bigg( \dfrac{\varepsilon^{\frac{3}{2}}}{\delta^4} + \dfrac{\sqrt{\varepsilon}}{\delta} \Bigg) \quad \text{for all} \, \, t \in [t_0, t_0 + T],
\end{equation}
\noindent where the vector-valued function $t \mapsto \Vec{Y}_{-}(t)$ is defined in \eqref{eqVec} with $\vec{Y}_0$ as initial condition at time $t_0$, the function $t \mapsto w_{-}^{\varepsilon}(t)$ is a wave packet $$\quad w_{-}^{\varepsilon}(t,x) = e^{\frac{i}{\varepsilon} S_{-}(t;t_0,z_0)} \mathrm{WP}_{z_{-}(t)}^{\varepsilon} u_{-}(t,x), \quad \text{for all } x \in \R^d$$ with the classical trajectory $t \mapsto z_{-}(t)$ defined in \eqref{eqclassicaltraj}, the action $t \mapsto S_{-}(t;t_0,z_0)$ introduced in \eqref{Actions} and the profile $t \mapsto u_{-}(t)$ \review{is the} solution of $\eqref{eqProfil}$ with $\varphi$ as initial data at time $t_0$.
\end{theorem}

\noindent A similar $\delta$-dependence of the approximation appears in \cite{FLRcodimun}, where the authors study a codimension ~$1$ crossing and provide an approximation, not only at leading order but at any order $N \in \mathbb{N}$. The remainder in their adiabatic theorem (\cite[Theorem 5.1]{FLRcodimun}) is of order \[ \left( \frac{\sqrt{\varepsilon}}{\delta} \right)^{N+1} \delta^{-2 \kappa_0}, \quad \text{with } \kappa_0 \in \mathbb{N}.\] We expect that a similar result holds in our case and we will \review{devote} future works to this question. Thus, in this lines of ideas, the result of Theorem \ref{theo:adiabprop} shows that the remainder is small provided that the classical trajectory \( t \mapsto z_{-}(t) \) satisfies
\begin{equation}
\label{eq:validitetheoadiab}
    \inf_{t \in [t_0, t_0 + T]} \, |w(q_{-}(t))| \gg \sqrt{\varepsilon}.
\end{equation}

\noindent Otherwise, we can apply Theorem~\ref{theo:adiabprop} only far from the crossing region and we must analyze the passage through the gap, that is, the region where $|w(q)| \leqslant \delta$. The case where
\begin{equation}
\label{eq:conditionFGH}
    \inf_{t \in [t_0, t_0 + T]} \, |w(q_{-}(t))| = 0
\end{equation}
is treated in~\cite{FGH2021}. The authors prove that, when condition~\eqref{eq:validitetheoadiab} is not satisfied (because \eqref{eq:conditionFGH} holds), a new non-negligible wave packet is generated on the other mode at leading order. More precisely, if we start with a wave packet propagating in the minus-mode, after the crossing a good approximate solution is the sum of two wave packets: one propagating in the minus mode and the other in the plus mode.\\

\noindent The aim of this article is to complement the results of \cite{FGH2021}. In fact, the main motivation is to approximate the propagator associated with Equation \eqref{equation} using a decomposition into a continuous sum of Gaussian wave packets, similar to the approach in \cite{FLRcodimun}. For general initial data, such a decomposition will, for every \( \varepsilon > 0 \), contain a non-negligible number of wave packets associated with classical trajectories that neither satisfy condition \eqref{eq:validitetheoadiab} or condition \eqref{eq:conditionFGH}. Therefore, it is crucial to analyze the evolution of a wave packet centered at such a point. Then, in this paper, we provide a complete picture of the dynamics for a system of two Schrödinger equations exhibiting a codimension $2$ crossing, as described in Assumption~\ref{hypcrossing}.

\subsection{The framework}
\label{subsec:neighCS}

\noindent To construct the set of initial points, we first focus on minimizing the functions \( J_{\pm} : t \in \mathrm{I} \mapsto \dfrac{|w(q_{\pm}(t))|^2}{2} .\) For all \( t \in \mathrm{I} \), we have
    \begin{equation*}
        J_{\pm}'(t) = w(q_{\pm}(t)) \cdot \diff w(q_{\pm}(t)) p_{\pm}(t).
    \end{equation*}
\noindent Therefore, we introduce the following set
\begin{center}
$\Sigma = \lbrace (x,\xi) \in \R^{2d}, \ps{w(x)}{\diff w(x)\xi} = 0 \rbrace$.
\end{center}

\noindent Note that $\Upsilon \subset \Sigma$. We say that a point \((x, \xi) \in \Sigma\) is \textit{non-degenerate} if \[\mathrm{d}w(x)\, \xi \neq 0,\]
and we denote by \(\Sigma_{\mathrm{nd}}\) the set of non-degenerate points in \(\Sigma\). Considering non-degenerate points implies that we are dealing with generic codimension $2$ crossings according to Colin de Verdière's classification (see \cite{CdVclassification} and also \cite{FGcodim3}, where a similar classification is provided).\\

\noindent Let us introduce the set of admissible initial data for the classical trajectories \( z_{\pm} \). We consider the parameters \((M, \alpha_0,r_0,\delta) \in \R_{+}^{*} \times \R_{+} \times \R_{+}^{*} \times (0, 1] \) and a compact interval \( \mathrm{I} \subset \mathbb{R} \) with \( t_0 = \min(\mathrm{I}) \).

\begin{defi}
\label{defi:initialdata}
    A point $z_0 = (q_0,p_0) \in \R^{2d}$ belongs to the set of admissible initial data \( \mathcal{A}_{\pm}(M,\alpha_0,r_0, \delta, \mathrm{I}) \) if it satisfies the following conditions
    \begin{enumerate}
        \item $|p_0| + |q_0| \leqslant M$,
        \item there exists $t^{\flat} \in \mathrm{I}$ such that $[t^{\flat} - \delta, t^{\flat} + \delta] \subset \mathrm{I}$ and \[ \Phi_{\pm}^{t^{\flat},t_0}(z_0) := z^{\flat} = (q^{\flat}, p^{\flat}) \in \Sigma_{\mathrm{nd}} \quad \text{and} \quad \Phi_{\pm}^{t,t_0}(z_0) \notin \Sigma \text{ for } t \in \mathrm{I} \setminus \lbrace t^{\flat} \rbrace, \]
        \item the following controls hold
        \[ |w(q^{\flat})| \leqslant \alpha_0 \quad \text{and} \quad |\mathrm{d}w(q^{\flat}) p^{\flat}| > r_0. \]
    \end{enumerate}
\end{defi} 

\noindent The main result of this article (Theorem~\ref{theo:main} below) provides a good approximation, as~$\varepsilon \to 0$, of the solution \( \psi^{\varepsilon} \) to Equation~\eqref{equation} on the time interval \( \mathrm{I} \), with the initial data at time \( t_0 \) satisfying Assumption~\ref{hypsurID} with its center \( z_0 \) belonging to \( \mathcal{A}_{-}(M,C \varepsilon^{\frac{1}{2}-\beta}, r_0, \varepsilon^{\frac{5}{14}}, \mathrm{I}) \) where \( C > 0 \) and $\review{\beta \in (0, \frac{1}{42})}$. 

\begin{remark} For all $t \in \mathrm{I}$, we compute
    \begin{equation*}
        J_{\pm}''(t) = \lvert \diff w(q_{\pm}(t)) p_{\pm}(t) \lvert^2 + \begin{pmatrix}
            - \nabla \lambda_{\pm}(q_{\pm}(t)) \cdot \nabla w_1 (q_{\pm}(t)) + p_{\pm}(t) \cdot \text{Hess} \, w_1 (q_{\pm}(t)) p_{\pm}(t)\\
            - \nabla \lambda_{\pm}(q_{\pm}(t)) \cdot \nabla w_2 (q_{\pm}(t)) + p_{\pm}(t) \cdot \text{Hess} \, w_2 (q_{\pm}(t)) p_{\pm}(t)
        \end{pmatrix} \cdot w(q_{\pm}(t)).
    \end{equation*}

\noindent \review{In view of Definition~\ref{defi:initialdata}, if we consider 
\( z_0 \in \mathcal{A}_{-}(M,C \varepsilon^{\frac{1}{2}-\beta}, r_0, \varepsilon^{\frac{5}{14}}, \mathrm{I}) \) with \( C > 0 \) and $\review{\beta \in (0, \frac{1}{42})}$, we have that 
\( t^{\flat} \in \mathrm{I} \) satisfies \( J'_{\pm}(t^{\flat}) = 0 \) and 
\( w(q_{\pm}(t^{\flat})) \) is small 
(\( \lvert w(q_{\pm}(t^{\flat})) \rvert \leqslant \alpha_0 := C \varepsilon^{\frac{1}{2}-\beta} \)). 
Then, the dominant term in $\review{J_{\pm}''(t^{\flat})}$ is 
\( \lvert \diff w(q_{\pm}(t^{\flat})) p_{\pm}(t^{\flat}) \rvert^2 > 0 \) 
(\( z_{\pm}(t^{\flat}) \in \Sigma_{\mathrm{nd}} \)). Therefore, 
\( J_{\pm} \) indeed attains a minimum at \( t^{\flat} \).}
\end{remark}

\subsection{Classical quantities after the crossing and main result} To state the main theorem of this article (Theorem~\ref{theo:main}), we must introduce new classical quantities, defined near and after the crossing. In fact, when we pass through a region close to the crossing, it is known that transitions between modes can occur. This effect was first described by L.~D.~Landau and C.~Zener in the early $1930$s. Their analysis shows that, after the crossing, a new wave packet can be generated on the mode opposite to the initial one. Then, the parameters of these outgoing coherent states have to be constructed and it is important to ask which initial condition to choose at time $t^{\flat}$, especially for the wave packet that has changed mode. This condition is chosen in particular to ensure that the energy of the particles that undergo the mode transition is conserved, up to an error of order \( \varepsilon \). When the initial wave packet is supported on the minus-mode, the conservation of energy implies
\begin{equation}
\label{eq:conservationNRJ}
    h_{+}(\widetilde{z}^{\flat}) = h_{-}(z^{\flat}) + \mathcal{O}(\varepsilon),
\end{equation}
\noindent
where the scalar Hamiltonians \( h_{\pm} \) are defined in~\eqref{eq:scalarhamiltonien}.\\

\noindent In theoretical chemistry, various strategies have been developed for the numerical resolution of the Schrödinger equation, leading to surface hopping algorithms (see the celebrated paper \cite{tully1971}). The hops reflect the generation of a new wave packet as a trajectory passes close to the crossing set (see~\cite{fermlass2008JCP} and ~\cite{fermlass2017CIMP}). These algorithms, based on the use of classical trajectories, require a choice of parameters that describe the approximate solution. In this paper, we explicitly compute the parameters of the outgoing wave packet. These computations provide both a \textit{drift term} for the classical trajectories and \textit{transfers coefficients} for the profiles. In this setting, the main result of this article has the objective of extending the methods mentioned just above to systems presenting a codimension $2$ crossing, in the spirit of \cite{FLRcodimun} where the codimension $1$ case is treated. The scalar case can be found in \cite{swart2009herman} and \cite{robert2010herman}, while a numerical analysis is presented in \cite{lassersatdiscretising}. Other strategies have been developed for the numerical analysis of such equations, in particular in \cite{lu2018hopping}.\\

\noindent In the following, we consider \( z_0 \in \mathcal{A}_{-}(M,\alpha_0,r_0,\delta,\mathrm{I})\). According to Definition \ref{defi:initialdata}, there exist a time \( t^{\flat} \in \mathrm{I} \) and a point $z^{\flat} \in \Sigma_{\text{nd}}$ such that $[t^{\flat} - \delta,\, t^{\flat} + \delta] \subset \mathrm{I}$ and
\[ z^{\flat} := \Phi_{-}^{t^{\flat}, t_0}(z_0) \in \Sigma_{\mathrm{nd}}, \quad \Phi_{-}^{t,t_0}(z_0) \notin \Sigma \text{ for } t \in \mathrm{I} \setminus \lbrace t^{\flat}\rbrace, \quad |w(q^{\flat})| \leqslant \alpha_0 \quad \text{and} \quad \left| \mathrm{d}w(q^{\flat})\, p^{\flat} \right| > r_0.\]

\noindent We introduce notations that will be used throughout the rest of the document.

\begin{notation} \label{notation:alpha} Let $r > 0$ and $\theta \in \R$ be such that $\diff w (q^{\flat})p^{\flat} = r \Vec{e}_{\theta}$ where $\Vec{e}_{\theta} = (\cos(\theta), \sin(\theta)) \in \R^2$. As $z^{\flat} \in \Sigma_{\text{nd}}$ and $(\Vec{e}_{\theta} , \Vec{e}_{\theta}^{\perp})$ is a basis of $\R^2$, we write $w(q^{\flat}) = \alpha \Vec{e}_{\theta}^{\perp}$ with $$\alpha = \lvert w(q^{\flat}) \lvert.$$
\end{notation}

\noindent Note that the quantities \(t^{\flat} \), \(z^{\flat} \), \( r \), \( \theta \), and \( \alpha \) depend on the choice of the point \( z_0 \) but we omit this dependence in the notation for simplicity. Subsection~\ref{subsec:drift} sets the condition imposed at time \( t^{\flat} \) for the classical trajectories, Subsection~\ref{subsec:newCQ} introduces the classical quantities used to describe the wave packets after the crossing and Subsection~\ref{subsec:maintheo} presents the main theorem.

\subsubsection{Drift}
\label{subsec:drift}

\noindent  Since ~\cite{hagedorn1998}, \cite{hagedorn1999}, ~\cite{fermanianlasserdrift2012} and \cite{Haridrift2016} it is classical to implement a drift for trajectories passing close to a crossing point (some talk of avoiding crossings) by setting $~\widetilde{z}^{\flat} = (q^{\flat}, p^{\flat} + \delta_{p^{\flat}})$ with $\delta_{p^{\flat}}$ the required drift. The energy conservation, stated in \eqref{eq:conservationNRJ}, imposes the following conditions on $\delta_{p^{\flat}}$
\begin{equation}
\label{eq:defdrift}
    p^{\flat} \cdot \delta_{p^{\flat}} = - 2\alpha \quad \text{and} \quad \delta_{p^{\flat}} = \mathcal{O}(\sqrt{\varepsilon}).
\end{equation}

\noindent These conditions do not provide a unique definition for $\delta_{p^{\flat}}$. In the literature (see \cite{tully1971}, ~\cite{hagedorn1998}, \cite{hagedorn1999}, \cite{fermanianlasserdrift2012} or \cite{Haridrift2016} for example), a commonly chosen particular solution for $\delta_{p^{\flat}}$ is given by \[ \delta_{p^{\flat}} = - 2 \alpha \dfrac{p^{\flat}}{\lvert p^{\flat} \lvert}.\] If \( z^{\flat} \in \Upsilon \), then \( \alpha = \lvert w (q^{\flat}) \lvert \, = 0 \) and therefore there is no drift, as it is the case in ~\cite{hagedornclassification} and ~\cite{FGH2021}. Our main result (Theorem ~\ref{theo:main}) does not hold for this choice of drift. As it turns out, our analysis leads us to the following, different solution of ~\eqref{eq:defdrift}.

\begin{defi}[] \label{def:drift} In this paper, the drift $\delta_{p^{\flat}}$ is given by 
\begin{equation}
\label{eq:ourdrift}
\delta_{p^{\flat}} = - \dfrac{2\alpha}{r} \transpose{\dint w(q^{\flat})} \Vec{e}_{\theta}.
\end{equation}
\end{defi}

\noindent We note that this drift satisfies Equation ~\eqref{eq:defdrift} of energy conservation for $\alpha$ small enough. This definition is dictated by the specific steps of the proof presented in Subsection~\ref{subsec:proof} (see also Figure ~\ref{fig:phase_diagram_both}).

\subsubsection{Classical quantities after the crossing}
\label{subsec:newCQ}

\noindent Let us now define the parameters that will be used to describe the wave packets generated after the crossing. A new coherent state appears in the mode distinct from the initial one. This requires the introduction of additional  ``drifted" classical trajectories, denoted with a tilde. Under Assumption \ref{hypsurID}, these new parameters are relevant only for the plus-mode (the corresponding details when we start for the minus-mode are provided in Appendix~\ref{appendix:autremode}). Moreover, there is also a wave packet that remains in the initial mode. This coherent state is constructed using the classical trajectories \( z_{-} \) introduced in \eqref{eqclassicaltraj} with \( z^{\flat} \) as initial condition at time \(t^{\flat}\). Here, we introduce the parameters associated with the plus-mode wave packets. More precisely, we explain how the initial data are chosen. Concerning the minus-mode, we only explain how we construct the direction \( \vec{Y}_{-} \) and the profile. Further details are provided in Section~\ref{sec:analysisCQ}, where we study the classical trajectories, the actions, and the directions of the wave packets, and in Section~\ref{sec:profiles}, where we analyze the profiles.\\

\noindent We consider the following  ``drifted" classical trajectories
\begin{equation}
    \label{eqclassicaltrajdrift}
    \left\lbrace 
    \begin{array}{cll}
        \dot{\widetilde{q}}_{+} & = & \widetilde{p}_{+} \\
        \dot{\widetilde{p}}_{+} & = & -\nabla \lambda_{+}(\widetilde{q}_{+})\\
        (\widetilde{q}_{+}(t^{\flat}),\widetilde{p}_{+}(t^{\flat})) & = & \widetilde{z}_{\flat}
    \end{array}
    \right.
\end{equation}
\noindent and the action associated with it
\begin{equation}
\label{Actionstilde}
    \widetilde{S}_{+}(t;t^{\flat},\widetilde{z}_{\flat}) = \ds\int_{t_0}^t (\lvert \widetilde{p}_{+}(s) \lvert^2 \, - \, h_{+}(\widetilde{z}_{+}(s))) \, \diff s \quad \text{and we set} \quad \widetilde{S}_{+}^{\flat} = \widetilde{S}_{+}(t^{\flat};t^{\flat},\widetilde{z}_{\flat}).
\end{equation}
\noindent \review{Regarding the profiles, we first consider the function \(v_{-}\), which solves Equation~\eqref{eqProfil} for \(t>t^{\flat}\) and satisfies the initial condition \[v_{-}(t^{\flat})=v_{-}^{\flat}.\] We then introduce the function \(\widetilde{v}_{+}\), which solves the Schrödinger equation
\begin{equation}
\label{eqProfiltilde}
    i \partial_t \widetilde{v}_{+}
    =
    -\frac12 \Delta \widetilde{v}_{+}
    + \frac12 \operatorname{Hess}\lambda_{+}(\widetilde{q}_{+})\, y\cdot y\, \widetilde{v}_{+},
\end{equation}
for \(t>t^{\flat}\) and satisfies the initial condition
\[
\widetilde{v}_{+}(t^{\flat})=\widetilde{v}_{+}^{\flat}.
\]
Here, the functions \(v_{-}^{\flat}\) and \(\widetilde{v}_{+}^{\flat}\) belong to \(\schwartz{\mathbb{R}^d}\) and are defined for all \(y\in\mathbb{R}^d\) by}
\begin{multline}
\label{eq:profilout}
    v_{-}^{\flat}(y) := \lim_{t \to t^{\flat}, \, t> t^{\flat}} \exp\left( \frac{i}{2} G_{\alpha}(t)\, y \cdot y \right) u_{-}^{\mathrm{out},\alpha}(t) \\ \text{ and } \widetilde{v}_{+}^{\flat}(y) := \lim_{t \to t^{\flat}, \, t > t^{\flat}} \exp\left( -\frac{i}{2} G_{\alpha}(t)\, y \cdot y \right) \widetilde{u}_{+}^{\mathrm{out},\alpha}(t),
\end{multline}
where $G_{\alpha}$ is a symmetric matrix-valued function defined for all $t \in \mathbb{R}$ by
\begin{multline}
\label{def:phase}
    G_{\alpha}(t) := \review{\Bigg( 
        \dfrac{\partial_i w(q^{\flat}) \cdot \partial_j w(q^{\flat})}{r} h_{\alpha}(t)}
        \\ \review{- \dfrac{\left( \partial_i w(q^{\flat}) \cdot \vec{e}_{\theta} \right) \left( \partial_j w(q^{\flat}) \cdot \vec{e}_{\theta} \right)}{r}
        \left( h_{\alpha}(t) - \dfrac{r \lvert t - t^{\flat}\lvert}{\sqrt{\alpha^2 + r^2 (t - t^{\flat})^2}} \right) \Bigg)_{1 \leq i,j \leq d},}
\end{multline}
with the scalar function $h_{\alpha}$ given by
\begin{equation}
    \label{eq:h_alpha}
    h_{\alpha}(t) := \ln\left( r \abs{t - t^{\flat}} + \sqrt{ \alpha^2 + r^2 (t - t^{\flat})^2 } \right).
\end{equation}

\noindent The outgoing functions $\left( \widetilde{u}_{+}^{\mathrm{out},\alpha},\ u_{-}^{\mathrm{out},\alpha} \right)$ that appear in \eqref{eq:profilout} are expressed for all $y \in \R^d$ by
\begin{equation}
    \label{coro:transitioninandout}
    \begin{pmatrix}
        \widetilde{u}_{+}^{\mathrm{out},\alpha}(y) \\
        u_{-}^{\mathrm{out},\alpha}(y)
    \end{pmatrix}
    =
    \begin{pmatrix}
        e^{i\widetilde{\Lambda}^{\varepsilon,\alpha}(\mathrm{H}(y))} b\left( \frac{\mathrm{H}_2(y)}{\sqrt{r}} + \frac{\alpha}{\sqrt{r\varepsilon}}\right)
        & 
        a\left( \frac{\mathrm{H}_2(y)}{\sqrt{r}} + \frac{\alpha}{\sqrt{r\varepsilon}} \right) \\
        a\left( \frac{\mathrm{H}_2(y)}{\sqrt{r}} + \frac{\alpha}{\sqrt{r\varepsilon}} \right)
        &
        -e^{-i\widetilde{\Lambda}^{\varepsilon,\alpha}(\mathrm{H}(y))} \overline{b}\left( \frac{\mathrm{H}_2(y)}{\sqrt{r}} + \frac{\alpha}{\sqrt{r\varepsilon}} \right)
    \end{pmatrix}
    \begin{pmatrix}
        0 \\
        e^{\frac{i}{\varepsilon} S_{-}^{\flat}} u_{-}^{\mathrm{in},\alpha}(y)
    \end{pmatrix},
\end{equation}
with the incoming profile $u_{-}^{\mathrm{in},\alpha}$ defined for $y \in \R^d$ by
\begin{equation}
    \label{eq:uinalpha}
    u_{-}^{\mathrm{in},\alpha}(y) := \lim_{t \to t^{\flat}, \, t < t^{\flat}} \exp\left( \frac{i}{2} G_{\alpha}(t)\, y \cdot y \right) u_{-}(t),
\end{equation}
where $u_{-}$ is the solution to~\eqref{eqProfil} for $t < t^{\flat}$ with initial condition $\varphi$ (introduced in Assumption~\ref{hypsurID}) at time $t_0$. In the above expression, we also have
\[
    S_{-}^{\flat} := S_{-}(t^{\flat}; t_0, z_0), 
    \qquad 
    \mathrm{H}(y) := 
    \Bigg( 
        \underbrace{ \diff w(q^{\flat})\, y \cdot \vec{e}_{\theta} }_{\mathrm{H}_1(y)}, 
        \underbrace{ \diff w(q^{\flat})\, y \cdot \vec{e}_{\theta}^{\perp}}_{\mathrm{H}_2(y)} 
    \Bigg),
\]
\noindent and 
\begin{equation}
\label{fonctionaetb}
    a(z) := e^{-\frac{\pi z^2}{2}}, \quad
    b(z) := \frac{2i}{z \sqrt{\pi}}\, 2^{-\frac{i z^2}{2}} e^{-\frac{\pi z^2}{4}} 
    \Gamma\left( 1 + \frac{i z^2}{2} \right) \sinh\left( \frac{\pi z^2}{2} \right) \quad \text{for all } z \in \C
\end{equation}
where $\Gamma$ is the Gamma function. Finally, the phase $\widetilde{\Lambda}^{\varepsilon,\alpha}$ is defined by
\begin{align*}
    \widetilde{\Lambda}^{\varepsilon,\alpha}(\mathrm{H}(y)) := \Gamma_1\, y \cdot y - \frac{1}{r} \mathrm{H}_1^2(y) + \frac{\alpha^2}{2r \varepsilon} 
    - \frac{1}{r} \mathrm{H}_2^2(y) \ln \left( \frac{1}{2\sqrt{r \varepsilon}} \right) - \frac{\alpha^2}{r \varepsilon} \ln \left( \frac{\alpha}{2\sqrt{r \varepsilon}} \right)
    - \frac{2 \alpha}{r \sqrt{\varepsilon}}\, \mathrm{H}_2(y) \ln \left( \frac{\alpha}{2\sqrt{r \varepsilon}} \right)
\end{align*}
with the matrix $\Gamma_1$ explicitly given by
\begin{equation}
\label{matGamma1}
    \Gamma_1 = \review{\dfrac{1}{r} \Bigg( (\partial_i w(q^{\flat}) \cdot \Vec{e}_{\theta})(\partial_j w(q^{\flat}) \cdot \Vec{e}_{\theta}) \Bigg)_{1 \leqslant i,j \leqslant d}.}
\end{equation}
\noindent All the limits mentioned to construct the profiles exist by Propositions ~\ref{prop:ingoingprofiles} and ~\ref{prop:outgoingprofiles}. Moreover, we can note that Equation \eqref{coro:transitioninandout} fully determines the relation between the incoming and outgoing profiles. The phases involved are explicit, contrary to \cite{kammererlasser2008}, where the authors studied such transfers in terms of Wigner transforms.\\

\noindent With regard to the directions, for \( t > t^{\flat} \), we consider the vectors \( \Vec{Y}_{-} \) and \( \Vec{\widetilde{Y}}_{+} \), solutions to Equation ~\eqref{eqVec} and to the following one
\begin{equation}
    \label{eqVectilde}
    \partial_t \Vec{\widetilde{Y}}_{+} = B_{+}(\widetilde{q}_{+}, \widetilde{p}_{+}) \, \Vec{\widetilde{Y}}_{+}.
\end{equation}
The initial data at time \( t^{\flat} \) are given in Equation \eqref{eq:ICvecafter}. These conditions are chosen to be the eigenvectors of the matrix \( V(q^{\flat}) \) associated with the eigenvalues \( \lambda_{-}(q^{\flat}) \) and \( \lambda_{+}(q^{\flat})\) respectively. This is detailed in Subsection~\ref{subsec:studyVec} and more precisely in Proposition~\ref{prop:Vecaftercrossing}.

\subsubsection{Main theorem}
\label{subsec:maintheo}
We can now state the main theorem.
\begin{theorem}
\label{theo:main}
Let Assumptions \ref{hypsurV2} and \ref{hypcrossing} hold.

\noindent We consider \( M, \, r_0 \) two positive real numbers and $k \in \N$. There exists a constant $C > 0$ such that the following holds. For all $\varepsilon \in (0,1)$ and \(\beta \in \big(0,\frac{1}{42}\big)\), for all compact interval\review{s} $\mathrm{I} \subset \R$ with $t_0 = \text{min}(\mathrm{I})$ and $z_0 \in \review{\mathcal{A}_{-}(M,\alpha_0,r_0,\delta,\mathrm{I})}$ \review{with $\alpha_0 \leqslant C \varepsilon^{\frac{1}{2}-\beta}$ and $\delta \leqslant \varepsilon^{\frac{5}{14}}$}, the solution $\psi^{\varepsilon}$ of \review{Equation} ~\eqref{equation}, \review{corresponding to the initial data \(\psi_0^\varepsilon\) given in Assumption~\ref{hypsurID}}, satisfies
\begin{equation*}
     \psi^{\varepsilon}(t) = \Vec{Y}_{-}(t) w_{-}^{\varepsilon}(t) + \Vec{\widetilde{Y}_{+}}(t) \widetilde{w}_{+}^{\varepsilon}(t) + \mathcal{O} \Big( \varepsilon^{\frac{1}{14} - \beta} (1 + \lvert \ln \varepsilon\lvert)\Big) \quad \text{for all } t > t^{\flat} := t^{\flat}(z_0), 
\end{equation*}
\noindent in $\sigmakeps{k}{\R^d}{\C^2}$, as $\varepsilon \to 0$. \review{The quantities involved in the statement are defined as follows}
\begin{itemize}[leftmargin=*, labelindent=0pt]
    \item the vector-valued functions \( t \mapsto \Vec{Y}_{-}(t) \) and \( t \mapsto \Vec{\widetilde{Y}}_{+}(t) \) are the solutions \review{of} Equations~\eqref{eqVec} and~\eqref{eqVectilde} with initial conditions given in \eqref{eq:ICvecafter}, 
    \item the functions $t \mapsto w_{-}^{\varepsilon}(t)$ and $t \mapsto \widetilde{w}_{+}^{\varepsilon}(t)$ are wave packets \review{defined for all $x \in \R^d$ by} $$w_{-}^{\varepsilon}(t,x) = e^{\frac{i}{\varepsilon} S_{-}(t;t^{\flat},z^{\flat})} \mathrm{WP}_{z_{-}(t)}^{\varepsilon} v_{-}(t,x) \quad \text{and} \quad \widetilde{w}_{+}^{\varepsilon}(t,x) = e^{\frac{i}{\varepsilon} \widetilde{S}_{+}(t;t^{\flat},\widetilde{z}^{\flat})} \mathrm{WP}_{\review{\widetilde{z}_{+}(t)}}^{\varepsilon} \widetilde{v}_{+}(t,x)$$ with 
    \begin{itemize}
        \item[$\star$] $t \mapsto v_{-}(t)$ and $t \mapsto \widetilde{v}_{+}(t)$ are solutions of \eqref{eqProfil} and \eqref{eqProfiltilde} for $t > t^{\flat}$, with initial conditions at $t^\flat$ given above,
        \item[$\star$] the trajectories $t \mapsto z_{-}(t)$ \review{and} $t \mapsto \widetilde{z}_{+}(t)$ are the classical trajectories defined in \eqref{eqclassicaltraj} and \eqref{eqclassicaltrajdrift}, with respectively \( z^{\flat} \) and \( \widetilde{z}^\flat \), as initial conditions at \( t^{\flat} \),
        \item[$\star$] $t \mapsto S_{-}(t;t^{\flat},z^{\flat}) := S_{-}(t;t_0,z_0) - S_{-}(t^{\flat};t_0,z_0)$ and $t \mapsto \widetilde{S}_{+}(t;t^{\flat},\widetilde{z}^{\flat})$ are the related actions introduced in \eqref{Actions} and \eqref{Actionstilde}.
    \end{itemize}
\end{itemize}
\end{theorem}

\noindent Regarding the accuracy of the remainder in Theorem ~\ref{theo:main}, it is reasonable to expect that the solution admits a wave packet representation at any order, as shown in \cite{FLRcodimun} for codimension 1 crossing. In this paper, we are interested in obtaining \review{precise} information about the transition rules and we have restricted ourselves to the previous leading-order approximation. In order to get a full wave packet expansion, further (non negligible) work is needed, in particular to construct the profiles of the higher-order wave packets.

\subsection{Ideas of the proof and organization of the article} \label{subsec:proof} The proof is inspired by the work presented in \cite{FGH2021}. The initial step of the proof is to examine the behavior of classical quantities when $t$ is close to $t^{\flat}$, that is to say, when the classical trajectories reach $\Sigma_{\text{nd}}$. This is the goal of Section ~\ref{sec:analysisCQ} for the classical trajectories, the actions and the directions of the wave packets and Section ~\ref{sec:profiles} for the profiles.

\subsubsection{Change of variables and unknown}

\noindent A subsequent step involves the analysis of the transition through the area close to the crossing. To this end, we first introduce the following change of time
\begin{equation}
\label{newtime}
    t = t^{\flat} + s \sqrt{\varepsilon}
\end{equation}
\noindent and the new unknown function $u^{\varepsilon}\in \Lp{2}{\R^d}{\C^2}$ defined by
\begin{equation}
\label{newinconnue}
    \psi^{\varepsilon}(t) = e^{\frac{i}{\varepsilon} S_0(t;t^{\flat},z^{\flat})} \mathrm{WP}_{\Phi_0^{t,t^{\flat}}(z^{\flat})}^{\varepsilon} u^{\varepsilon}\Big( \dfrac{t - t^{\flat}}{\sqrt{\varepsilon}} \Big)
\end{equation}
\noindent where $\psi^{\varepsilon}$ is the solution of \eqref{equation} and the flow $\Phi_0^{t,t^{\flat}}(z^{\flat}) = (q_0(t), p_0(t))$ is associated with the smooth Hamiltonian $h_0(z) = \dfrac{\lvert \xi \lvert^2}{2} + v(x)$. In other words, $z_0 = (q_0,p_0)$ satisfies the following system of ordinary differential equations
\begin{equation}
\label{eqclassicaltrajzero}
    \left\lbrace 
    \begin{array}{cll}
        \dot{q}_{0} & = & p_{0} \\
        \dot{p}_{0} & = & -\nabla v(q_{0})\\
        (q_{0}(t^{\flat}),p_{0}(t^{\flat})) & = & z^{\flat}.
    \end{array}
    \right.
\end{equation}
\noindent We also introduce the associated action $S_0(t;t^{\flat},z^{\flat})$, given by
\begin{equation}
\label{actionzero}
    S_{0}(t;t^{\flat},z^{\flat}) = \ds\int_{t^{\flat}}^t \lvert p_{0}(s)\lvert^2 - \, h_{0}(z_{0}(s)) \, \dint s.
\end{equation}

\begin{remark}
    We can note that when $t = t^{\flat} \pm \delta$, $s = \pm \, s_0 := \pm \, \dfrac{\delta}{\sqrt{\varepsilon}}$. Since we will choose $\delta$ such that $\sqrt{\varepsilon} \delta^{-1} \ll 1$, we will have ~$s_0 \gg 1$.
\end{remark}

\noindent We also consider the new space variable $y = \dfrac{x - q_0}{\sqrt{\varepsilon}}$. The profile $u^{\varepsilon}$ defined in \eqref{newinconnue} will be analyzed in terms of this variable. Subsequently, it is advantageous to use the norm
\begin{equation*}
    \norm{f}{\widetilde{\Sigma}^{k}_{\varepsilon}(\R^d,\C)} = \underset{\lvert \alpha \lvert + \lvert \beta \lvert \, \leqslant k}{\text{sup}} \varepsilon^{\frac{\lvert \alpha \lvert + \lvert \beta \lvert }{2}} \norm{f}{{\Sigma}^{\lvert \alpha \lvert + \lvert \beta \lvert}(\R^d,\C)}.
\end{equation*}

\subsubsection{The Landau--Zener model}

\noindent The previous change of unknown and variables allows us to reduce the analysis of our problem to solving a Landau--Zener model (Subsection~\ref{subsec:LZmodel}), since the function ~$u^{\varepsilon}$, defined in \eqref{newinconnue}, satisfies a system that can be viewed as a perturbation of such a model. Landau--Zener systems were introduced in the $1950$s in the context of quantum mechanical transition processes (see references \cite{livrelandau} and \cite{zener1932}) and can be written in the following form 
\begin{equation}
\label{eqLZ}
\dfrac{1}{i} \partial_s u_{\text{LZ}}^{\varepsilon}(s,z) = \begin{pmatrix}
            s + z_1 & z_2 \\
            z_2 & - s - z_1
        \end{pmatrix} u_{\text{LZ}}^{\varepsilon}(s,z)
\end{equation}
\noindent for $z = (z_1, z_2) \in \C^2$.\\ 

\noindent In this paper, we will use the results from \cite{fermaniangerard2002} and \cite{kammererlasser2008} to obtain the behavior of the solutions to ~\eqref{eqLZ} as $s \to \pm \infty$. These papers provide asymptotics of $u_{\text{LZ}}^{\varepsilon}$ as $s \to \pm \infty$, that we call \textit{Landau--Zener asymptotics}. They also establish an algebraic relation relating the profile at $s \to - \infty$ and $s \to + \infty$.

\subsubsection{Phases analysis and justification of the drift}

\noindent In view of the change of unknown, the phases of the ingoing and outgoing wave packets must be compared with the phase $\frac{i}{\varepsilon} S_0 + \frac{i}{\sqrt{\varepsilon}} p_0 \cdot \frac{x - q_0}{\sqrt{\varepsilon}}$ of the average wave packet for $t$ close to $t^{\flat}$. To perform this comparison, the change of time enables us to use the Taylor expansions of classical trajectories and actions, which are studied in Section ~\ref{sec:analysisCQ}. This comparison allows us to relate the wave packet profiles to the Landau--Zener asymptotics. This allows us to identify four distinct contributions in the following two terms $$\frac{i}{\varepsilon} S_- + \frac{i}{\sqrt{\varepsilon}} p_- \cdot \frac{x - q_-}{\sqrt{\varepsilon}} \quad \text{and} \quad \frac{i}{\varepsilon} \widetilde{S}_+ + \frac{i}{\sqrt{\varepsilon}} \widetilde{p}_+ \cdot \frac{x - \widetilde{q}_+}{\sqrt{\varepsilon}}$$ 
\noindent The first, $\Lambda^{\varepsilon,\alpha}$, coincides with the Landau--Zener asymptotics; the second, $\frac{i}{2} G_{\alpha}(t^{\flat} + \sqrt{\varepsilon} s) y \cdot y$, allows to regularize the profile; the third $\Phi^{\varepsilon,\alpha}$, participates to the definition of the transfer coefficient for the profiles; and the last term is used to define the drift. This is formalized in the following lemma.

\begin{lemma} \label{prop:changetime} We consider \( M \) and \( r_0 \) two positive real numbers. There exists a constant $C > 0$ such that the following holds.

\smallskip

\noindent For all $(\varepsilon,\delta, \alpha_0,R) \in (0,1) \times (0,1] \times \R_{+} \times [1, + \infty [ \) such that $\sqrt{\varepsilon} \delta^{-1} \leqslant 1$ and $\alpha_0 \leqslant \dfrac{R}{2} \sqrt{\varepsilon}$, for all compact interval\review{s} $\mathrm{I} \subset \R$ with $t_0 = \text{min}(\mathrm{I})$ and $z_0 \in \mathcal{A}_{-}(M,\alpha_0,r_0,\delta,\mathrm{I})$, we define\review{,} for $(s,\eta) \in \R \times \R^2$\review{,} the phases $\Lambda^{\varepsilon,\alpha}$ and $\Phi^{\varepsilon,\alpha}$ by 
\begin{equation}
\label{phase}
    \Lambda^{\varepsilon,\alpha}(s,\eta) = \dfrac{1}{2r} \Big[ ( sr + \eta \cdot \Vec{e}_{\theta})^2 + \Big(\eta \cdot \Vec{e}_{\theta}^{\perp} + \frac{\alpha}{\sqrt{\varepsilon}}\Big)^2 \ln(|s|\sqrt{r}) \Big]
\end{equation}
\noindent and
\begin{multline}
\label{phase2}
    \Phi^{\varepsilon,\alpha}(\eta) = \dfrac{\alpha^2}{4r \varepsilon} - \dfrac{\alpha^2}{2r \varepsilon} \ln \Big( \dfrac{\alpha}{2\sqrt{r \varepsilon}} \Big) - \dfrac{1}{2r} (\eta \cdot \Vec{e}_{\theta})^2 \\ - \dfrac{\alpha}{r \sqrt{\varepsilon}} \eta \cdot \Vec{e}_{\theta}^{\perp} \ln \Big( \dfrac{\alpha}{2\sqrt{r \varepsilon}} \Big) - \dfrac{1}{2r} (\eta \cdot \Vec{e}_{\theta}^{\perp})^2 \ln \Big( \dfrac{1}{2\sqrt{r \varepsilon}} \Big) + \dfrac{1}{2} \Gamma_1 y \cdot y \review{,}
\end{multline}

\noindent where the matrix $\Gamma_1$ is explicitly given in \eqref{matGamma1}. 

\medskip

\noindent For $t = t^{\flat} + s \sqrt{\varepsilon}$, $ y = \frac{x - q_0(t)}{\sqrt{\varepsilon}}$, $\eta(y) = \diff w(q^{\flat})y$ and the matrix-valued function $G_{\alpha}$ defined in \eqref{def:phase}, \review{the following pointwise estimates hold.}
\begin{enumerate}[leftmargin=*, labelindent=0pt]
    \item If $s \to - \infty$ (i.e. $\varepsilon \to 0$, $\lvert t - t^{\flat} \lvert \, \leqslant \delta$ and $t < t^{\flat}$)\review{, we have}
\begin{multline*}
\dfrac{i}{\varepsilon}  S_{-}(t;t_0,z_0) + \dfrac{i}{\varepsilon} p_{-}(t) \cdot (x - q_{-}(t)) = \dfrac{i}{\varepsilon} S_{-}(t^{\flat};t_0,z_0) + \dfrac{i}{\varepsilon} S_{0}(t;t^{\flat},z^{\flat}) + \dfrac{i}{\sqrt{\varepsilon}} p_0(t) \cdot y - i \Lambda^{\varepsilon,\alpha}(s, \eta(y)) \\ - i \Phi^{\varepsilon,\alpha}(\eta(y)) + \dfrac{i}{2} G_{\alpha}(t^{\flat} + \sqrt{\varepsilon}s) y \cdot y - \dfrac{i \alpha}{r \sqrt{\varepsilon}} \eta(y) \cdot \Vec{e}_{\theta} + \sigma.
\end{multline*}
    \item If $s \to + \infty$ (i.e. $\varepsilon \to 0$, $\lvert t - t^{\flat} \lvert \, \leqslant \delta$ and $t > t^{\flat}$)\review{, we have}
\begin{multline*}
\dfrac{i}{\varepsilon} S_{0}(t;t^{\flat},z^{\flat}) + \dfrac{i}{\sqrt{\varepsilon}} p_0(t) \cdot y = \dfrac{i}{\varepsilon}  S_{-}(t;t^{\flat},z^{\flat}) + \dfrac{i}{\varepsilon} p_{-}(t) \cdot (x - q_{-}(t)) - i \Lambda^{\varepsilon,\alpha}(s, \eta(y)) \\ - i \Phi^{\varepsilon,\alpha}(\eta(y)) + \dfrac{i}{2} G_{\alpha}(t^{\flat} + \sqrt{\varepsilon}s) y \cdot y + \dfrac{i \alpha}{r \sqrt{\varepsilon}} \eta(y) \cdot \Vec{e}_{\theta} +  \sigma \text{ and }
\end{multline*}
\begin{multline*}
\dfrac{i}{\varepsilon} S_{0}(t;t^{\flat},z^{\flat}) + \dfrac{i}{\sqrt{\varepsilon}} p_0(t) \cdot y = \dfrac{i}{\varepsilon} \widetilde{S}_{+}(t;t^{\flat},\widetilde{z}^{\flat}) + \dfrac{i}{\varepsilon} \widetilde{p}_{+}(t) \cdot (x - \widetilde{q}_{+}(t)) - \dfrac{i}{\sqrt{\varepsilon}} \delta_{p^{\flat}} \cdot y + i \Lambda^{\varepsilon,\alpha}(s, \eta(y)) \\ + i \Phi^{\varepsilon,\alpha}(\eta(y)) - \dfrac{i}{2} G_{\alpha}(t^{\flat} + \sqrt{\varepsilon}s) y \cdot y - \dfrac{i \alpha}{r \sqrt{\varepsilon}} \eta(y) \cdot \Vec{e}_{\theta} + \sigma.
\end{multline*}
\end{enumerate}
\noindent \review{The remainder terms $\sigma$ appearing in the two previous estimates satisfies} $$\lvert \sigma \lvert \, \leqslant C \Big( \dfrac{1}{|s|} + \sqrt{\varepsilon} |s|^3 + \sqrt{\varepsilon} s^2 \lvert y \lvert \Big).$$
\end{lemma}

\noindent In the proof, by analyzing the minus-phase for $t < t^{\flat}$ (using the first point of the previous lemma) and the plus-phase for $t > t^{\flat}$ (using the second point of the previous lemma), we deduce that the function $\widetilde{u}_{+}^{\text{out}}$, introduced in \eqref{coro:transitioninandout}, must satisfy
\begin{equation*}
    \widetilde{u}_{+}^{\text{out}}(y) 
    = \text{Exp}\Big( \frac{i}{\varepsilon} S_{-}^{\flat} \Big) 
      \text{Exp}\Big( -\frac{i}{\sqrt{\varepsilon}}\, \delta_{p^{\flat}} \cdot y 
      - \frac{2 i \alpha}{r \sqrt{\varepsilon}}\, \eta(y) \cdot \vec{e}_{\theta} \Big)\,
      a\Big( \review{\dfrac{\eta \cdot \Vec{e}_{\theta}^{\perp}}{\sqrt{r}} + \dfrac{\alpha}{\sqrt{r\varepsilon}}} \Big)\,
      u_{-}^{\text{in},\alpha}(y).
\end{equation*}
In order to ensure sufficient regularity (we want that all its derivatives and momenta are uniformly bounded in $\mathrm{L}^2$), the term 
\[
- \frac{i}{\sqrt{\varepsilon}}\, \delta_{p^{\flat}} \cdot y - \frac{2 i \alpha}{r \sqrt{\varepsilon}}\, \mathrm{d}w(q^{\flat})\, y \cdot \vec{e}_{\theta}
\]
must be canceled. The definition of $\delta_{p^{\flat}}$ given in \eqref{eq:ourdrift} is motivated by the latter result.

\subsubsection{The ingoing wave packet in the new coordinates}

\noindent The third step of the proof consists in identifying the ingoing wave packet at time \( t^{\flat} - \delta \) (see Subsection~\ref{subsec:ingoingWP}), that is, determining the initial data for the Landau--Zener model as \( s \to -\infty \). Using the asymptotics when $s$ tends to $- \infty$ of the solution of \eqref{eqLZ} and the previous lemma, we obtain the following theorem.

\begin{theorem}[The ingoing wave packet] \label{theo:ingoingWP} Let Assumptions \ref{hypsurV2} and \ref{hypcrossing} hold. We consider \( M, \, r_0 \) two positive real numbers. There exists a constant $C > 0$ such that the following holds.

\noindent For all $(\varepsilon,\delta,\alpha_0,R,k) \in (0,1) \times (0,1] \times \R_{+} \times [1, + \infty [ \, \times \, \N \) such that $\sqrt{\varepsilon} \delta^{-1} \leqslant 1$ and $\alpha_0 \leqslant \dfrac{R}{2} \sqrt{\varepsilon}$, for all compact interval\review{s} $\mathrm{I} \subset \R$ with $t_0 = \text{min}(\mathrm{I})$, $z_0 \in \mathcal{A}_{-}(M,\alpha_0,r_0,\delta,\mathrm{I})$, the solution $\psi^{\varepsilon}$ of ~\eqref{equation}, with an initial data $\psi_0^{\varepsilon}$ as in Assumption \ref{hypsurID}, satisfies \eqref{newinconnue} at time $t = t^{\flat} - \delta$ \Bigg(i.e. at time $- s_0 = - \dfrac{\delta}{\sqrt{\varepsilon}}$\Bigg) with 
\begin{equation}
\label{Uepsminus}
    u^{\varepsilon}(-s_0,y) = e^{-i \Lambda^{\varepsilon,\alpha}(-s_0,\eta(y))} \alpha_2^{\text{in}}(\eta(y)) \, \Vec{Y}_{\flat} \, +  \, \sigma(\varepsilon, \delta)
\end{equation}
\noindent where
\begin{equation*}
    \| \sigma(\varepsilon, \delta) \|_{_{\widetilde{\Sigma}_{\varepsilon}^{k}}} \leqslant C (\sqrt{\varepsilon} \delta^{-1} + \varepsilon^{\frac{3}{2}} \delta^{-4} + \varepsilon^{-1} \delta^3)(1 + \lvert \ln \delta \lvert).
\end{equation*}
\noindent The vector $\Vec{Y}_{\flat}$ is introduced in \eqref{eq:ICvecafter}, the phase $\Lambda^{\varepsilon,\alpha}$ is defined by \eqref{phase} and the function $y \mapsto \alpha_2^{\text{in}}(\eta(y))$ is given by
\begin{multline}
\label{alphain}
    \alpha_2^{\text{in}}(\eta(y)) = \text{Exp} \Bigg( \dfrac{i}{\varepsilon} S_{-}^{\flat} - \dfrac{i\alpha^2}{4r \varepsilon} + \dfrac{i \alpha^2}{2r \varepsilon} \ln\Big(\dfrac{\alpha}{2\sqrt{r\varepsilon}} \Big) + \dfrac{i}{2r} (\eta(y) \cdot \Vec{e}_{\theta})^2 + \dfrac{i \alpha}{r \sqrt{\varepsilon}} \eta(y) \cdot \Vec{e}_{\theta}^{\perp} \ln\Big(\dfrac{\alpha}{2\sqrt{r\varepsilon}} \Big)\\ + \dfrac{i}{2r} (\eta(y) \cdot \Vec{e}_{\theta}^{\perp})^2 \ln\Big(\dfrac{1}{2\sqrt{r\varepsilon}} \Big) - \dfrac{i}{2} \Gamma_1 y \cdot y \Bigg) \text{Exp} \Bigg( - \dfrac{i \alpha}{r \sqrt{\varepsilon}} \eta(y) \cdot \Vec{e}_{\theta} \Bigg)u_{-}^{\text{in},\alpha}(y),
\end{multline}
\noindent with $u_{-}^{\text{in},\alpha}$ is defined in \eqref{eq:uinalpha}.
\end{theorem}

\subsubsection{The transfer laws and the outgoing wave packet}

\noindent Thanks to the study of the Landau--Zener model and the algebraic relation between the asymptotic solution of \eqref{eqLZ} in $-\infty$ and the solution at $+ \infty$, we establish the following theorem.

\begin{theorem} \label{theo:outgoingWPfctu} Let Assumptions \ref{hypsurV2} and \ref{hypcrossing} hold. We consider \( M, \, r_0 \) two positive real numbers and $N_0 \in \N^{*} \review{:= \N \setminus \lbrace 0 \rbrace}$. There exists a constant $C > 0$ such that the following holds. For all $R \in [1, + \infty[$, there exists a cut-off function $\chi_0 : \R^d \longrightarrow [0,1]$ satisfying 
\begin{equation}
\label{eq:cutoff}
|\eta(y)| := \Big\lvert \diff w(q^{\flat}) y \Big\lvert \, \leqslant \dfrac{R}{2} \text{ when } \dfrac{y}{R} \in \text{supp}(\chi_0),
\end{equation} such that for all $(\varepsilon,\delta,\alpha_0,k) \in (0,1) \times (0,1] \times \, \R_{+} \times \, \N \) with
\begin{equation}
\label{eq:conditiontheo}
    \sqrt{\varepsilon} \delta^{-1} \leqslant 1, \alpha_0 \leqslant \dfrac{R\sqrt{\varepsilon}}{2}, R^2 \sqrt{\varepsilon} \ll 1, R \delta \ll 1 \quad \text{and} \quad R \varepsilon^2 \delta^{-4} \ll 1,
\end{equation} for all compact interval\review{s} $\mathrm{I} \subset \R$ with $t_0 = \text{min}(\mathrm{I})$ and for all $z_0 \in \mathcal{A}_{-}(M,\alpha_0,r_0,\delta,\mathrm{I})$, the solution $\psi^{\varepsilon}$ of Equation \eqref{equation}, with an initial data $\psi_0^{\varepsilon}$ as in Assumption ~\ref{hypsurID}, satisfies \eqref{newinconnue} at time $t = t^{\flat} + \delta$ \Big( i.e. $s_0 = \dfrac{\delta}{\sqrt{\varepsilon}}$ \Big) with 
\begin{multline*}
    u^{\varepsilon}(s_0,y) = \chi_0 \Big( \dfrac{y}{R} \Big) \times \Big( e^{i \Lambda^{\varepsilon,\alpha}(s_0,\eta(y))} \alpha_1^{\text{out}}(\eta(y)) \Vec{Y}_{\flat}^{\perp} + e^{-i \Lambda^{\varepsilon,\alpha}(s_0,\eta(y))} \alpha_2^{\text{out}}(\eta(y)) \Vec{Y}_{\flat} \Big) + \sigma(\varepsilon, \delta) \\ \text{where } \| \sigma(\varepsilon, \delta) \|_{_{\widetilde{\Sigma}_{\varepsilon}^{k}}} \leqslant C \Big( (\sqrt{\varepsilon} \delta^{-1} + \varepsilon^{\frac{3}{2}} \delta^{-4} + \delta + R \varepsilon^{-1} \delta^3 + R^3 \sqrt{\varepsilon} \delta^{-1} + R^{-N_0})(1 + \lvert \ln \delta \lvert) \Big).
\end{multline*}
\noindent In the formula above, we have $\eta(y) := \diff w(q^{\flat}) y$ and the functions $\begin{pmatrix}
        \alpha_1^{\text{out}} \\
        \alpha_2^{\text{out}}
    \end{pmatrix}$ are defined for all $\eta \in \C^2$ by
\begin{equation}
\label{eq:relationalphainetout}
    \begin{pmatrix}
        \alpha_1^{\text{out}}(\eta) \\
        \alpha_2^{\text{out}}(\eta)
    \end{pmatrix} = \mathcal{S}\Bigg(\review{\dfrac{\eta \cdot \Vec{e}_{\theta}^{\perp}}{\sqrt{r}} + \dfrac{\alpha}{\sqrt{r\varepsilon}}} \Bigg) \begin{pmatrix}
        0 \\
        \alpha_2^{\text{in}}(\eta)
    \end{pmatrix} \quad \text{with} \quad \mathcal{S}(z) = \begin{pmatrix}
        a(z) & -\overline{b}(z) \\
        b(z) & a(z)
    \end{pmatrix} \quad \text{for } z \in \C
\end{equation}
\noindent and the functions $a$ and $b$ given in \eqref{fonctionaetb}.
\end{theorem}

\noindent The fact that the function ~$a$ is nonzero explains why a new wave packet is generated on the other mode\review{:} both functions $\alpha_1^{\text{out}}$ and $\alpha_2^{\text{out}}$ are nonzero. The second point of Lemma \ref{prop:changetime} allows us to return to the original coordinates (see Subsection \ref{subsec:outgoingWP}) and to obtain the following theorem.

\begin{theorem}[The outgoing wave packet] \label{theo:outgoingWP} Let Assumptions \ref{hypsurV2} and \ref{hypcrossing} hold. We consider \( M, \, r_0 \) two positive real numbers and $N_0 \in \N^{*}$. There exists a constant $C > 0$ such that the following holds. For all $(\varepsilon,\delta,\alpha_0,R,k) \in (0,1) \times (0,1] \times \, \R_{+} \times \, [1, + \infty[ \, \times \, \N \) satisfying conditions \eqref{eq:conditiontheo}, for all compact interval\review{s} $\mathrm{I} \subset \R$ with $t_0 = \text{min}(\mathrm{I})$ and for all $z_0 \in \mathcal{A}_{-}(M,\alpha_0,r_0,\delta,\mathrm{I})$, the solution $\psi^{\varepsilon}$ of Equation \eqref{equation}, with an initial data $\psi_0^{\varepsilon}$ as in Assumption \ref{hypsurID}, satisfies at time $t = t^{\flat} + \delta$ 
\begin{equation}
\label{outgoingWP}
\psi^{\varepsilon} = \psi^{\varepsilon}_{+} + \psi^{\varepsilon}_{-} + \sigma(\varepsilon, \delta, R, N_0)  
\end{equation}
\noindent with for all $x \in \R^d$
\begin{equation*}
    \psi^{\varepsilon}_{-}(t^{\flat} + \delta,x) = e^{\frac{i}{\varepsilon} S_{-}(t^{\flat} + \delta;t^{\flat},z^{\flat})} \mathrm{WP}_{\Phi_-^{t,t^{\flat}}(z^{\flat})}^{\varepsilon} \Big(\mathrm{Exp} \Big( \frac{i}{2} G_{\alpha}(t^{\flat} + \delta) x \cdot x \Big) u_{-}^{\text{out},\alpha}(x) \Big) \Vec{Y}_{\flat}^{\perp},
\end{equation*}
\begin{equation*}
    \psi^{\varepsilon}_{+}(t^{\flat} + \delta,x) = e^{\frac{i}{\varepsilon}  \tilde{S}_{+}(t^{\flat} + \delta;t^{\flat},\tilde{z}^{\flat})} \mathrm{WP}_{\widetilde{\Phi}_+^{t,t^{\flat}}(\widetilde{z}^{\flat})}^{\varepsilon} \Big(\mathrm{Exp} \Big( - \frac{i}{2} G_{\alpha}(t^{\flat} + \delta) x \cdot x \Big) \widetilde{u}_{+}^{\text{out},\alpha}(x) \Big) \Vec{Y}_{\flat}
\end{equation*}
\noindent where the symmetric matrix-valued function $G_{\alpha}$ is explicitly given in \eqref{def:phase} below, $(\widetilde{u}_{+}^{\text{out},\alpha}, u_{-}^{\text{out},\alpha})$ are defined in \eqref{coro:transitioninandout}, the vector $(\Vec{Y}_{\flat}, \Vec{Y}_{\flat}^{\perp})$ are introduced in \eqref{eq:ICvecafter} and \[ \| \sigma(\varepsilon, \delta, R, N_0)  \|_{_{\Sigma_{\varepsilon}^{k}}} \leqslant C \Bigg( (\sqrt{\varepsilon} \delta^{-1} + \varepsilon^{\frac{3}{2}} \delta^{-4} + \delta + R \varepsilon^{-1} \delta^3 + R^3 \sqrt{\varepsilon} \delta^{-1} + \varepsilon^{-1} \delta^3 + R^{-N_0})(1 + \lvert \ln \delta \lvert) \Bigg).\]
\end{theorem}

\noindent Theorem~\ref{theo:main} is derived from Theorem~\ref{theo:outgoingWP} by propagating the solution outside the gap region using the more general adiabatic theorem~\cite[Proposition~3.1]{FGH2021} (slightly adapting the proof for a different initial error) for $t \geqslant t^{\flat} + \delta$ and by making a suitable choice, detailed at the end of the paper, of $N_0$, $\alpha_0$ and of the parameters $\delta$ and $R$ as functions of $~\varepsilon$, to obtain the desired remainder.\\

\noindent The following diagram illustrates the evolution of the phase and the transfer coefficient of the profile during the propagation. The central comments apply to both modes. The drift is justified in the final box associated with the plus-mode. The functions $a$ and $b$, which appear in the diagram, originate from the analysis of the Landau--Zener model and are explicitly defined in ~\eqref{fonctionaetb}.

\begin{figure}[htbp!]
    \hspace*{-2cm}
    \begin{tikzpicture}[
        node distance=1.4cm and 1.4cm,
        boxminus/.style={
            draw, fill=cyan!20, minimum width=4.5cm, minimum height=1.2cm,
            font=\sffamily\footnotesize, align=center, rounded corners=8pt
        },
        boxplus/.style={
            draw, fill=blue!20, minimum width=4.5cm, minimum height=1.2cm,
            font=\sffamily\footnotesize, align=center, rounded corners=8pt
        },
        arrow/.style={->, thick},
        dashed arrow/.style={->, thick, dashed}
    ]

    % MINUS MODE
    \node[boxminus] (mn1) at (0,0) {
    \begin{tabular}{c}
    \textbf{Phase}\\
    $\dfrac{i}{\varepsilon}  S_{-}(t_0,z_0) + \dfrac{i}{\varepsilon} p_{-} \cdot (x - q_{-})$\\[0.8em] \hline \\[-0.8em]
    \textbf{Profile coefficient}\\
    1
    \end{tabular}};
    
    \node[boxminus, below=of mn1] (mn2) {
    \begin{tabular}{c}
    \textbf{Phase}\\ $\dfrac{i}{\varepsilon} S_0 + \dfrac{i}{\sqrt{\varepsilon}} p_0 \cdot y - \dfrac{i \alpha}{r \sqrt{\varepsilon}} \eta(y) \cdot \vec{e}_{\theta}$\\[0.8em] \hline \\[-0.8em]
    \textbf{Profile coefficient}\\ $\exp(-i \Phi^{\varepsilon,\alpha}(\eta(y)))$
    \end{tabular}};
    
    \node[boxminus, below=of mn2] (mn3) {
    \begin{tabular}{c}
    \textbf{Phase}\\ $\dfrac{i}{\varepsilon} S_0 + \dfrac{i}{\sqrt{\varepsilon}} p_0 \cdot y - \dfrac{i \alpha}{r \sqrt{\varepsilon}} \eta(y) \cdot \vec{e}_{\theta}$\\[0.8em] \hline \\[-0.8em]
    \textbf{Profile coefficient}\\ $\exp(-i \Phi^{\varepsilon,\alpha}(\eta(y))) \cdot \overline{b}\left(\review{\frac{\eta \cdot \Vec{e}_{\theta}^{\perp}}{\sqrt{r}} + \frac{\alpha}{\sqrt{r\varepsilon}}}\right)$
    \end{tabular}};
    
    \node[boxminus, below=of mn3] (mn4) {
    \begin{tabular}{c}
    \textbf{Phase}\\ $\dfrac{i}{\varepsilon} S_{-} + \dfrac{i}{\varepsilon} p_{-} \cdot (x - q_{-})$\\ \\[0.8em] \hline \\[-0.8em]
    \textbf{Profile coefficient}\\ $\exp(-2i \Phi^{\varepsilon,\alpha}) \cdot \overline{b}\left(\review{\frac{\eta \cdot \Vec{e}_{\theta}^{\perp}}{\sqrt{r}} + \frac{\alpha}{\sqrt{r\varepsilon}}}\right)$
    \end{tabular}};

    % PLUS MODE - aligned horizontally with minus-mode blocks
    \node[boxplus, right=6cm of mn1] (pn1) {
    \begin{tabular}{c}
    \textbf{Phase}\\ 0 \\[0.8em] \hline \\[-0.8em]
    \textbf{Profile coefficient} \\ 1 (the profile is equal to 0)
    \end{tabular}};
    
    \node[boxplus, below=of pn1] (pn2) {
    \begin{tabular}{c}
    \textbf{Phase}\\ 0\\[0.8em] \hline \\[-0.8em]
    \textbf{Profile coefficient} \\ 1 (the profile is equal to 0)
    \end{tabular}};
    
    \node[boxplus, below=of pn2] (pn3) {
    \begin{tabular}{c}
    \textbf{Phase}\\ $\dfrac{i}{\varepsilon} S_0 + \dfrac{i}{\sqrt{\varepsilon}} p_0 \cdot y - \dfrac{i \alpha}{r \sqrt{\varepsilon}} \eta(y) \cdot \vec{e}_{\theta}$\\[0.8em] \hline \\[-0.8em]
    \textbf{Profile coefficient}\\ $\exp(-i \Phi^{\varepsilon,\alpha}) \cdot a\left(\review{\frac{\eta \cdot \Vec{e}_{\theta}^{\perp}}{\sqrt{r}} + \frac{\alpha}{\sqrt{r\varepsilon}}}\right)$
    \end{tabular}};
    
    \node[boxplus, below=of pn3] (pn4) {
        \begin{tabular}{c}
        \textbf{Phase}\\
        $\dfrac{i}{\varepsilon} \widetilde{S}_{+} + \dfrac{i}{\varepsilon} \widetilde{p}_{+} \cdot (x - \widetilde{q}_{+})$\\[0.5em] $\underbrace{- \dfrac{2i \alpha}{r \sqrt{\varepsilon}} \eta(y) \cdot \vec{e}_\theta - \delta_{p^\flat} \cdot y}$\\ must be equal to 0 \\ $\rightsquigarrow$ definition of the drift\\ [0.8em] \hline \\[-0.8em]
        \textbf{Profile coefficient}\\
        $a\left(\review{\frac{\eta \cdot \Vec{e}_{\theta}^{\perp}}{\sqrt{r}} + \frac{\alpha}{\sqrt{r\varepsilon}}}\right)$
        \end{tabular}
    };

% TITRES DES COLONNES
\node[font=\bfseries\sffamily, align=center] at ($(mn1)+(0,1.5)$) {Minus-mode};
\node[font=\bfseries\sffamily, align=center] at ($(pn1)+(0,1.5)$) {Plus-mode};

%ARROWS
    \draw[dashed, ->, thick] (mn1.south) -- (mn2.north) node[midway, left, align=center] {Change of representation \\ with (1) of Lemma~\ref{prop:changetime}};
    \draw[->, thick] (mn2.south) -- (mn3.north) node[midway, left, align=center] {Landau--Zener transition\\(Proposition~\ref{prop:soleqmodel})};
    \draw[dashed, ->, thick] (mn3.south) -- (mn4.north) node[midway, left, align=center] {Change of representation \\ with (2) of Lemma~\ref{prop:changetime}};

    \draw[dashed, ->, thick] (pn1.south) -- (pn2.north) node[midway, right, align=center] {Change of representation \\ with (1) of Lemma~\ref{prop:changetime}};
    \draw[->, thick] (pn2.south) -- (pn3.north) node[midway, right, align=center] {Landau--Zener transition\\(Proposition~\ref{prop:soleqmodel})};
    \draw[dashed, ->, thick] (pn3.south) -- (pn4.north) node[midway, right, align=center] {Change of representation \\ with (2) of Lemma~\ref{prop:changetime}};

    % COMMENTAIRES CENTRAUX
    \node[align=center, font=\footnotesize] at ($(mn1)!0.5!(pn1)$) {\textit{Representation of the initial wave packet}};
    \node[align=center, font=\footnotesize] at ($(mn2)!0.5!(pn2)$) {\begin{tabular}{c}
    \textit{Representation with the average}\\
    \textit{wave packet using the change} \\
    \textit{of unknown \eqref{newinconnue}}
    \end{tabular}};
    \node[align=center, font=\footnotesize] at ($(mn3)!0.5!(pn3)$) {\begin{tabular}{c}
    \textit{Still in the representation with the}\\
    \textit{average wave packet: the profile} \\
    \textit{coefficients have been replaced using}\\
    \textit{the Landau--Zener model}
    \end{tabular}};
    \node[align=center, font=\footnotesize] at ($(mn4)!0.5!(pn4)$) {\begin{tabular}{c}
    \textit{Return to the initial representation}
    \end{tabular}};

    \end{tikzpicture}
    \caption{Evolution of phase and transfer coefficient during wave packet propagation.}
    \label{fig:phase_diagram_both}
\end{figure}
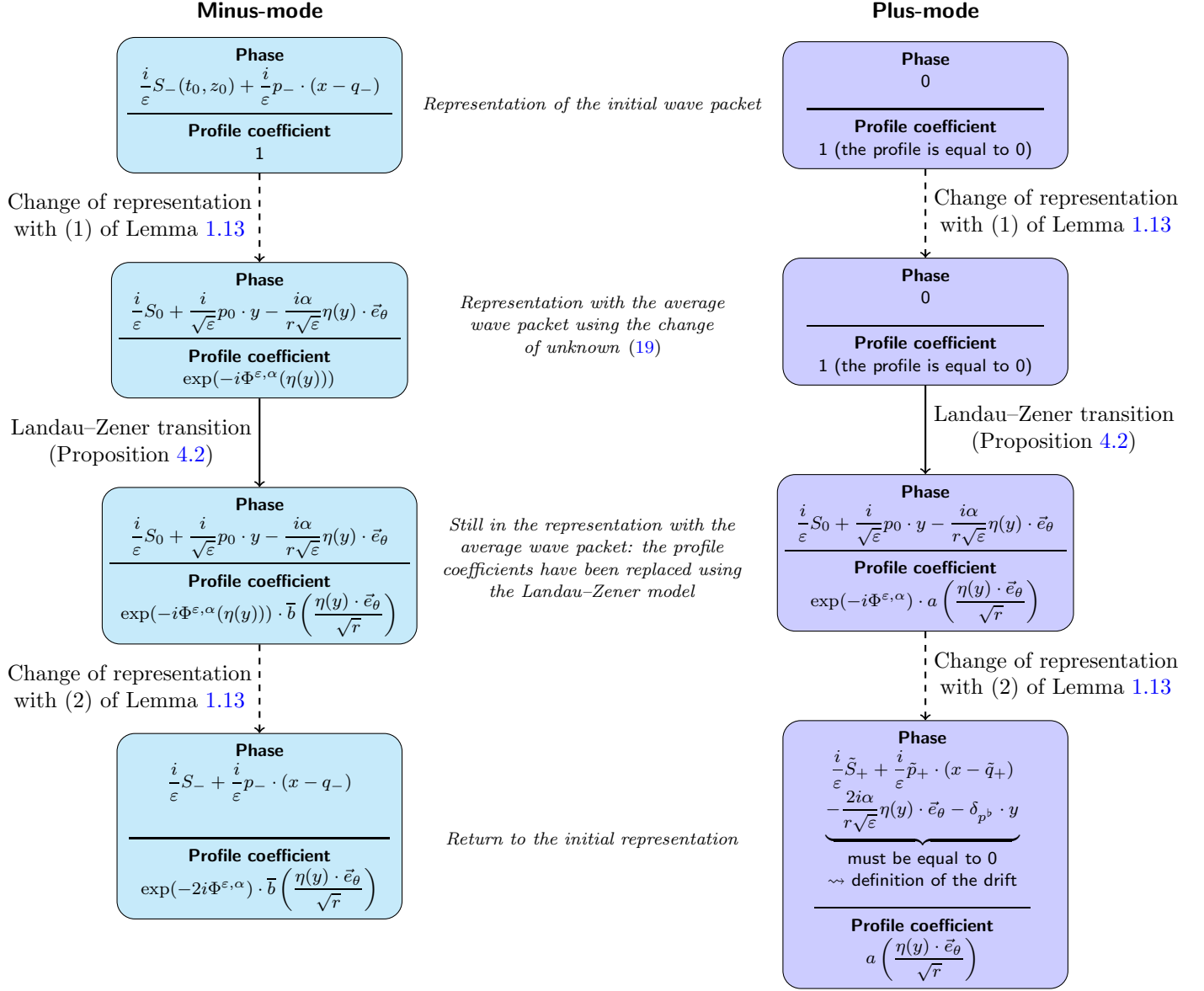
\FloatBarrier

\noindent \textbf{Acknowledgments.} The author would like to express her sincere gratitude to Clotilde Fermanian Kammerer, Lysianne Hari and Martin Averseng for their precious advice, productive discussions and constructive suggestions throughout the preparation of this paper. The author would also like to thank Caroline Lasser for her time and help during the writing. This work is supported by the Région Pays de la Loire via the Connect Talent Project HiFrAn $2022$ $07750$, as well as by the France $2030$ program, Centre Henri Lebesgue ANR-$11$-LABX-$0020$-$01$.

%%%%%%%%%%%%%%%%%%%%%%%%%%%%%%%%%%%%%%%%%%%%%%%%%%%%%%%%%%%

\section{Analysis of classical quantities}
\label{sec:analysisCQ}

\noindent \review{In this section, we analyze the classical quantities in a neighborhood of the crossing set, which will be used in the construction of our approximate solution. More precisely, we study the behavior close to $t^{\flat}$ of the classical trajectories (see Subsection~\ref{subsec:studyCT}), the actions (see Subsection~\ref{subsec:studyAC}) and the vectors that guide our approximate solution (see Subsection~\ref{subsec:studyVec}).}\\

\noindent We fix the parameters \((M, \alpha_0,r_0,\delta) \in \R_{+}^{*} \times \R_{+} \times \R_{+}^{*} \times (0, 1] \) and a compact interval \( \mathrm{I} \subset \mathbb{R} \) with \( t_0 = \min(\mathrm{I}) \). For convenience, throughout the proofs, we will denote by \( C \) various constants that may depend on \( M \) and $r_0$. Moreover, for all ~$k \in \N$, $\sigma_k$ will denote a remainder that satisfies \[\lvert \sigma_k \lvert \, \leqslant C \lvert t - t^{\flat} \lvert^k \quad \text{with } C > 0.\]

\noindent According to Definition~\ref{defi:initialdata}, recall that when an initial point \( z_0 \) belongs to \( \mathcal{A}_{\pm}(M,\alpha_0,r_0,\delta,\mathrm{I}) \), this means that there exist a time \( t^{\flat} \in \mathrm{I} \) and a point \( z^{\flat} \in \Sigma_{\mathrm{nd}} \) such that
\([t^{\flat} - \delta,\, t^{\flat} + \delta] \subset \mathrm{I}\)
and the trajectory \(t \in \mathrm{I} \mapsto z_{\pm}(t) := \Phi^{t,t_0}(z_0)\) satisfies
\[z^{\flat} := \Phi^{t^{\flat},t_0}(z_0) \in \Sigma_{\mathrm{nd}}, \qquad
\Phi^{t,t_0}(z_0) \notin \Sigma \ \text{for } t \in \mathrm{I} \setminus \{ t^{\flat} \},
\qquad |w(q^{\flat})| \leqslant \alpha_0,
\qquad \text{and} \qquad \left| \mathrm{d}w(q^{\flat})\, p^{\flat} \right| > r_0.\]

\noindent We will adopt Notations~\ref{notation:alpha} throughout the rest of the article. Then, we write
\[\mathrm{d} w(q^{\flat})\, p^{\flat} := r \vec{e}_{\theta}, 
\quad 
w(q^{\flat}) := \alpha \vec{e}_{\theta}^{\perp},\]
where \( \vec{e}_{\theta} = (\cos\theta, \sin\theta) \in \mathbb{R}^2 \), \( r > 0 \) and \( \alpha \geqslant 0 \).

\subsection{The classical trajectories}
\label{subsec:studyCT}

\noindent It is interesting to compare the classical trajectories $(q_{\pm},p_{\pm})$ and $(\widetilde{q}_{\pm},\widetilde{p}_{\pm})$ with the trajectories $(q_0, p_0)$ solutions of \eqref{eqclassicaltraj}, \eqref{eqclassicaltrajdrift} and \eqref{eqclassicaltrajzero}, respectively. There exists a constant \( A > 0 \) such that for all \( t \in \mathrm{I} \) satisfying \( \alpha \leqslant \lvert t - t^{\flat} \rvert \), the following Taylor expansions hold
    \begin{equation}
    \label{DLp0q0}
    \left\lbrace 
    \begin{array}{cll}
        p_{0}(t)& = p^{\flat} - \nabla v(q^{\flat}) (t-t^{\flat}) + \sigma_{p_{0}} \quad \text{with } \lvert \sigma_{p_{0}}\lvert \, \leqslant A \lvert t - t^{\flat}\lvert^2, \\
        q_{0}(t) & = q^{\flat} + p^{\flat} (t-t^{\flat}) - \dfrac{1}{2} \nabla v(q^{\flat}) (t-t^{\flat})^2 + \sigma_{q_{0}} \quad \text{with } \lvert \sigma_{q_{0}}\lvert \, \leqslant A \lvert t - t^{\flat}\lvert^3.
    \end{array}
    \right.
\end{equation}

\noindent The next proposition gives the corresponding expansions for $p_{\pm}$ and $q_{\pm}$.

\begin{prop}
\label{prop:DLpq} For all $M > 0$ and $r_0 > 0$, there exist $\delta_0$ and a constant $C > 0$ such that the following holds. We consider $(\alpha_0, \delta) \in \R_+ \times \R_+^*$, $\mathrm{I} \subset \R$ a compact interval with $t_0 = \text{min}(\mathrm{I})$ and $z_0 \in \mathcal{A}_{\pm}(M,\alpha_0,r_0,\delta,\mathrm{I})$. For all $t \in \mathrm{I}$ satisfying $\alpha \leqslant | t - t^{\flat} | \leqslant \delta_0$, we have
\begin{multline*}
p_{\pm}(t) = p_0(t) \mp \dfrac{1}{r} \transpose{\diff w(q^{\flat})} \Vec{e}_{\theta} \Big[ \sqrt{\alpha^2 + r^2 (t-t^{\flat})^2} - \alpha \Big] \mp \dfrac{\alpha}{r} \transpose{\diff w(q^{\flat})} \Vec{e}_{\theta}^{\perp} \Argsh{\dfrac{r}{\alpha}(t-t^{\flat})} + \sigma_{p_{\pm}},
\end{multline*}

\begin{multline*}
q_{\pm}(t) = q_0(t) \mp \dfrac{1}{2r} \transpose{\diff w(q^{\flat})} \Vec{e}_{\theta} (t - t^{\flat}) \sqrt{\alpha^2 + r^2 (t-t^{\flat})^2} \mp \dfrac{\alpha^2}{2r^2} \transpose{\diff w(q^{\flat})} \Vec{e}_{\theta} \Argsh{\dfrac{r}{\alpha}(t-t^{\flat})} \pm \dfrac{\alpha}{r} \transpose{\diff w(q^{\flat})} \Vec{e}_{\theta} (t - t ^{\flat})\\ \mp \transpose{\diff w(q^{\flat})} \Vec{e}_{\theta}^{\perp} \Big[ \dfrac{\alpha(t-t^{\flat})}{r} \Argsh{\dfrac{r}{\alpha}(t-t^{\flat})} - \dfrac{\alpha}{r^2} \sqrt{\alpha^2 + r^2 (t-t^{\flat})^2} \Big] \pm \dfrac{\alpha^2}{r^2} \transpose{\diff w(q^{\flat})} \Vec{e}_{\theta}^{\perp} + \sigma_{q_{\pm}}
\end{multline*}
\noindent with
\[ \lvert \sigma_{p_{\pm}}\lvert \, \leqslant C \lvert t - t^{\flat}\lvert^2 \quad \text{and} \quad  \lvert \sigma_{q_{\pm}}\lvert \, \leqslant C \lvert t - t^{\flat}\lvert^3. \]
\end{prop}

\begin{proof}[Proof of Proposition~\ref{prop:DLpq}] Let $z_0$ be a point of $\mathcal{A}_{\pm}(M,\alpha_0,r_0,\delta,\mathrm{I})$. We consider \(t \in \mathrm{I} \) such that ~$t$ belongs to $[t^{\flat} + \alpha, t^{\flat} + \delta]$. We consider $s \in [t^{\flat} + \alpha, t ]$ (the proof is the same if $t \in [t^{\flat} - \delta, t^{\flat} - \alpha ]$ and $s \in [t , t^{\flat} - \alpha ]$). Using the Taylor-Lagrange formula, there exists $c \in \, ]t^{\flat},s[$ such that
\begin{equation*}
    w(q_{\pm}(s)) = \alpha \Vec{e}_{\theta}^{\perp} + r (s - t^{\flat})\Vec{e}_{\theta} + \dfrac{1}{2} f_{\pm}''(c)(s-t^{\flat})^2 \quad \text{where } f_{\pm} : t \mapsto  w(q_{\pm}(t)).
\end{equation*}
\noindent According to Corollary \ref{coro:deriveeCQ} of Appendix \ref{appendix:classicalq}, we deduce that
\begin{equation*}
\label{TEw}
    w(q_{\pm}(s)) = \alpha \Vec{e}_{\theta}^{\perp} + r (s - t^{\flat})\Vec{e}_{\theta} + \sigma_2. \tag{a}
\end{equation*}
\noindent Then, computations give that
\begin{equation*}
    |w(q_{\pm}(s))|^2 =  \big(  \alpha^2 + r^2(s-t^{\flat})^2 \big)  \Bigg( 1 + \underbrace{\dfrac{2(\alpha + r(s-t^{\flat}))\sigma_2 + |\sigma_2|^2}{\alpha^2 + r^2(s-t^{\flat})^2}}_{:= \, \sigma} \Bigg).
\end{equation*}

\noindent Using the assumption $\alpha \leqslant \lvert s - t^{\flat} \lvert$, the fact that $\alpha^2 + r^2 (s - t^{\flat})^2 \geqslant r_0^2 (s - t^{\flat})^2$ and the control on $\sigma_2$, there exists $C > 0$ such that
\begin{equation}
\label{eq:majr2}
\lvert \sigma \lvert \, \leqslant 2 C \lvert s - t^{\flat}\lvert \Big( \dfrac{1}{r_0^2} + \frac{1}{r_0} \Big) + C^2 \dfrac{ \lvert s - t^{\flat} \lvert^2}{r_0^2}.
\end{equation}

\noindent It follows that there exists a constant $\delta_0$ (that depends on \(M \) and \(r_0 \)) such that for \( \lvert s - t^{\flat} \lvert \, \leqslant \delta_0 \), we have \( \lvert \sigma \lvert \, \leqslant \dfrac{1}{2}.\) Using Lemma ~\ref{lem:estim1} of Appendix ~\ref{appendixB}, we obtain that for \( \lvert s - t^{\flat} \lvert \, \leqslant \delta_0 \)
\begin{equation*}
    \dfrac{1}{\lvert w(q_{\pm}(s)) \lvert} = \dfrac{1}{\sqrt{\alpha^2 + r^2 (s - t^{\flat})^2}} ( 1 + C \, \sigma).
\end{equation*}

\noindent In view of the expression of $\sigma$, we deduce that 
\begin{equation*}
\label{TEinvw}
    \dfrac{1}{\lvert w(q_{\pm}(s)) \lvert} = \dfrac{1}{\sqrt{\alpha^2 + r^2 (s - t^{\flat})^2}} + \sigma , \quad \text{ with } \lvert \sigma \lvert \, \leqslant C \dfrac{\lvert s-t^{\flat}\lvert^3}{(\alpha^2 + r^2 (s - t^{\flat})^2)^{\frac{3}{2}}} \text{ and } C > 0.  \tag{b} 
\end{equation*}

\noindent Using the Taylor-Lagrange formula again, we also have
\begin{align}
    & \nabla v(q_{\pm}(s)) = \nabla v(q^{\flat}) + \sigma_1, \tag{c} \label{TEgradv}\\
    & \transpose{\diff w(q_{\pm}(s))} w(q_{\pm}(t)) = \transpose{\diff w(q^{\flat})} \Big( \alpha \Vec{e}_{\theta}^{\perp} + r (s - t^{\flat})\Vec{e}_{\theta} \Big) + \sigma_2. \tag{d} \label{TEdw}
\end{align}

\noindent Since $\alpha^2 + r^2 (s - t^{\flat})^2 \geqslant r_0^2 (s - t^{\flat})^2$, we deduce
\begin{equation*}
\dfrac{\transpose{\diff w(q_{\pm}(s))} w(q_{\pm}(s))}{\lvert w(q_{\pm}(s)) \lvert} = \dfrac{\transpose{\diff w(q^{\flat})}}{\sqrt{\alpha^2 + r^2 (s - t^{\flat})^2}} \Big( \alpha \Vec{e}_{\theta}^{\perp} + r (s - t^{\flat}) \Vec{e}_{\theta} \Big) + \sigma_1.
\end{equation*}

\noindent Summing up, we obtain
\begin{equation*}
\dot{p}_{\pm}(s) = - \nabla v(q_{\pm}^{\flat}) \mp \dfrac{\transpose{\diff w(q^{\flat})}}{\sqrt{\alpha^2 + r^2 (s - t^{\flat})^2}} \Big( \alpha \Vec{e}_{\theta}^{\perp} + r (s - t^{\flat}) \Vec{e}_{\theta} \Big) + \sigma_1.
\end{equation*}
\noindent Integrating between $t$ and $t^{\flat}$ for $t$ satisfying $\lvert t - t^{\flat} \lvert \, \leqslant \delta_0$, we find that
\begin{multline*}
p_{\pm}(t) = p^{\flat} - \nabla v(q_{\pm}^{\flat}) (t - t^{\flat}) \mp \transpose{\diff w(q^{\flat})} \Vec{e}_{\theta}^{\perp} \times \alpha \ds\int_t^{t^{\flat}} \dfrac{\diff s}{\sqrt{\alpha^2 + r^2 (s - t^{\flat})^2}} \\ \mp r \transpose{\diff w(q^{\flat})} \Vec{e}_{\theta} \ds\int_t^{t^{\flat}} \dfrac{(s - t^{\flat})}{\sqrt{\alpha^2 + r^2 (s - t^{\flat})^2}} \diff s + \sigma_2.
\end{multline*}
\noindent But, using the change of variable $u = \dfrac{r}{\alpha} (s - t^{\flat})$ and computations of Appendix \ref{appendixB}, we have
\begin{align*}
    \alpha \ds\int_t^{t^{\flat}} \dfrac{\diff s}{\sqrt{\alpha^2 + r^2 (s - t^{\flat})^2}} & = \dfrac{\alpha}{r} \ds\int_0^{\frac{r}{\alpha}(t - t^{\flat})} \dfrac{\diff u}{\sqrt{1 + u^2}} = \dfrac{\alpha}{r} \Argsh{\dfrac{r}{\alpha}(t-t^{\flat})},\\
    \ds\int_t^{t^{\flat}} \dfrac{(s - t^{\flat})}{\sqrt{\alpha^2 + r^2 (s - t^{\flat})^2}} \diff s & = \dfrac{\alpha}{r^2} \ds\int_0^{\frac{r}{\alpha}(t - t^{\flat})} \dfrac{u}{\sqrt{1 + u^2}} \diff u = \dfrac{\alpha}{r^2} \Bigg[ \sqrt{1 + \frac{r^2 (t-t^{\flat})^2}{\alpha^2}} - 1 \Bigg].
\end{align*}

\noindent By simplifying the expressions and identifying the first terms with the Taylor expansion of $p_0$, we obtain the result for $p_{\pm}$. For $q_{\pm}$, we simply apply the computations from Appendix \ref{appendixB} with the same change of variable as before and integrate the Taylor expansion of $p_{\pm}$ between $t$ and $~t^{\flat}$.
\end{proof}

\begin{remark}
    Equation \eqref{eq:majr2} ensures the existence of \( \delta_0 \) and shows that \( \delta_0 \) decreases as \( r_0 \) decreases. However, it is important to note that \( \delta_0 \) does not depend on the semiclassical parameter \( \varepsilon \), unlike \( \delta \), which will be specified at the end of the paper as a function that tends to $0$ as \( \varepsilon \to 0 \). Consequently, it will always be possible to ensure that \( \delta \leqslant \delta_0 \) by choosing \( \varepsilon \) small enough and the controls of Proposition \ref{prop:DLpq} hold for $\lvert t - t^{\flat} \lvert \, \leqslant \delta$.
\end{remark}

\noindent Concerning the trajectories $(\widetilde{q}_{\pm}, \widetilde{p}_{\pm})$, a similar reasoning is used. The main difference is the condition at $t^{\flat}$. The expansions \eqref{TEw}, \eqref{TEgradv}, \eqref{TEinvw}, \eqref{TEdw} remain valid, relying on the fact that $\widetilde{p}^{\flat} = p^{\flat} + \sigma$ with $\lvert \sigma \lvert \, \leqslant \alpha$ and the assumption $\alpha \, \leqslant \lvert t - t^{\flat} \lvert$. After manipulating these expansions, we integrate between $t$ and $t^{\flat}$, using $\widetilde{z}^{\flat}$ as condition at time ~$t^{\flat}$. This implies the following proposition.

\begin{prop}
\label{prop:DLtildez} There exist $\delta_0, \,C > 0$ such that under the same assumptions as in Proposition \ref{prop:DLpq}, we have
    \begin{equation*}
    \left\lbrace 
    \begin{array}{cll}
        \widetilde{p}_{\pm}(t)& = p_{\pm}(t) + \delta_{p^{\flat}} + \sigma_{\widetilde{p}_{\pm}}, \quad \text{with } \lvert \sigma_{\widetilde{p}_{\pm}}\lvert \, \leqslant C \lvert t - t^{\flat}\lvert^2,\\
        \widetilde{q}_{\pm}(t) & = q_{\pm}(t) + \delta_{p^{\flat}}(t-t^{\flat}) + \sigma_{\widetilde{q}_{\pm}}, \quad \text{with } \lvert \sigma_{\widetilde{q}_{\pm}}\lvert \, \leqslant C \lvert t - t^{\flat}\lvert^3.
    \end{array}
    \right.
    \end{equation*}
\end{prop}

\subsection{The actions}
\label{subsec:studyAC}

\noindent The following proposition provides a comparison of the actions $S_{\pm}(t;t^{\flat},z^{\flat})$ and $\widetilde{S}_{\pm}(t;t^{\flat},\widetilde{z}^{\flat})$ defined in \eqref{Actions} and \eqref{Actionstilde} with $S_0(t;t^{\flat}, z^{\flat})$ defined in \eqref{actionzero}. 

\begin{prop} \label{prop:DLaction} There exist $\delta_0, \,C > 0$ such that under the same assumptions as in Proposition \ref{prop:DLpq}, we have
\begin{multline*}
S_{\pm}(t;t_0,z_0) = S_{\pm}^{\flat} + S_0(t;t^{\flat}, z^{\flat}) \pm \alpha (t - t^{\flat}) \mp (t - t^{\flat}) \sqrt{\alpha^2 + r^2 (t-t^{\flat})^2} \mp \dfrac{\alpha^2}{r} \Argsh{\dfrac{r}{\alpha}(t - t^{\flat})} + \, \sigma_{S_{\pm}},
\end{multline*}
\begin{equation*}
\widetilde{S}_{\pm}(t;t_0,z_0) = \widetilde{S}_{\pm}^{\flat} + S_{\pm}(t;t^{\flat},z^{\flat}) + \delta_{p^{\flat}} \cdot p^{\flat} (t - t^{\flat}) + \, \sigma_{\widetilde{S}_{\pm}}
\end{equation*}
where $S_{\pm}^{\flat} = S_{\pm}(t^{\flat};t_0,z_0)$, $\widetilde{S}_{\pm}^{\flat}$ is defined in \eqref{Actionstilde},
\begin{equation*}
S_0(t;t^{\flat},z^{\flat}) = \Big( \dfrac{\lvert p^{\flat} \lvert^2}{2} - v(q^{\flat}) \Big) (t - t^{\flat}) - p^{\flat} \cdot \nabla v(q^{\flat}) (t - t^{\flat})^2 + \, \sigma_{S_{0}}
\end{equation*}
and
\[ \lvert \sigma_{S_{\pm}}\lvert \, \leqslant C \lvert t - t^{\flat}\lvert^3, \quad \lvert \sigma_{\widetilde{S}_{\pm}}\lvert \, \leqslant C \lvert t - t^{\flat}\lvert^3 \quad \text{and} \quad  \lvert \sigma_{S_{0}}\lvert \, \leqslant C \lvert t - t^{\flat}\lvert^3. \]
\end{prop}

\begin{proof}[Proof of Proposition~\ref{prop:DLaction}]
Let $z_0$ be a point of $\mathcal{A}_{\pm}(M,\alpha_0,r_0,\delta,\mathrm{I})$. We consider \(t \in \mathrm{I} \) such that $t$ belongs to $[t^{\flat} + \alpha, t^{\flat} + \delta]$. We consider $s \in [t^{\flat} + \alpha, t ]$ (the proof is the same if $t \in [t^{\flat} - \delta, t^{\flat} - \alpha ]$ and $s \in [t , t^{\flat} - \alpha ]$). Knowing that the functions $t \in \mathrm{I} \mapsto h_{\pm}(z_{\pm}(t))$ and $t \mapsto h_0(z_0(t))$ are constants, we can write
    \begin{equation*}
        \dot{S}_{\pm}(s;t_0,z_0) = \lvert p_{\pm}(s) \lvert^2 - \, h_{\pm}(z^{\flat}) \quad \text{and} \quad \dot{S}_{0}(s;t_0,z_0) = \lvert p_0(s) \lvert^2 - \, h_{0}(z^{\flat}).
    \end{equation*}

\noindent According to Proposition \ref{prop:DLpq}, there exists $\delta_0$ such that for $\alpha \leqslant \lvert s - t^{\flat} \lvert \, \leqslant \lvert t - t^{\flat} \lvert \, \leqslant\delta_0 $, we have
\[ p_{\pm}(s) = p_0(s) \mp A \mp B + \sigma_{p_\pm}\]
with
\begin{itemize}
    \item $\lvert \sigma_{\pm} \lvert \leqslant C  \lvert t - t^{\flat}\lvert^2$ (according Proposition \ref{prop:DLpq}),
    \item $p_0(s) = p^{\flat} \, + \, \sigma_1$ (see Equation \eqref{DLp0q0}),
    \item $p^{\flat} \cdot A = \sqrt{\alpha^2 + r^2(s - t^{\flat})^2} - \alpha$ and $p^{\flat} \cdot B = 0$ (as $\diff w(q^{\flat})p^{\flat} = r \Vec{e}_{\theta}$ and $ \Vec{e}_{\theta} \cdot  \Vec{e}_{\theta}^{\perp} = 0$),
    \item $\lvert A \lvert , \lvert B \lvert \, \leqslant C \lvert t - t^{\flat}\lvert$ (using $r \geqslant r_0$, $\alpha \leqslant \lvert s - t^{\flat} \rvert$ and the fact that the function $u \mapsto \dfrac{\mathrm{Argsh}(u)}{u}$ is bounded).
\end{itemize}  

\noindent Then, thanks to the relation \( h_{\pm}(z^{\flat}) =  h_{0}(z^{\flat}) \pm \alpha \), computations give
\begin{equation*}
\lvert p_{\pm}(s) \lvert^2 = \lvert p_0(s) \lvert^2 - \, h_0(z^{\flat}) \mp \alpha \mp 2 \Big( \sqrt{\alpha^2 + r^2(s-t^{\flat})^2} - \alpha \Big) \, + \,  \sigma_1.
\end{equation*}
        
\noindent Therefore, according to the definition of $\dot{S}_0(s;t_0,z_0)$, we have
\vspace{-0.1cm}
\begin{equation*}
\lvert p_{\pm}(s) \lvert^2 = \dot{S}_0(s;t_0,z_0) \pm \alpha \mp 2 \sqrt{\alpha^2 + r^2(s-t^{\flat})^2} \, + \,  \sigma_1.
\end{equation*}

\noindent Thanks to the integrals calculated in Appendix \ref{appendixB}, integrating  between $t$ and $t^{\flat}$ for ~$t$ satisfying $\alpha \leqslant \lvert t - t^{\flat} \lvert \, \leqslant \delta_0$, we achieve the desired result for $S_{\pm}(t;t_0,z_0)$. Regarding $\widetilde{S}_{\pm}(t;t_0,z_0)$, the first thing to note is
\vspace{-0.1cm}
\begin{align*}
    \dot{\widetilde{S}}_{\pm}(s;t_0,z_0) & = \lvert p_{\pm}(s) + \delta_{p^{\flat}} \lvert^2 - \dfrac{\lvert p^{\flat} + \delta_{p^{\flat}} \lvert^2}{2} - v(q^{\flat}) \mp \alpha \\
    & = \lvert p_{\pm}(s) \lvert^2 + \, 2 p_{\pm}(s) \cdot \delta_{p^{\flat}} - \dfrac{\lvert p^{\flat} \lvert^2}{2} - p^{\flat} \cdot \delta_{p^{\flat}} - v(q^{\flat}) \mp \alpha + \sigma \quad \text{with } \lvert \sigma \lvert \, \leqslant C \alpha^2.
\end{align*}

\noindent By Proposition \ref{prop:DLpq}, we have that $p_{\pm}(s) = p^{\flat} + \sigma_1$ (with $\sigma_1 \, \leqslant C \lvert s-t^{\flat} \lvert$) and since $\lvert \delta_{p^{\flat}} \lvert \, \leqslant C \alpha $ with $\alpha \leqslant \lvert s-t^{\flat} \lvert$, we have
\vspace{-0.1cm}
\begin{equation*}
    2 p_{\pm}(s) \cdot \delta_{p^{\flat}} - p^{\flat} \cdot \delta_{p^{\flat}} = p^{\flat} \cdot \delta_{p^{\flat}} + \, \sigma_2.
\end{equation*}

\noindent Thus, we deduce that
\vspace{-0.1cm}
\begin{align*}
    \dot{\widetilde{S}}_{\pm}(s;t_0,z_0) & = \lvert p_{\pm}(s) \lvert^2 - \dfrac{\lvert p^{\flat} \lvert^2}{2} - v(q^{\flat}) \mp \alpha + p^{\flat} \cdot \delta_{p^{\flat}} + \, \sigma_2 \\
    & = \dot{S}_{\pm}(s;t_0,z_0) + p^{\flat} \cdot \delta_{p^{\flat}} + \, \sigma_2
\end{align*}

\noindent and it remains to integrate between $t$ and $t^{\flat}$ to obtain the desired result.
\end{proof}

\subsection{Adiabatic basis}
\label{subsec:studyVec}

\noindent This subsection provides all the necessary details for studying the vectors that guide the wave packets considered in this article. We follow an approach similar to that in ~\cite{FGH2021}, where the case $\alpha = 0$ was established. For clarity and to distinguish our analysis from that of the case $\alpha = 0$, we will denote the eigenprojectors considered in our study by $\Pi_{\pm,\alpha}$ (introduced in \eqref{projspectraux}) and those used in \cite{FGH2021} by $\Pi_{\pm,0}$. Later in this article, we will be able to compare these two projectors by making assumptions on $\alpha$ and using a rescaling argument.\\

\noindent The following proposition constructs the direction of the wave packets before time $t^{\flat}$.

\begin{prop}
\label{prop:Vec}
Let us consider the normalized \review{vector} $\Vec{Y}_{-}$ solution of \eqref{eqVec} for $t \in [t_0, t^{\flat}]$ where
\begin{equation*}
    \Vec{Y}_{-}(t_0) = \Vec{Y}_0 \quad \text{with} \quad V(q_0) \Vec{Y}_0 = \lambda_-(q_0) \Vec{Y}_0.
\end{equation*}
 This vector satisfies
\begin{enumerate}[leftmargin=*, labelindent=0pt]
    \item For all $t \in [t_0, t^{\flat}]$, $\| \Vec{Y}_{-}(t) \| = \| \Vec{Y}_0 \| = 1$.
    \item For all $t \in [t_0, t^{\flat}]$, $\Pi_{-,\alpha}(q_{-}(t)) \Vec{Y}_{-}(t) = \Vec{Y}_{-}(t)$. \label{item1}
    \item There exist $\tau > 0$ and \review{a function} $x \mapsto \Vec{\mathcal{V}}_{-}(x)$ smooth in a neighborhood of $(\Phi_{-}^{t,t^{\flat}}(z^{\flat}))_{t \in [t^{\flat} - \tau, t^{\flat})}$, such that $\Pi_{-} \Vec{\mathcal{V}}_{-} = \Vec{\mathcal{V}}_{-}$ and for all $t \in [t^{\flat} - \tau, t^{\flat})$, $\Vec{Y}_{-}(t) = \vec{\mathcal{V}}_{-}(q_{-}(t))$.
\end{enumerate}
\end{prop}

\begin{proof}[Proof of Proposition ~\ref{prop:Vec}]
\noindent Since the case $\alpha = 0$ is treated in \cite{FGH2021}, we can assume that $\alpha \neq 0$. Then, in this setting, the vector \(\Vec{Y}_{-}\) exists on \([t_0, t^{\flat}]\) by applying the Cauchy--Lipschitz theorem (there are no regularity issues at $t^{\flat}$). The conservation of the norm is obtained by using the fact that the matrix-valued functions $B_{\pm}$, defined in \eqref{matB}, satisfy \[ B_{-} = - \transpose{B_{+}} \quad \text{and} \quad B_{+} \Vec{Y}_{-} = 0. \] Moreover, the assertion that $\vec{Y}_{-}$ is an eigenvector, as well as the construction of the vector $\vec{\mathcal{V}}_{-}$, can be carried out in the same manner as in \cite{FGH2021}, since the initial conditions at time $t_0$ is an eigenvector.
\end{proof}

\noindent The result is the same for $\Vec{Y}_+$, solution of \eqref{eqVec}, if the initial vector at time $t_0$ is normalized and associated with the plus-mode, meaning that 
\begin{equation}
\label{eq:ICvecplus}
    \Vec{Y}_{+}(t_0) = \Vec{Y}_1 \quad \text{with} \quad V(q_0) \Vec{Y}_1= \lambda_+(q_0) \Vec{Y}_1.
\end{equation}

\noindent In Subsection \ref{subsec:newscale}, we will compare the eigenprojectors $\Pi_{\pm,\alpha}$ with the eigenprojector $\Pi_{\pm,0}$ close to ~$t^{\flat}$. The expansion of the projector $\Pi_{\pm,0}$ at time $t^{\flat}$ is obtained in the proof of ~\cite[Proposition $1.5$]{FGH2021} and we recall here that it is given by
\begin{equation}
\label{rq:devpi0} 
\Pi_{\pm,0}(q_{\pm}(t)) = \dfrac{1}{2}\Big( I_2 \,\pm\, \operatorname{sgn}(t - t^{\flat})\, A(\vec e_{\theta}) \Big)+ \mathcal{O}(|t - t^{\flat}|).
\end{equation}
The following proposition, derived from the Taylor expansions \eqref{TEw} and \eqref{TEinvw} proved in Proposition~\ref{prop:DLpq}, provides the Taylor expansion of $\Pi_{\pm,\alpha}$ near $t^{\flat}$.

\begin{prop}
\label{prop:DLPi}
There exist $\delta_0, \,C > 0$ such that under the same assumptions as in Proposition ~\ref{prop:DLpq}, we have
\begin{equation*}
    \Pi_{\pm,\alpha}(q_{\pm}(t)) =  \frac{1}{2} \Bigg( \mathrm{I}_2 \, \pm \, \dfrac{r(t-t^\flat)}{\sqrt{\alpha^2 + r^2(t-t^{\flat})^2}} A(\vec{e}_{\theta}) \, \pm \, \dfrac{\alpha}{\sqrt{\alpha^2 + r^2(t-t^{\flat})^2}} A(\vec{e}_{\theta}^{\perp}) \Bigg) + \, \sigma_{_{\Pi_{\pm,\alpha}}}
\end{equation*}
with \( \lvert \sigma_{_{\Pi_{\pm,\alpha}}} \lvert \, \leqslant C \lvert t - t^{\flat} \lvert.\)
\end{prop}

\noindent In order to construct the outgoing vectors and obtain a result similar to that stated in Proposition ~\ref{prop:Vec} for \( t > t^{\flat} \), we must fix the initial conditions for Equations \eqref{eqVec} and \eqref{eqVectilde} at time \( t^{\flat} \). These vectors must be eigenvectors of $V(q^{\flat})$. To do that, we consider the following eigenprojectors \[ \Pi_{\pm,\alpha}^{\flat} := \Pi_{\pm,\alpha}(q_{\pm}(t^{\flat})).\] Unlike the projector $\Pi_{\pm,\alpha}(q_{\pm})$, the projector $\Pi_{\pm,0}$ is not defined at time $t^{\flat}$. Therefore, we must consider the limits as $t$ approaches $t^{\flat}$. Thanks to Equation \eqref{rq:devpi0}, we define
\begin{equation*}
    \Pi_{\pm,0}^{\flat,-} := \lim\limits_{t \to t^{\flat}, \, t < t^{\flat}} \Pi_{\pm,0}(q_{\pm}(t)) = \dfrac{1}{2}\Big( I_2 \,\mp\, A(\vec e_{\theta}) \Big) \quad \text{and} \quad \Pi_{\pm,0}^{\flat,+} := \lim\limits_{t \to t^{\flat}, \, t > t^{\flat}} \Pi_{\pm,0}(q_{\pm}(t)) = \dfrac{1}{2}\Big( I_2 \,\pm\, A(\vec e_{\theta}) \Big).
\end{equation*}
Then, we set
\begin{equation}
\label{eq:ICvecafter}
    \Vec{Y}^{\flat} = \Pi_{-,0}^{-,\flat} \Bigg( \lim\limits_{t \to t^{\flat}, \, t < t^{\flat}} \vec{Y}_{-}(t) \Bigg) \quad \text{and} \quad \Vec{Y}^{\flat\perp} = \begin{pmatrix}
            0 & -1 \\
            1 & 0
        \end{pmatrix} \Vec{Y}^{\flat}.
\end{equation}

\begin{remark}
    By construction, the vector $\Vec{Y}^{\flat}$ is a normalized eigenvector of the matrix $A(\vec e_{\theta})$ associated with the eigenvalue $1$. Since $\Vec{Y}^{\flat\perp}$ is the $\frac{\pi}{2}$–rotation of $\Vec{Y}^{\flat}$, it is a normalized eigenvector of the matrix $A(\vec e_{\theta})$ associated with the eigenvalue $-1$.
\end{remark}

\noindent The following proposition \review{summarizes} a similar result to that stated in Proposition \ref{prop:Vec} for \( t > t^{\flat} \). The initial conditions are well chosen at \( t^{\flat}\), using the projection of the vectors \( \Vec{Y}^{\flat\perp} \) and \( \Vec{Y}^{\flat} \) on the minus and the plus eigenspaces of $V(q^{\flat})$, thanks to the eigenprojectors $\Pi_{\pm,\alpha}^{\flat}$.

\begin{prop}
\label{prop:Vecaftercrossing}
Let us consider $T = \text{max}(\mathrm{I})$ and the normalized vectors $\Vec{Y}_{-}$ and $\Vec{\widetilde{Y}_{+}}$ solutions of 
\begin{equation*}
    \left\lbrace 
    \begin{array}{cll}
        \partial_t  \Vec{Y}_{-} & = & B_{-}(q_{-},p_{-}) \Vec{Y}_{-}\\
        \Vec{Y}_{-}(t^{\flat}) & = & \Pi_{-,\alpha}^{\flat} \Vec{Y}^{\flat\perp}
    \end{array}
    \right. \quad \text{and} \quad \left\lbrace 
    \begin{array}{cll}
        \partial_t  \Vec{\widetilde{Y}_{+}} & = & B_{+}(\widetilde{q}_{+},\widetilde{p}_{+}) \Vec{\widetilde{Y}_{+}}\\
        \Vec{\widetilde{Y}_{+}}(t^{\flat}) & = & \Pi_{+,\alpha}^{\flat} \Vec{Y}_{\flat}
    \end{array}
    \right.
\end{equation*}
on $[t^{\flat},T]$. These vectors satisfy
\begin{enumerate}[leftmargin=*, labelindent=0pt]
    \item For all $t \in [t^{\flat},T]$, $\| \Vec{Y}_{-}(t) \| = \| \Vec{Y}_{-}(t^{\flat}) \|$ and $\| \Vec{\widetilde{Y}}_{+}(t) \| = \| \Vec{\widetilde{Y}}_{+}(t^{\flat}) \|$.
    \item For all $t \in [t^{\flat},T]$, $\Pi_{-,\alpha}(q_{-}(t)) \Vec{Y}_{-}(t) = \Vec{Y}_{-}(t)$ and $\Pi_{+,\alpha}(\widetilde{q}_{+}(t)) \Vec{\widetilde{Y}}_{+}(t) = \Vec{\widetilde{Y}}_{+}(t)$. \label{item1after}
    \item There exist $\tau > 0$ and two functions $x \mapsto \vec{\mathcal{V}}_{\pm}(x)$ smooth in a neighborhood of $(\Phi_{\pm}^{t,t^{\flat}}(z^{\flat}))_{t \in (t^{\flat}, t^{\flat} + \tau]}$, such that $\Pi_{\pm} \vec{\mathcal{V}}_{\pm} = \vec{\mathcal{V}}_{\pm}$ and for all $t \in (t^{\flat}, t^{\flat} + \tau]$, $\Vec{Y}_{-}(t) = \vec{\mathcal{V}}_{-}(q_{-}(t))$ and $\Vec{\widetilde{Y}}_{+}(t) = \vec{\mathcal{V}}_{+}(\widetilde{q}_{+}(t))$.
\end{enumerate}
\end{prop}

\begin{proof}[Proof of Proposition~\ref{prop:Vecaftercrossing}]
    The proof uses the same argument as in the proof of Proposition ~\ref{prop:Vec} since by definition $\Pi_{-,\alpha}^{\flat} \Vec{Y}^{\flat\perp}$ and $\Pi_{+,\alpha}^{\flat} \Vec{Y}_{\flat}$ are eigenvectors of $V(q^{\flat})$ associated with the respective eigenvalues $\lambda_{-}(q^{\flat})$ and $\lambda_{+}(q^{\flat})$.
\end{proof}

\noindent Of course, if the initial condition at time $t_0$ is associated with the plus-mode, as stated in Equation ~\eqref{eq:ICvecplus}, we can prove a similar result as long as the initial condition at time $t^{\flat}$ is well chosen. This is detailed in Appendix ~\ref{appendix:autremode}.

\begin{remark}
    The initial data introduced in Proposition ~\ref{prop:Vecaftercrossing} are normalized up to a remainder that tends to zero as $\varepsilon$ tends to zero. This is a consequence of Proposition ~\ref{lem:veclimit}, which is stated in Subsection ~\ref{subsec:newscale} below.
\end{remark}

%%%%%%%%%%%%%%%%%%%%%%%%%%%%%%%%%%%%%%%%%%%%%%%%%%%%%%%%%%%

\section{Profiles of the approximate solution}
\label{sec:profiles}

\noindent We remain in the same setting as in Section~\ref{sec:analysisCQ}. We fix \((M, \alpha_0,r_0,\delta) \in \R_{+}^{*} \times \R_{+} \times \R_{+}^{*} \times (0, 1] \) and \( \mathrm{I} \subset \mathbb{R} \) a compact interval with \( t_0 = \min(\mathrm{I}) \). Then, we consider an initial point $z_0 \in \mathcal{A}_{\pm}(M,\alpha_0,r_0,\delta,\mathrm{I})$ and we are interested in the two Schrödinger equations with a time-dependent harmonic potential, defined in \eqref{eqProfil} and \eqref{eqProfiltilde}, which we recall

\begin{equation*}
i \partial_t u_{\pm} = - \dfrac{1}{2} \Delta u_{\pm} + \dfrac{1}{2} \text{Hess} \, \lambda_{\pm}(q_{\pm}) y \cdot y u_{\pm} \quad \text{and} \quad i \partial_t \widetilde{v}_{\pm} = - \dfrac{1}{2} \Delta \widetilde{v}_{\pm} + \dfrac{1}{2} \text{Hess} \, \lambda_{\pm}(\widetilde{q}_{\pm}) y \cdot y \widetilde{v}_{\pm}.
\end{equation*}

\noindent These equations will be coupled with an initial condition in $\schwartz{\R^d}$ either at time $t_0$ or $t^{\flat} + \delta$. For $\alpha \neq 0$, we observe that the functions $t \mapsto \text{Hess} \, \lambda_{\pm}(q_{\pm}(t))$ and $t \mapsto \text{Hess} \, \lambda_{\pm}(\widetilde{q}_{\pm}(t))$ are smooth. Let us introduce the operators $$Q_{\pm}(t) = -\dfrac{1}{2} \Delta + \dfrac{1}{2} \text{Hess} \, \lambda_{\pm}(q_{\pm}(t)) y \cdot y \quad \text{and} \quad \widetilde{Q}_{\pm}(t) = -\dfrac{1}{2} \Delta + \dfrac{1}{2} \text{Hess} \, \lambda_{\pm}(\widetilde{q}_{\pm}(t)) y \cdot y.$$
\noindent Noticing that the classical symbols of the operators $Q_{\pm}$ and $\widetilde{Q}_{\pm}$ satisfy subquadratic estimates, this leads to the existence of solutions according to \cite{MasperoRobert17}. Since the potentials of Equations ~\eqref{eqProfil} and ~\eqref{eqProfiltilde} are real-valued, the $L^2$-norm of the solutions is conserved. This is summarized in the following proposition.

\begin{prop}
\label{prop:existenceu}
For all $z_0 \in \mathcal{A}_{\pm}(M,\alpha_0,r_0, \delta, \mathrm{I})$, there exist solutions $u_{\pm}$ and $\widetilde{u}_{\pm}$ to \eqref{eqProfil} and \eqref{eqProfiltilde} on ~$\mathrm{I}$ with $\varphi_{\pm} \in \schwartz{\R^d}$ and $\widetilde{\varphi}_{\pm} \in \schwartz{\R^d}$ as initial conditions at a time $t_{\text{init}} \in \mathrm{I}$. For any $t \in \mathrm{I}$, $u_{\pm}(t)$ and $\widetilde{u}_{\pm}(t)$ belong to $\schwartz{\R^d}$ and $$\normLp{u_{\pm}(t)}{2}{\R^d} = \normLp{\varphi_\pm}{2}{\R^d}, \quad \normLp{\widetilde{u}_{\pm}(t)}{2}{\R^d} = \normLp{\widetilde{\varphi}_\pm}{2}{\R^d}.$$
\end{prop}

\subsection{Analysis of the Hessian of the eigenvalues along the trajectories}
\label{subsec:hessian}

\noindent We now focus on the asymptotic behavior of the Hessian of the eigenvalues $\lambda_{\pm}$ along the classical trajectories when ~$t$ is close to ~$t^{\flat}$ to study the equations \eqref{eqProfil} and \eqref{eqProfiltilde} close to the crossing.\\

\noindent For all $t \in \R$, we introduce $\Gamma_\alpha(t)$ the following matrix
\begin{equation}
\label{def:Gammaalpha}
\Gamma_\alpha (t) = \dfrac{1}{r} \Bigg( \partial_i w(q^{\flat}) \cdot \partial_j w(q^{\flat}) - \dfrac{r^2(t-t^{\flat})^2(\partial_i w(q^{\flat}) \cdot \Vec{e}_{\theta})(\partial_j w(q^{\flat}) \cdot \Vec{e}_{\theta})}{\alpha^2 + r^2(t-t^{\flat})^2} \Bigg)_{1 \leqslant i,j \leqslant d}.
\end{equation}

\noindent We start by analyzing the asymptotics of $\text{Hess} \, \lambda_{\pm}(q_{\pm})$ close to $t^{\flat}$.

\begin{prop}
\label{prop:DLHesslambda} For all $M > 0$ and $r_0 > 0$, there exist $\delta_0$ and a constant $C > 0$ such that the following holds. We consider $(\alpha_0, \delta) \in \R_+ \times \R_+^*$, $\mathrm{I} \subset \R$ a compact interval with $t_0 = \text{min}(\mathrm{I})$ and $z_0 \in \mathcal{A}_{\pm}(M,\alpha_0,r_0,\delta,\mathrm{I})$. For all $t \in \mathrm{I}$ satisfying $\alpha \leqslant | t - t^{\flat} | \leqslant \delta_0$, there exist two smooth matrix-valued functions $m_{\pm}^{\alpha}(t)$ such that
\vspace{-0.4cm}
\begin{equation*}
\text{Hess} \, \lambda_{\pm}(q_{\pm}(t)) = \underbrace{m_{\pm}^{\alpha}(t) + \sigma(t)}_{= M_{\pm}^{\alpha}(t)} \pm \dfrac{r}{\sqrt{\alpha^2 + r^2 (t - t^{\flat})^2}} \Gamma_{\alpha}(t), \quad \text{with } |\sigma(t)| \leqslant \dfrac{C |t - t^{\flat}|^3}{\left( \alpha^2 + r^2 (t - t^{\flat})^2 \right)^{3/2}}.
\end{equation*}
\noindent Moreover, for $(i,j) \in [\![ 1 , d ]\!]^2$, the $(i,j)$-coefficient of the matrix $m_{\pm}^{\alpha}(t)$ is given by
\vspace{-0.4cm}
\begin{multline*}
(m_{\pm}^{\alpha}(t))_{i,j} = \partial_{i,j}^2 v (q^{\flat}) \pm \dfrac{\partial_{i,j}^2 w(q^{\flat}) \cdot \left( \alpha \vec{e}_{\theta}^{\perp} + r(t - t^{\flat}) \vec{e}_{\theta} \right)}{\sqrt{\alpha^2 + r^2 (t - t^{\flat})^2}} \mp \dfrac{1}{\left( \alpha^2 + r^2 (t - t^{\flat})^2 \right)^{3/2}} \times \\ \Bigg[ 
\alpha^2 (\partial_i w(q^{\flat}) \cdot \vec{e}_{\theta}^{\perp})(\partial_j w(q^{\flat}) \cdot \vec{e}_{\theta}^{\perp}) + \, r \alpha (t - t^{\flat}) \left( (\partial_i w(q^{\flat}) \cdot \vec{e}_{\theta}^{\perp})(\partial_j w(q^{\flat}) \cdot \vec{e}_{\theta}) + (\partial_i w(q^{\flat}) \cdot \vec{e}_{\theta})(\partial_j w(q^{\flat}) \cdot \vec{e}_{\theta}^{\perp}) \right) 
\Bigg].
\end{multline*}
\end{prop}

\begin{remark}
The matrix $\Gamma_0 \in \R^{d \times d}$ given by \(\Gamma_0 = \dfrac{1}{r} \transpose{\diff w(q^{\flat})} (I_2 - \Vec{e}_{\theta} \otimes \Vec{e}_{\theta}) \diff w(q^{\flat}) \) can be rewritten in terms of its coefficients as follows
\begin{equation*}
\Gamma_0 = \dfrac{1}{r} \Big( \partial_i w(q^{\flat}) \cdot \partial_j w(q^{\flat}) - (\partial_i w(q^{\flat}) \cdot \Vec{e}_{\theta})(\partial_j w(q^{\flat}) \cdot \Vec{e}_{\theta} \Big)_{1 \leqslant i,j \leqslant d}
\end{equation*}
\noindent and it is the one used in \cite{FGH2021}. Noticing that $\Gamma_{\alpha}(t) = \Gamma_0 + \dfrac{\alpha^2}{\alpha^2 + r^2 (t-t^{\flat})^2} \Gamma_1$ for all $t \in \R$, with $\Gamma_1$ the matrix given in \eqref{matGamma1}, we obtain the same result as \cite[Lemma $2.4$]{FGH2021} in the case $\alpha = 0$.
\end{remark}

\begin{proof}[Proof of Proposition~\ref{prop:DLHesslambda}]
\noindent By definition, $\lambda_{\pm} = v \pm \lvert w \lvert$. Then, for all $(i,j) \in [\![1 , d ]\!]^2$
\begin{equation*}
\partial_{i,j}^2 \lambda_{\pm} = \underset{T_1}{\underbrace{\partial_{i,j}^2 v}} \mp \partial_{i,j}^2 (\lvert w \lvert).
\end{equation*}

\noindent But, computations give

\begin{equation*}
\partial_{i,j}^2 (\lvert w \lvert) = \underset{T_2}{\underbrace{\dfrac{\partial_i w \cdot \partial_j w}{\lvert w \lvert}}} + \underset{T_3}{\underbrace{\partial_{i,j}^2 w \cdot \dfrac{w}{\lvert w \lvert}}} - \underset{T_4}{\underbrace{\dfrac{(\partial_i w \cdot w)(\partial_j w \cdot w)}{\lvert w \lvert^3}}}.
\end{equation*}

\noindent Let $z_0$ be a point of $\mathcal{A}_{\pm}(M,\alpha_0,r_0,\delta,\mathrm{I})$. We consider \(t \in \mathrm{I} \) such that $t$ belongs to $[t^{\flat} + \alpha, t^{\flat} + \delta]$. We consider $s \in [t^{\flat} + \alpha, t ]$ (the proof is the same if $t \in [t^{\flat} - \delta, t^{\flat} - \alpha ]$ and $s \in [t , t^{\flat} - \alpha ]$). Using the expansions \eqref{TEw} and \eqref{TEinvw} stated in the proof of Proposition \ref{prop:DLpq}, we obtain that there exists $\delta_0 > 0$ such that the following expansions hold for $\alpha \, \leqslant \lvert t - t^{\flat} \lvert \, \leqslant \delta_0$

\begin{align*}
    T_1 = & \, \, \partial_{i,j}^2 v (q^{\flat}) + \, \sigma_1 \quad \text{with } \lvert \sigma_1 \lvert \, \leqslant C_1 \lvert t-t^{\flat} \lvert \text{ and } C_1 > 0,\\
    T_2 = & \, \, \dfrac{\partial_i w(q^{\flat}) \cdot \partial_j w(q^{\flat})}{\sqrt{\alpha^2 + r^2 (t-t^{\flat})^2}} + \, \sigma_2 \quad \text{with } \lvert \sigma_1 \lvert \, \leqslant \frac{ C_2 \lvert t - t^{\flat} \lvert^3}{\Big( \alpha^2 + r^2 (t-t^{\flat})^2 \Big)^{\frac{3}{2}}} \text{ and } C_2 > 0, \\
    T_3 = & \, \, \dfrac{\partial_{i,j}^2 w(q^{\flat}) \Big( \alpha \Vec{e}_{\theta}^{\perp} + r(t - t^{\flat}) \Vec{e}_{\theta} \Big)}{\sqrt{\alpha^2 + r^2 (t-t^{\flat})^2}} + \, \sigma_3 \quad \text{with } \lvert \sigma_3 \lvert \, \leqslant C_3 \lvert t-t^{\flat} \lvert \text{ and } C_3 > 0,\\
    T_4 =& \, \, \dfrac{\alpha^2 (\partial_i w(q^{\flat}) \cdot \Vec{e}_{\theta}^{\perp})(\partial_j w(q^{\flat}) \cdot \Vec{e}_{\theta}^{\perp}) + r^2 (t-t^{\flat})^2 (\partial_i w(q^{\flat}) \cdot \Vec{e}_{\theta})(\partial_j w(q^{\flat}) \cdot \Vec{e}_{\theta})}{\Big( \alpha^2 + r^2 (t-t^{\flat})^2 \Big)^{\frac{3}{2}}}\\ 
    & + \dfrac{r \alpha (t-t^{\flat}) (\partial_i w(q^{\flat}) \cdot \Vec{e}_{\theta}^{\perp})(\partial_j w(q^{\flat}) \cdot \Vec{e}_{\theta}) + r \alpha (t-t^{\flat}) (\partial_i w(q^{\flat}) \cdot \Vec{e}_{\theta})(\partial_j w(q^{\flat}) \cdot \Vec{e}_{\theta}^{\perp})}{\Big( \alpha^2 + r^2 (t-t^{\flat})^2 \Big)^{\frac{3}{2}}}\\
    &  + \, \sigma_4 \quad \text{with } \lvert \sigma_4 \lvert \, \leqslant \frac{ C_4 \lvert t - t^{\flat} \lvert^5}{\Big( \alpha^2 + r^2 (t-t^{\flat})^2 \Big)^{\frac{5}{2}}} \text{ and } C_4 > 0.
\end{align*}

\noindent Summing up and using the definition \eqref{def:Gammaalpha} of the matrix $\Gamma_{\alpha}$, we deduce that
\begin{align*}
    \partial_{i,j}^2 (\lvert w \lvert) & = \dfrac{r}{\sqrt{\alpha^2 + r^2 (t-t^{\flat})^2}} \Gamma_{\alpha}(t) - \dfrac{1}{\Big( \alpha^2 + r^2 (t-t^{\flat})^2 \Big)^{\frac{3}{2}}} \Bigg[ \alpha^2 (\partial_i w(q^{\flat}) \cdot \Vec{e}_{\theta}^{\perp})(\partial_j w(q^{\flat}) \cdot \Vec{e}_{\theta}^{\perp}) \\ 
    & + r \alpha (t-t^{\flat}) (\partial_i w(q^{\flat}) \cdot \Vec{e}_{\theta}^{\perp})(\partial_j w(q^{\flat}) \cdot \Vec{e}_{\theta}) + r \alpha (t-t^{\flat}) (\partial_i w(q^{\flat}) \cdot \Vec{e}_{\theta})(\partial_j w(q^{\flat}) \cdot \Vec{e}_{\theta}^{\perp}) \Bigg] \\
    & + \dfrac{\partial_{i,j}^2 w(q^{\flat}) \Big( \alpha \Vec{e}_{\theta}^{\perp} + r(t - t^{\flat}) \Vec{e}_{\theta} \Big)}{\sqrt{\alpha^2 + r^2 (t-t^{\flat})^2}} + \, \sigma_5 \quad \text{with } \lvert \sigma_5 \lvert \, \leqslant \frac{ C_5 \lvert t - t^{\flat} \lvert^3}{\Big( \alpha^2 + r^2 (t-t^{\flat})^2 \Big)^{\frac{3}{2}}} \text{ and } C_5 > 0.
\end{align*}

\noindent Consequently, we obtain the desired expansion stated in Proposition \ref{prop:DLHesslambda}.
\end{proof}

\begin{remark}
Under the assumptions of Proposition \ref{prop:DLHesslambda}, since $\lvert \delta_{p^{\flat}} \lvert \, \leqslant C \alpha$ and $\alpha \leqslant \lvert t - t^{\flat} \lvert$, we have
\begin{align*}
w(\widetilde{q}_{\pm}(t)) & = \alpha \Vec{e}_{\theta}^{\perp} + r \Vec{e}_{\theta} (t - t^{\flat}) + \sigma \quad \text{with }\lvert \sigma \lvert \, \leqslant C \Big( \lvert \delta_{p^{\flat}} \lvert \lvert t-t^{\flat} \lvert + \lvert t-t^{\flat} \lvert^2 \Big) \\
& = w(q_{\pm}(t)) + \widetilde{\sigma}, \quad \text{with }\lvert \widetilde{\sigma} \lvert \, \leqslant \widetilde{C} \lvert t-t^{\flat} \lvert^2.
\end{align*}
\noindent Thus, Proposition \ref{prop:DLHesslambda} also holds for $\text{Hess} \, \lambda_{\pm}(\widetilde{q}_{\pm})$.
\end{remark}

\noindent For $\alpha \neq 0$, the two terms that appear in the expansions of $\text{Hess} \, \lambda_{\pm}(q_{\pm})$ close to $t^{\flat}$ are smooth and hence integrable. The main difference between them is that the integral from $t_0$ to $t^{\flat}$ of $M_{\pm}^{\alpha}$ is uniformly bounded with respect to $\alpha$, while the integral of $t \in \mathrm{I} \mapsto \frac{r}{\sqrt{\alpha^2 + r^2 (t- t^{\flat})^2}} \Gamma_{\alpha}(t)$ over the same interval depends on $\alpha$. We start to study the term $M_{\pm}^{\alpha}$. The objective of the subsequent proposition is to examine the independence of the norm and the integral of $M_{\pm}^{\alpha}$ with respect to $\alpha$.

\begin{prop}[Study of $M_{\pm}^{\alpha}$]
\label{prop:Malpha} Under the same assumptions as in Proposition \ref{prop:DLHesslambda}, the matrix-valued function \( M_{\pm}^{\alpha} \), as defined in Proposition \ref{prop:DLHesslambda}, satisfies the following properties.
\begin{enumerate}
	\item Let $\tau$ be a positive real number such that $[t^{\flat}, t^{\flat} + \tau] \subset \mathrm{I}$. $M_{\pm}^{\alpha}$ is integrable on $[t_0, t^{\flat}]$ (resp. on $[ t^{\flat} , t^{\flat} + \tau]$) and there exists $C > 0$ (independent of $\alpha$) such that $$\Bigg\lvert \ds\int_{t_0}^{t^{\flat}} M_{\pm}^{\alpha}(s) \dint s \, \Bigg\lvert \leqslant C, \, \, \, \, \Bigg( \text{resp.} \, \, \Bigg\lvert \ds\int_{t^{\flat}}^{t^{\flat} + \tau} M_{\pm}^{\alpha}(s) \dint s \Bigg\lvert \leqslant C \Bigg).$$
	\item There exists a constant $C > 0$ (independent of $\alpha$) such that for all $t \in \mathrm{I}$, $$\Bigg\lvert 1 + \Big\| M^{\alpha}_{\pm}(t) \Big\|_{\R^{d \times d}} \Bigg\lvert \leqslant \dfrac{C}{\sqrt{\alpha^2 + r^2(t-t^{\flat})^2}}.$$
\end{enumerate}
\end{prop}

\begin{proof}[Proof of Proposition~\ref{prop:Malpha}] Let $z_0$ be a point of $\mathcal{A}_{\pm}(M,\alpha_0,r_0,\delta,\mathrm{I})$. According to Proposition \ref{prop:DLHesslambda}, the functions \( t \mapsto M_{\pm}^{\alpha}(t) \) exist and we will \review{study} them on the time interval \( \mathrm{I} \).
\begin{enumerate}[leftmargin=*, labelindent=0pt] 
    \item We consider $\tau$ a positive real number such that $[t^{\flat}, t^{\flat} + \tau] \subset \mathrm{I}$. The functions $t \in \mathrm{I} \mapsto M_{\pm}^{\alpha}(t)$ are continuous on $[t_0, 	t^{\flat}]$ (resp. on $[ t^{\flat} , t^{\flat} + \tau]$) and therefore integrable. Now, we have to look at the $\alpha$-dependence of the integral. Here, we write the proof for the integral $\ds\int_{t_0}^{t^{\flat}} M_{\pm}^{\alpha}(s) \dint s$ but the other one is done in the same way. In view of the expression of $M_{\pm}^{\alpha}$ given in Proposition ~\ref{prop:DLHesslambda}, we have to look at the following six integrals
\begin{align*}
I_0 = \ds\int_{t_0}^{t^{\flat}} \dint s \quad & , \quad I_1 = \ds\int_{t_0}^{t^{\flat}} \dfrac{\alpha}{\sqrt{\alpha^2 + r^2(s-t^{\flat})^2}} \dint s , \\ I_2  = \ds\int_{t_0}^{t^{\flat}} \dfrac{(s - t^{\flat})}{\sqrt{\alpha^2 + r^2(s-t^{\flat})^2}} \dint s \quad & , \quad I_3  = \ds\int_{t_0}^{t^{\flat}} \dfrac{\alpha^2}{\Big(\alpha^2 + r^2(s-t^{\flat})^2\Big)^{\frac{3}{2}}} \dint s , \\
I_4  = \ds\int_{t_0}^{t^{\flat}} \dfrac{\alpha(s-t^{\flat})}{\Big(\alpha^2 + r^2(s-t^{\flat})^2\Big)^{\frac{3}{2}}} \dint s \quad & , \quad I_5 = \ds\int_{t_0}^{t^{\flat}} \dfrac{(s-t^{\flat})^3}{\Big(\alpha^2 + r^2(s-t^{\flat})^2\Big)^{\frac{3}{2}}} \dint s.
\end{align*}

\noindent For $I_0$, $I_1$, $I_2$ and $I_5$, it is sufficient to observe that the integrands are uniformly bounded with respect to $\alpha$ using the fact that $x \mapsto \sqrt{x}$ and $x \mapsto x^{\frac{3}{2}}$ are increasing and that for all $t \in [t_0,t^{\flat}]$ $$\alpha^2 + r^2(s-t^{\flat})^2 \geqslant r_0^2(s-t^{\flat})^2 \quad \text{and} \quad \alpha^2 + r^2(s-t^{\flat})^2 \geqslant \alpha^2.$$

\noindent It remains to examine the case of the integrals $I_3$ and $I_4$. If $\alpha = 0$, the result follows immediately. Otherwise, by using the change of variable $u = \frac{r(s-t^{\flat})}{\alpha}$ and the computations of Appendix \ref{appendixB}, we obtain
\begin{align*}
    & I_3 = \dfrac{1}{r} \ds\int_{\frac{r}{\alpha}(t_0 - t^{\flat})}^{0} \dfrac{\diff u}{(1 + u^2)^{\frac{3}{2}}} = - \dfrac{t_0 - t^{\flat}}{\sqrt{\alpha^2 + r^2 (t_0-t^{\flat})^2}}\\
    & I_4 = \dfrac{1}{r^2} \ds\int_{\frac{r}{\alpha}(t_0 - t^{\flat})}^{0} \dfrac{u}{(1 + u^2)^{\frac{3}{2}}} \, \diff u = \dfrac{1}{r^2} \Bigg( \dfrac{\alpha}{\sqrt{\alpha^2 + r^2(t_0-t^{\flat})^2}} - 1 \Bigg).
\end{align*}

\noindent As $\alpha \leqslant 1$, $\alpha^2 + r^2(t_0-t^{\flat})^2 \geqslant r_0^2(s-t^{\flat})^2$ and $\alpha^2 + r^2(t_0-t^{\flat})^2 \geqslant \alpha^2$, we deduce
\begin{equation*}
    \lvert I_3 \lvert \, \leqslant \dfrac{1}{r_0} \quad \text{and} \quad \lvert I_4 \lvert \, \leqslant \dfrac{2}{r_0^2}.
        \end{equation*} 

\noindent Since there exists a constant $C > 0$ such that $\Bigg\lvert \ds\int_{t_0}^{t^{\flat}} M_{\pm}^{\alpha}(s) \dint s \, \Bigg\lvert \leqslant C \sum\limits_{i=0}^{5} \lvert I_i \lvert$, we obtain the desired result from the first point according to all the previous computations.

    \item In the following, we will write $\mathrm{I} = [t_0,T]$ with $T \in \R$ and $a$ the function defined for all $t \in \mathrm{I}$ by $a(t) = r(t-t^{\flat})$. Let us consider $t \in \mathrm{I}$. Since $t \leqslant T$ and $\alpha \leqslant 1$ by assumption, there exists a constant $C > 0$ such that for all $(i,j) \in [\![1 , d ]\!]^2$
    \begin{equation*}
        \lvert (M^{\alpha}_{\pm}(t))_{i,j}\lvert \leqslant \dfrac{C}{\sqrt{\alpha^2 + r^2(t-t^{\flat})^2}} \Bigg( 3 + \dfrac{\alpha^2}{\alpha^2 + a^2(t)} + \dfrac{2 \alpha a(t)}{\alpha^2 + a^2(t)} \Bigg).
    \end{equation*}
    Using the fact that $\alpha^2 \leqslant \alpha^2 + a^2(t)$ and $2\alpha a(t) \leqslant \alpha^2 + a^2(t)$ and again that $t_0 \leqslant t \leqslant T$ and $\alpha \leqslant 1$, we deduce 
    \begin{equation*}
        \Bigg\lvert 1 + \Big\| M^{\alpha}_{\pm}(t) \Big\| \Bigg\lvert \leqslant 1 + \dfrac{C_1}{\sqrt{\alpha^2 + r^2(t-t^{\flat})^2}} \leqslant \dfrac{C_2}{\sqrt{\alpha^2 + r^2(t-t^{\flat})^2}}
    \end{equation*}
    with $C_1 > 0$ \review{independent} of $\alpha$ and $C_2 = \text{max}\Big(C_1, \sqrt{1 + r^2(t_0-t^{\flat})^2} \Big)$ or $C_2 = \text{max}\Big(C_1, \sqrt{1 + r^2(T-t^{\flat})^2} \Big)$ depending on whether $t \leqslant t^{\flat}$ or if $t > t^{\flat}$.
\end{enumerate}
\end{proof}

\noindent We now turn our attention to the second term in \( \text{Hess} \, \lambda_{\pm}(q_{\pm}) \) involving the matrix \( \Gamma_{\alpha} \). To this end, we define \( g_{\alpha} \) as the function given for all \( t \in \mathrm{I} \) by
\begin{equation}
\label{eq:g_alpha}
g_{\alpha}(t) = \frac{r}{\sqrt{\alpha^2 + r^2(t - t^{\flat})^2}} \, \Gamma_\alpha(t),
\end{equation}
\noindent which appears in the expression of \( \text{Hess} \, \lambda_{\pm}(q_{\pm}) \). The following lemma provides \review{an antiderivative} of $g_{\alpha}$ and the goal of the following properties is to establish satisfactory controls on this antiderivative.

\begin{lemma}[Antiderivative of $g_\alpha$]
\label{lem:galpha}
    Under the same assumptions as in Proposition \ref{prop:DLHesslambda}, for all $t \in \mathrm{I}$, the matrix-valued function $t \in \mathrm{I} \mapsto \text{sgn}(t-t^{\flat}) G_{\alpha}(t)$, with $G_{\alpha}$ defined in \eqref{def:phase}, is \review{an antiderivative} of $g_{\alpha}$. 
\end{lemma}

\begin{proof}[Proof of Lemma ~\ref{lem:galpha}] We recall that for all $t \in \mathrm{I}$ $$\Gamma_{\alpha}(t) = \Gamma_0 + \dfrac{\alpha^2}{\alpha^2 + r^2 (t-t^{\flat})^2} \Gamma_1.$$
\noindent Therefore, in order to find \review{an antiderivative} of \(g_{\alpha}\), we must compute \review{an antiderivative} of the functions 
\[ t \mapsto \frac{r}{\sqrt{\alpha^2 + r^2(t - t^{\flat})^2}} \, \Gamma_0 
\quad \text{and} \quad t \mapsto \frac{r \alpha^2}{\left(\alpha^2 + r^2(t - t^{\flat})^2\right)^{3/2}} \, \Gamma_1.\]

\noindent Using the antiderivatives given in Appendix~\ref{appendixB} with well-chosen constants, and observing that for all \(t \in \mathrm{I}\),
\[G_{\alpha}(t) = \Gamma_0 h_{\alpha}(t) + \frac{r \lvert t - t^{\flat} \rvert}{\sqrt{\alpha^2 + r^2(t - t^{\flat})^2}} \, \Gamma_1,\]
where \(\Gamma_1\) is defined in \eqref{matGamma1}, we obtain the desired result.
\end{proof}

\noindent Before analyzing estimates on the function \( G_\alpha \), let us first examine the function $h_{\alpha}$.

\begin{lemma}[Study of $h_\alpha$] \label{lem:halpha} Let $\tau$ be a positive real number such that $[t^{\flat},t^{\flat} + \tau] \review{\, \subset I}$. Under the same assumptions as in Proposition \ref{prop:DLHesslambda}, $h_{\alpha}$ is integrable on $[t_0,t^{\flat}]$ (resp. on $[ t^{\flat} , t^{\flat} + \tau]$). Moreover, there exists $C > 0$ (independent of $\alpha$) such that $$\Bigg\lvert \ds\int_{t_0}^{t^{\flat}} h_{\alpha}(s) \dint s \, \Bigg\lvert \leqslant C, \, \, \, \, \Big( \text{resp.} \, \, \Bigg\lvert \ds\int_{t^{\flat}}^{t^{\flat} + \tau} h_{\alpha}(s) \dint s \Bigg\lvert \leqslant C \Big).$$
\end{lemma}

\begin{proof}[Proof of Lemma ~\ref{lem:halpha}] We prove the result for the integral $I^{\alpha} := \ds\int_{t_0}^{t^{\flat}} h_{\alpha}(s) \dint s$ but the other one is done in the same way. If \textbf{$\alpha = 0$}, using the change of variable $u = 2r(t^{\flat}-s)$, we have
    \begin{equation*}
        I^{0} = \dfrac{1}{2r}\ds\int_{0}^{2r(t^{\flat}-t_0)} \ln(u) \dint u.
    \end{equation*}
Since $u \mapsto \ln(u)$ is integrable at $0$, we deduce that there exists $C > 0$ such that $I^{0} \leqslant C$.\\
\noindent Now, if \textbf{$\alpha \neq 0$}, since $t_0 \leqslant s \leqslant t^{\flat}$, we can write $$I^{\alpha} = \ds\int_{t_0}^{t^{\flat}} \ln \Bigg( r(t^{\flat}-s) + \alpha \sqrt{1 + \frac{r^2(s-t^{\flat})^2}{\alpha^2}} \Bigg) \dint s.$$ Using the change of variable $u = \dfrac{r}{\alpha}(t^{\flat} - s)$ and Appendix \ref{appendixB}, we obtain 
    \begin{align*}
        I^{\alpha} & = \dfrac{\alpha}{r} \ds\int_{0}^{\frac{r}{\alpha}(t^{\flat} - t_0)} \ln(\alpha u + \alpha \sqrt{1 + u^2}) \dint u\\
        & = \dfrac{\alpha}{r} \ds\int_{0}^{\frac{r}{\alpha}(t^{\flat} - t_0)} \ln (\alpha) \dint u + \dfrac{\alpha}{r} \ds\int_{0}^{\frac{r}{\alpha}(t^{\flat} - t_0)} \ln (u + \sqrt{1 + u^2}) \dint u\\
        & = (t^{\flat} - t_0) \ln \Bigg(r(t^{\flat} - t_0) + \sqrt{\alpha^2 + r^2(t^{\flat} - t_0)^2} \Bigg) - \dfrac{1}{r} \sqrt{\alpha^2 + r^2(t^{\flat} - t_0)^2} + \dfrac{\alpha}{r}.
        \end{align*}
\noindent As $\alpha \leqslant 1$, we obtain $$\lvert I^{\alpha} \lvert \, \leqslant (t^{\flat} - t_0) \Bigg\lvert \ln \Bigg(r(t^{\flat} - t_0) + \sqrt{\alpha^2 + r^2(t^{\flat} - t_0)^2} \Bigg) \Bigg\lvert + \dfrac{1}{r} \Bigg( \sqrt{1 + r^2(t^{\flat} - t_0)^2} + 1 \Bigg).$$ We recall that the goal of our study is to obtain control bounds that are independent of $\alpha$, this is why some additional work is still required. If $2r_0(t^{\flat} - t_0) \leqslant r(t^{\flat} - t_0) + \sqrt{\alpha^2 + r^2(t^{\flat} - t_0)^2} \leqslant 1$, then $$\ln(2r_0(t^{\flat} - t_0)) \leqslant \ln \Bigg( r(t^{\flat} - t_0) + \sqrt{\alpha^2 + r^2(t^{\flat} - t_0)^2} \Bigg) \leqslant 0.$$ Otherwise, $1 \leqslant r(t^{\flat} - t_0) + \sqrt{\alpha^2 + r^2(t^{\flat} - t_0)^2} \leqslant r(t^{\flat} - t_0) + \sqrt{1 + r^2(t^{\flat} - t_0)^2}$, then $$\ln \Bigg( r(t^{\flat} - t_0) + \sqrt{\alpha^2 + r^2(t^{\flat} - t_0)^2} \Bigg) \leqslant \ln \Bigg( r(t^{\flat} - t_0) + \sqrt{1 + r^2(t^{\flat} - t_0)^2} \Bigg).$$ In each case, the lemma is satisfied.
\end{proof}

\begin{remark}
\label{rq:halpha}
    For all $t \in \R$, $r\lvert t- t^{\flat}\lvert \, \leqslant \sqrt{\alpha^2 + r^2(t-t^{\flat})^2}$ and therefore $$\ln(2r\lvert t- t^{\flat}\lvert) \leqslant \ln \Bigg(r\lvert t- t^{\flat}\lvert + \sqrt{\alpha^2 + r^2(t-t^{\flat})^2} \Bigg).$$ But, for $t$ close enough to $t^{\flat}$, we have $\ln \Bigg(r\lvert t- t^{\flat}\lvert + \sqrt{\alpha^2 + r^2(t-t^{\flat})^2} \Bigg) \leqslant 0$ and then
    \begin{equation*}
        \lvert h_{\alpha}(t) \lvert \leqslant \Big\lvert \ln(2r) + \ln\lvert t - t^{\flat}\lvert \Big\lvert \leqslant C\Big( 1 + \Big\lvert \ln\lvert t - t^{\flat}\lvert \Big\lvert \Big)
    \end{equation*}
    with $C = \text{max}(1, \lvert\ln(2r)\lvert)$. By adjusting the constant, this control holds on the interval $[t_0,t^{\flat})$.
\end{remark}

\noindent Using the definition of $G_{\alpha}$ given by \eqref{def:phase}, together with the fact that the function $t \mapsto \dfrac{r\lvert t - t^{\flat}\lvert}{\sqrt{\alpha^2 + r^2(t-t^{\flat})^2}}$ is uniformly bounded with respect to $\alpha$, we obtain that there exists $C >0$ such that $$\| G_{\alpha}(t) \| \leqslant C \Big( 1 + \Big\lvert h_{\alpha}(t) \Big\lvert \Big).$$ This leads to the following lemma.

\begin{prop}[Study of $G_\alpha$] \label{lem:Galpha} Under the same assumptions as in Proposition \ref{prop:DLHesslambda}, it follows that
    \begin{enumerate}
        \item For all $t \in \mathrm{I} \setminus \lbrace t^{\flat} \rbrace$, there exists $C >0$, such that $$\| G_{\alpha}(t) \|  \leqslant C \Big( 1 + \Big\lvert \ln(\lvert t - t^{\flat}\lvert)\Big\lvert \Big).$$ 
        \item The matrix-valued function $G_{\alpha}$ is integrable on $[t_0,t^{\flat}]$ (resp. on $[ t^{\flat} , t^{\flat} + \tau]$). Moreover, there exists $C > 0$ (independent of $\alpha$) such that $$\Bigg\lvert \ds\int_{t_0}^{t^{\flat}} G_{\alpha}(s) \dint s \, \Bigg\lvert \leqslant C, \, \, \, \, \Big( \text{resp.} \, \, \Bigg\lvert \ds\int_{t^{\flat}}^{t^{\flat} + \tau} G_{\alpha}(s) \dint s \Bigg\lvert \leqslant C \Big).$$
    \end{enumerate}
\end{prop}

\subsection{Profiles of the approximate solution}
\label{subsec:consProfil}

\noindent The following propositions allow to construct the profiles of wave packets that enter and leave the gap region and give a control on some norms in the space $\Sigma^k$. The results of this subsection are similar to those obtained in \cite{FGH2021} but a crucial step of the proof is to obtain that the bounds exhibited by the controls are independent of the parameter ~$\alpha$.

\begin{prop}[Control of the $\Sigma^k$-norm]
\label{prop:normSigu} For all $z_0 \in \mathcal{A}_{\pm}(M,\alpha_0,r_0, \delta, \mathrm{I})$, there exist constants $C_k > 0$ and $\widetilde{C}_k > 0$, independent of $\alpha$, such that for all $t \in \mathrm{I}$, we have $$\norm{u_{\pm}(t)}{\Sigma^k} \leqslant C_k \Bigg( 1 + \Bigg\lvert \ln\Big(  \lvert t - t^{\flat} \lvert \Big) \Bigg\lvert \Bigg) \quad \text{and} \quad \norm{\widetilde{u}_{\pm}(t)}{\Sigma^k} \leqslant C_k \Bigg( 1 + \Bigg\lvert \ln\Big(  \lvert t - t^{\flat} \lvert \Big) \Bigg\lvert \Bigg).$$
\end{prop}

\begin{proof}[Proof of Proposition ~\ref{prop:normSigu}]
    In this proof, we denote the interval $\mathrm{I}$ by $\mathrm{I} = [t_0, T]$.  Let $z_0$ be a point of $\mathcal{A}_{\pm}(M,\alpha_0,r_0,\delta,\mathrm{I})$. We consider $t \in [t_0, t^{\flat}]$ (the proof is the same if $t \in [t^{\flat},T]$). Our goal is to control the norms $\norm{u_{\pm}(t)}{\Sigma^k}$ and $\norm{\widetilde{u}_{\pm}(t)}{\Sigma^k}$ for $k \geqslant 1$. We follow the strategy outlined in \cite{FGH2021}. For convenience, we fix a mode, the proof is identical for both modes, and we focus only on $u_{\pm}$ (the argument for $\widetilde{u}_{\pm}$ is analogous). In the following, $Q_{\pm}$, $\lambda_{\pm}(q_{\pm})$, $M^{\alpha}_{\pm}$ and $u_{\pm}$ will be denoted $Q$, $\lambda(q)$, $M^{\alpha}$ and ~$u$. We introduce the vector $U$ of $\C^{2d}$ defined by $U = \begin{pmatrix}
        yu \\
        D_y u
    \end{pmatrix}.$ Using the commutators
    \begin{equation*}
        [Q(t), D_y] = i \text{Hess} \, \lambda(q(t)) y \, \, \, \text{and} \, \, \, [Q(t), y] = -i D_y
    \end{equation*}
    we obtain
    \begin{equation*}
        \Bigg[ Q(t), \begin{pmatrix}
        yu \\
        D_y u
    \end{pmatrix}\Bigg] u = \begin{pmatrix}
        -iD_y u\\
        i \text{Hess} \, \lambda(q(t)) y u
    \end{pmatrix} = i \begin{pmatrix}
        0 & - I_d \\
        \text{Hess} \, \lambda(q(t)) y & 0
    \end{pmatrix} U.
    \end{equation*}
    Thus, using Proposition \ref{prop:DLHesslambda}, we deduce that $U$ satisfies the following equation 
    $$i\partial_t U - Q(t) U = \underset{\mathrm{N}^{\alpha}(t)}{\underbrace{\begin{pmatrix}
        0 & -i I_d\\
        iM^{\alpha}(t)& 0
    \end{pmatrix}}} U  + i g_{\alpha}(t) \underset{\mathrm{L}}{\underbrace{\begin{pmatrix}
        0 & 0\\
        I_d & 0
    \end{pmatrix}}} U.$$
    According to Propositions \ref{prop:DLHesslambda} and \ref{prop:Malpha}, $t \mapsto \mathrm{N}^{\alpha}(t)$ is smooth on $[t_0, t^{\flat}]$ and there exists $C > 0$ such that \begin{equation}
    \label{eq:propNalpha}
        \Bigg\lvert \ds\int_{t_0}^{t^{\flat}} \mathrm{N}^{\alpha}(s) \dint s \Bigg\lvert \leqslant C \quad \text{and} \quad \|\mathrm{N}^{\alpha}(t)\| \leqslant C \, \, \text{for} \, \, t \in [t_0,t^{\flat}].
    \end{equation}

    \noindent Our aim is to prove that for all $k \in \N$, there exists $C_k > 0$, independent of $\alpha$ such that  for all $t \in [t_0, t^{\flat}]$, $$\lvert \lvert U(t) \lvert \lvert_{\Sigma^k(\R^d, \C^{2d})} \leqslant C_k \Bigg( 1 + \Big\lvert \ln\Big(  \lvert t - t^{\flat} \lvert \Big) \Bigg\lvert \Bigg).$$ We prove this claim by induction. We need to introduce the following projector of rank $d$ $$\mathbb{P} = \begin{pmatrix}
        0 & 0 \\
        0 & I_d
    \end{pmatrix}.$$ This projector satisfies the following relations $$\mathbb{P}(I_{2d} - \mathbb{P}) = 0 \, \, \,, \, \, \,(I_{2d} - \mathbb{P}) \mathrm{L} = 0 \, \, \,\text{and} \, \, \,  \mathbb{P} \mathrm{L} = \mathbb{P} \mathrm{L} (I_{2d} - \mathbb{P}).$$

    \noindent We detail only the first step and we explain why the upper bound is $\alpha$-independent. The following steps use the same arguments and the induction assumption.\\
    
    \noindent \textbf{Step one: $k=0$.} We introduce $V = (I_{2d} - \mathbb{P})U$ and $W = \mathbb{P}U$. Using $V + W = U$, the equation satisfied by $U$ and the relations on the matrix $\mathbb{P}$, we have $$i\partial_t V - Q(t) V = (I_{2d} - \mathbb{P}) \mathrm{N}^{\alpha}(t) U$$ and $$i\partial_t W - Q(t) W = i g_{\alpha}(t) \mathbb{P} \mathrm{L} V + \mathbb{P} \mathrm{N}^{\alpha}(t) U.$$
    Then, we introduce the new variable $$\widetilde{V} = W - G_{\alpha}(t) \mathbb{P} \mathrm{L} V.$$ We notice that $U = V + \widetilde{V} + G_{\alpha}(t) \mathbb{P} \mathrm{L} V$ and $\widetilde{V}$ satisfies
    \begin{align*}
        i\partial_t \widetilde{V} - Q(t) \widetilde{V} & = i\partial_t W - (i \partial_t G_{\alpha}(t)) \mathbb{P} \mathrm{L} V - G_{\alpha}(t) \mathbb{P} \mathrm{L} (i \partial_t V) - Q(t) W + G_{\alpha}(t) \mathbb{P} \mathrm{L} Q(t) V \\
        & = \Big( i\partial_t W - Q(t) W \Big) - G_{\alpha}(t) \mathbb{P} \mathrm{L} \Big(i \partial_t V - Q(t) V \Big) - i g_{\alpha}(t) \mathbb{P} \mathrm{L} V\\
        & = i g_{\alpha}(t) \mathbb{P} \mathrm{L} V + \mathbb{P} \mathrm{N}^{\alpha}(t) U - G_{\alpha}(t) \mathbb{P} \mathrm{L} (I_{2d} - \mathbb{P}) \mathrm{N}^{\alpha}(t) U - i g_{\alpha}(t) \mathbb{P} \mathrm{L} V\\
        & = \mathbb{P} \mathrm{N}^{\alpha}(t) U - G_{\alpha}(t) \mathbb{P} \mathrm{L} \mathrm{N}^{\alpha}(t) U\\
        & = \Bigg(\mathbb{P} - G_{\alpha}(t) \mathbb{P} \mathrm{L} \Bigg) \mathrm{N}^{\alpha}(t) \Bigg( V + \widetilde{V} + G_{\alpha}(t) \mathbb{P} \mathrm{L} V\Bigg)\review{.}
    \end{align*}
    
    \noindent The equations on $V$ and $\widetilde{V}$ can be rewritten as the following system
    \begin{equation*}
    \left\lbrace 
    \begin{array}{cll}
        i\partial_t V - Q(t) V & = & A^{\alpha}(t) V + B^{\alpha}(t) \widetilde{V} \\
        i\partial_t \widetilde{V} - Q(t) \widetilde{V} & = & \widetilde{A}^{\alpha}(t) V + \widetilde{B}^{\alpha}(t) \widetilde{V},
    \end{array}
    \right.
    \end{equation*}
    where $t \mapsto A^{\alpha}(t), B^{\alpha}(t), \widetilde{A}^{\alpha}(t), \widetilde{B}^{\alpha}(t)$ are functions dependent on constant functions and the functions $t \mapsto \mathrm{N}^{\alpha}(t)$, $t \mapsto G_{\alpha}(t)$. Thanks to an energy estimate, we deduce that
    \begin{multline*}
    \normLp{V(t)}{2}{\R^d} + \normLp{\widetilde{V}(t)}{2}{\R^d} \leqslant C_0 \, + \ds\int_{t_0}^{t} \Big(\normLp{B^{\alpha}(s)}{2}{\R^d} + \normLp{\widetilde{A}^{\alpha}(s)}{2}{\R^d} \Big) \Big( \normLp{V(s)}{2}{\R^d} + \normLp{\widetilde{V}(s)}{2}{\R^d} \Big) \diff s.
    \end{multline*}
    Using Grönwall lemma, we obtain that $$\normLp{V(t)}{2}{\R^d} + \normLp{\widetilde{V}(t)}{2}{\R^d} \leqslant C_0 \text{Exp}\Bigg( \ds\int_{t_0}^{t} \normLp{B^{\alpha}(s)}{2}{\R^d} + \normLp{\widetilde{A}^{\alpha}(s)}{2}{\R^d} \diff s \Bigg).$$ Since the functions \( t \mapsto G_{\alpha}(t) \) and \( t \mapsto G_{\alpha}^2(t) \) are integrable by Proposition~\ref{lem:Galpha} and using \eqref{eq:propNalpha}, we deduce that the functions \( B^{\alpha} \) and \( \widetilde{A}^{\alpha} \) are integrable on the interval \( [t_0, t^{\flat}] \). Consequently, there exists a constant \( C > 0 \) such that for all \( t \in [t_0, t^{\flat}] \),
    \begin{equation*}
    \ds\int_{t_0}^{t} \normLp{B^{\alpha}(s)}{2}{\R^d} + \normLp{\widetilde{A}^{\alpha}(s)}{2}{\R^d} \diff s \leqslant C.
    \end{equation*}
    It allows us to conclude that there exists $\widetilde{C} > 0$ such that for all $t \in [t_0, t^{\flat}]$, $$\normLp{V(t)}{2}{\R^d} + \normLp{\widetilde{V}(t)}{2}{\R^d} \leqslant \widetilde{C}.$$
    Thanks to the relation between $U$, $V$ and $\widetilde{V}$, we obtain the existence of a constant $C_1 > 0$ such that for all $t \in [t_0, t^{\flat}]$,
    \begin{align*}
        \normLp{U(t)}{2}{\R^d} & \leqslant \Big( 1 + M_1 \| G_{\alpha}(t) \| \Big) \normLp{V(t)}{2}{\R^d} + \normLp{\widetilde{V}(t)}{2}{\R^d}\\
        & \leqslant \widetilde{C} \Big( 2 + M_1 \| G_{\alpha}(t) \| \Big)\\
        & \leqslant \text{max}(2\widetilde{C}, \widetilde{C}M_1) \Big( 1 + \| G_{\alpha}(t) \| \Big)
    \end{align*}
    which \review{concludes the first step} thanks to Proposition \ref{lem:Galpha}.
\end{proof}

\noindent In the following proposition, we modify the solutions of \eqref{eqProfil} and \eqref{eqProfiltilde} using the phase \( G_{\alpha} \) defined in ~\eqref{def:phase}. This allows us to describe more precisely the behavior of the profiles as \( t \) approaches \( t^{\flat} \).

\begin{prop}[Ingoing profiles] \label{prop:ingoingprofiles} Under the assumptions of Proposition \ref{prop:normSigu}, there exist two functions $u_{\pm}^{\text{in},\alpha} \in \schwartz{\R^d}$ such that for all $k \in \N$, there exists $C_k > 0$ (\review{independent} of $\alpha$) such that for $t < t^{\flat}$, we have
\begin{equation}
\label{controluin}
\Bigg\| \text{Exp} \Big(\mp \dfrac{i}{2} G_{\alpha}(t) y \cdot y \Big) u_{\pm}(t) - u_{\pm}^{\text{in},\alpha} \Bigg\|_{\Sigma^k} \leqslant C_k \lvert t - t^{\flat} \lvert \Big( 1 + \Big\lvert \ln(\lvert t - t^{\flat} \lvert) \Big\lvert \Big).
\end{equation}
\end{prop}

\begin{proof}[Proof of Proposition ~\ref{prop:ingoingprofiles}] For all $z_0 \in \mathcal{A}_{\pm}(M,\alpha_0,r_0,\delta,\mathrm{I})$ and $t \in [t_0, t^{\flat})$, we set $$v_{\pm}^{\alpha}(t) = \text{Exp} \Big(\mp \dfrac{i}{2} G_{\alpha}(t) y \cdot y \Big) u_{\pm}(t).$$

\noindent The aim is to show that there exist $u_{\pm}^{\text{in},\alpha}$ such that $v_{\pm}^{\alpha}(t) \underset{t \to t^{\flat}}{\longrightarrow} u_{\pm}^{\text{in},\alpha}$ in the space ~$\Sigma^k$. Since the space $\Sigma^k$ is complete, we will prove this using a Cauchy criterion. Let us consider $(t,t') \in \R^2$. We can write
\begin{equation*}
    \Big\| v_{\pm}^{\alpha}(t) - v_{\pm}^{\alpha}(t ')\Big\|_{\Sigma^k} = \Bigg\| \ds\int_{t}^{t'} \partial_t v_{\pm}^{\alpha}(s) \diff s \Bigg\|_{\Sigma^k} \leqslant \ds\int_{t}^{t'} \Big\| \partial_t v_{\pm}^{\alpha}(s) \Big\|_{\Sigma^k} \diff s.
\end{equation*}

\noindent According to Lemma \ref{lem:galpha}, we have $$i \partial_t \Big(\mp \dfrac{i}{2} G_{\alpha}(t) y \cdot y \Big) = \mp \dfrac{1}{2} g_{\alpha}(t) y \cdot y$$ and thanks to the definition \eqref{eq:g_alpha} of $g_{\alpha}$, we deduce that $$i \partial_t v_{\pm}^{\alpha}(s) = \text{Exp} \Big(\mp \dfrac{i}{2} G_{\alpha}(s) y \cdot y \Big) \times \Bigg( i \partial_t u_{\pm} \mp \dfrac{r}{2\sqrt{\alpha^2 + r^2 (s - t^{\flat})^2}} \Gamma_{\alpha}(s) y \cdot y \Bigg).$$

\noindent Using the equation $i \partial_t u_{\pm} = -\dfrac{1}{2} \Delta_y u_{\pm} + \dfrac{1}{2} \text{Hess} \, \lambda_{\pm}(q_{\pm}) y \cdot y u_{\pm}$ and Proposition \ref{prop:DLHesslambda}, we obtain $$i \partial_t v_{\pm}(s) =\text{Exp} \Big(\mp \dfrac{i}{2} G_{\alpha}(s) y \cdot y \Big) \times \Bigg( -\dfrac{1}{2} \Delta_y u_{\pm} + \dfrac{1}{2} M_{\pm}^{\alpha}(s) y \cdot y u_{\pm}\Bigg).$$

\noindent Thus, for all $k \in \N$, there exist $C_k > 0$ and $N_k \in \N$ such that for all $\alpha \geqslant 0$ and all $t \in [t_0, t^{\flat}]$,
\begin{align*}
\Big\| \partial_t v_{\pm}^{\alpha}(s) \Big\|_{\Sigma^k} & \leqslant C \Big\| u_{\pm}(s) \Big\|_{\Sigma^{k+2}} \Big( 1 + \Big\| M^{\alpha}_{\pm}(s) \Big\| \Big) \Big( 1 + \Big\| G_{\alpha}(s) \Big\| \Big)^{k}\\
& \leqslant \widetilde{C} \Big\| u_{\pm}(s) \Big\|_{\Sigma^{k+2}} \dfrac{( 1 + \lvert h^{\alpha}(s) \lvert)^{k}}{\sqrt{\alpha^2 + r^2(s-t^{\flat})^2}}
\end{align*}

\noindent according to Propositions \ref{prop:Malpha} and \ref{lem:Galpha}. Using the proof of Proposition \ref{prop:normSigu}, we known that
\begin{equation*}
    \Big\| u_{\pm}(s) \Big\|_{\Sigma^{k+2}} \leqslant C_{k} ( 1 + \lvert h^{\alpha}(s) \lvert).
\end{equation*}
\noindent Thus, we deduce that $$\Big\| \partial_t v_{\pm}^{\alpha}(s) \Big\|_{\Sigma^k} \leqslant \dfrac{C_k ( 1 + \lvert h^{\alpha}(s) \lvert)^{k+1}}{\sqrt{\alpha^2 + r^2(s-t^{\flat})^2}}$$ with $C_k > 0$ and $\widetilde{N_k} \in \N$. Noticing that $\dfrac{1}{\sqrt{\alpha^2 + r^2(s-t^{\flat})^2}} = \text{sgn}(s-t^{\flat}) h'_{\alpha}(s)$, we deduce that the function $s \mapsto \dfrac{C_k ( 1 + \lvert h^{\alpha}(s) \lvert)^{k+1}}{\sqrt{\alpha^2 + r^2(s-t^{\flat})^2}}$ is integrable in $[t,t']$ and that $\ds\int_{t}^{t'} \Big\| \partial_t v_{\pm}^{\alpha}(s) \Big\|_{\Sigma^k} \diff s \longrightarrow 0$ when $t$ tends to $t'$. Consequently, $v_{\pm}^{\alpha}(t)$ has a limit in $\Sigma^k$ when $t$ tends to $t^{\flat}$ and we denote $u_{\pm}^{\text{in},\alpha}$ this limit. The control given in \review{Remark} \ref{rq:halpha} allows to control $\Big\| v_{\pm}^{\alpha}(t) - u_{\pm}^{\text{in},\alpha} \Big\|_{\Sigma^k}$ for $t$ close to $t^{\flat}$ as given in \eqref{controluin}.
\end{proof}

\begin{remark}
In \cite[Corollary $1.9$]{FGH2021}, the phase used to define the ingoing and outgoing profiles is given by \[ \dfrac{i}{2} \Gamma_0 y \cdot y \ln\left(\lvert t - t^{\flat} \rvert\right). \] Even if our analysis also covers the case $\alpha = 0$, we observe that a different phase arises in this setting. Indeed, we have
\[
G_0(t) = \dfrac{i}{2} \Gamma_0 y \cdot y \ln\left(\lvert t - t^{\flat} \rvert\right) + \Gamma_0 \ln(2r) + \Gamma_1.
\] Thus, the function denoted by $u_{\pm}^{\text{in}}$ in \cite{FGH2021} differs from the $u_{\pm}^{\text{in},0}$ defined here by the phase factor ~$e^{i \Gamma_0 \ln(2r) + i \Gamma_1}$.
\end{remark}

\noindent Thanks to Proposition \ref{prop:ingoingprofiles}, we have constructed, for $s,t \in [t_0,t^{\flat})$, a two-parameter semi-group $\mathcal{U}(t,s)$ that provides the solution at time $t$ in terms of that at time $s$ for $s < t$. This propagator satisfies
\[ \mathcal{U}_{\pm}(t,s)\, u_{\pm}(s) = u_{\pm}(t) \quad \text{and} \quad \lim_{t \to t^{\flat},\, t < t^{\flat}} \mathcal{U}_{\pm}(t,s)\, u_{\pm}(s) = u_{\pm}^{\mathrm{in},\alpha}.\]
As explained in the proof of \cite[Corollary~$1.9$]{FGH2021}, using this semi-group, it is possible to construct another two-parameters semi-group $\mathcal{U}(s,t)$, which provides the solution at time $s$ in terms of that at time $t$ for $s < t$. It satisfies
\[\mathcal{U}_{\pm}(s,t)\, u_{\pm}(t) = u_{\pm}(s) \quad \text{and} \quad \quad 
\mathcal{U}_{\pm}(s,t^{\flat})\, u_{\pm}^{\mathrm{in},\alpha} = u_{\pm}(s)\]
and for all $s \in [t_0,t^{\flat})$, $u_{\pm}(s)$ verifies the control \eqref{controluin}. Thus, by prescribing the condition at $t^{\flat}$ and adapting the previous proof for $t > t^{\flat}$, we obtain the following proposition.

\begin{prop}[Outgoing profiles] \label{prop:outgoingprofiles} Let us consider $v_{\pm}^{\text{out},\alpha} \in \schwartz{\R^d}$ and $\widetilde{v}_{\pm}^{\text{out},\alpha} \in \schwartz{\R^d}$. Under the assumptions of Proposition~\ref{prop:existenceu}, there exist functions $v_{\pm}(t)$ and $\widetilde{v}_{\pm}(t)$, defined for $t > t^{\flat}$, satisfying Equations~\eqref{eqProfil} and \eqref{eqProfiltilde} with respectively $v_{\pm}^{\text{out},\alpha}$ and $\widetilde{v}_{\pm}^{\text{out},\alpha}$ as initial data at time $t^{\flat}$, such that for all $k \in \mathbb{N}$ there exist constants $C_k > 0$ and $\widetilde{C}_k > 0$ (independent of $\alpha$) satisfying $$\Bigg\lvert \Bigg\lvert \text{Exp} \Big(\pm \dfrac{i}{2} G_{\alpha}(t) y \cdot y \Big) v_{\pm}(t) - v_{\pm}^{\text{out},\alpha} \Bigg\lvert \Bigg\lvert_{\Sigma^k} \leqslant C_k \lvert t - t^{\flat} \lvert \Big( 1 + \Big\lvert \ln ( \lvert t-t^{\flat} \lvert) \Big\lvert \Big), \quad \text{for } t > t^{\flat}$$
\noindent and 
$$\Bigg\lvert \Bigg\lvert \text{Exp} \Big(\pm \dfrac{i}{2} G_{\alpha}(t) y \cdot y \Big) \widetilde{v}_{\pm}(t) - \widetilde{v}_{\pm}^{\text{out},\alpha} \Bigg\lvert \Bigg\lvert_{\Sigma^k} \leqslant C_k \lvert t - t^{\flat} \lvert \Big( 1 + \Big\lvert \ln ( \lvert t-t^{\flat} \lvert) \Big\lvert \Big) \quad \text{for } t > t^{\flat}.$$
\end{prop}

%%%%%%%%%%%%%%%%%%%%%%%%%%%%%%%%%%%%%%%%%%%%%%%%%%%%%%%%%%%

\section{Passing through the gap region}
\label{sec:passage}

\noindent Now, we focus on the behavior between the time $t^{\flat} - \delta$ and $t^{\flat} + \delta$. According to our assumptions, we know that classical trajectories reach $\Sigma_{\text{nd}}$ during this time interval. In this section, we prove the main theorem (Theorem~\ref{theo:main}), which provides an approximation of the solution \( \psi^{\varepsilon} \) of Equation~\eqref{equation}, when $\varepsilon \to 0$, after passing close to the crossing set.

\subsection{Study of the new unknown}

\noindent As explained in the introduction, we first perform the change of time \eqref{newtime} and unknown \eqref{newinconnue}. This allows us to study the family $(u^{\varepsilon})_{\varepsilon > 0}$ and the next proposition identifies the equation it satisfies. We begin by establishing the equation satisfied by ~$u^{\varepsilon}$, as done in ~\cite{FGH2021}.

\begin{prop}
\label{prop:studyueps}
    Let $k \in \N$. The family $(u^{\varepsilon})_{\varepsilon > 0}$ satisfies for all $(s,y) \in \R \times \R^d$
    \begin{equation}
    \label{eqUeps}
        i \partial_s u^{\varepsilon} = A \Big( sr\Vec{e}_{\theta} + dw(q^{\flat})y + \dfrac{\alpha \Vec{e}_{\theta}^{\perp}}{\sqrt{\varepsilon}} \Big) u^{\varepsilon} + \sqrt{\varepsilon} \Big( - \dfrac{1}{2} \Delta u^{\varepsilon} + B^{\varepsilon}(s,y) u^{\varepsilon} \Big)
    \end{equation}

\noindent where $B^\varepsilon$ is a smooth \review{Hermitian} matrix-valued potential with the following properties: there exist constants $C_0, C_1 > 0$ such that for all $s \in [-s_0, s_0]$ and $y \in \R^d$
\begin{itemize}
    \item $\lvert B^{\varepsilon}(s,y) \lvert \, \leqslant C_0 (s^2 \langle \sqrt{\varepsilon} \lvert y \lvert \rangle + \lvert y \lvert^2),$
    \item $\lvert \nabla B^{\varepsilon}(s,y) \lvert \, \leqslant C_1 (s^2 \sqrt{\varepsilon} \lvert y \lvert^2 + \lvert y \lvert + \sqrt{\varepsilon} s^2),$
    \item for all $\lvert \beta \lvert \, \geqslant 2$, $\lvert \partial_y^{\beta} B^{\varepsilon}(s,y) \lvert \, \leqslant C_{\beta} \varepsilon^{\frac{\lvert \beta \lvert - 2}{2}} \langle \sqrt{\varepsilon} y \rangle^2$ with $C_{\beta} > 0$.
\end{itemize}
\end{prop}

\begin{proof}[Proof of Proposition ~\ref{prop:studyueps}] Computations and the relation between $\psi^{\varepsilon}$ and $u^{\varepsilon}$ give
\begin{equation*}
    \Bigg[ i \sqrt{\varepsilon} \partial_s + \dfrac{\varepsilon}{2} \Delta_y - \Big(v(q_0 + \sqrt{\varepsilon} y) - v(q_0) - \nabla v(q_0) \cdot \sqrt{\varepsilon} y \Big) \Bigg] u^{\varepsilon} = A(w(q_0 + \sqrt{\varepsilon} y)) u^{\varepsilon}.
\end{equation*}
Applying the Taylor formula with an integral remainder close to $q_0$ for the function $v$ and $w$, we obtain 
\begin{equation*}
    i \partial_s u^{\varepsilon} = \dfrac{1}{\sqrt{\varepsilon}} A(w(q_0) + \sqrt{\varepsilon} \dint w(q_0) \cdot y) u^{\varepsilon} + \sqrt{\varepsilon} \Big( - \Delta_y u^{\varepsilon} + R(t,\sqrt{\varepsilon}y) y \cdot y u^{\varepsilon} \Big)
\end{equation*}
where $R$ depends on the remainder of the Taylor expansions. To conclude, we use Taylor expansions close to $t^{\flat}$ of $q_0$ (see \eqref{DLp0q0}) and of $t \mapsto w(q_0(t))$ and $t \mapsto \dint w(q_0(t))$. The remainder arising from the expansion, and its derivative, both have the desired properties stated in the proposition, thanks to Assumption\review{s} \ref{hypsurV2}.
\end{proof}

\subsection{The Landau--Zener model}
\label{subsec:LZmodel}

For $\eta \in \C^2$, we consider the model problem 
\vspace{-0.1cm}
\begin{equation}
\label{eqmodel}
    \left\{
        \begin{array}{r c l }
        i \partial_s w^{\varepsilon} & = & A \Big( sr\Vec{e}_{\theta} + \eta  + \dfrac{\alpha \Vec{e}_{\theta}^{\perp}}{\sqrt{\varepsilon}} \Big) w^\varepsilon  \\
        w^\varepsilon(0,\eta) & = & w_0^{\varepsilon}(\eta).
        \end{array}
    \right.
\end{equation}

\noindent The goal of this subsection is to compare the solution of the model problem with the solution of ~\eqref{eqUeps}. To achieve this, we first transform the system \eqref{eqmodel} into a Landau--Zener problem, for which the asymptotics of the solution are known when $\eta \, + \, \dfrac{\alpha}{\sqrt{r \varepsilon}}$ is in a compact set.\\

\noindent Let $\phi$ be a real number such that $\Vec{e}_{\phi} = - \Vec{e}_{\theta}$ and $\mathcal{R}(\phi)$ be the rotation matrix \begin{equation}
\label{eqmatrot}
    \mathcal{R}(\phi) = \begin{pmatrix}
        \cos \Big( \dfrac{\phi}{2} \Big) & - \sin \Big( \dfrac{\phi}{2} \Big)\\
        \sin \Big( \dfrac{\phi}{2} \Big) & \cos \Big( \dfrac{\phi}{2} \Big)
    \end{pmatrix}.
\end{equation}
\noindent If $w^\varepsilon$ is the solution of the system \eqref{eqmodel}, the function $u_{\text{LZ}}^{\varepsilon}$ defined for all $(s,\eta) \in \R \times \C^2$ by 
\begin{equation}
\label{eqLZv}
    u_{\text{LZ}}^{\varepsilon}(s,\eta) = \mathcal{R}(\phi)^{-1} w^\varepsilon \Big( \frac{s}{\sqrt{r}} , \sqrt{r} \eta \Big),
\end{equation} is solution of the Landau--Zener problem \eqref{eqLZ} with $z = z(\eta) := \Big( \Vec{e}_{\theta} \cdot \eta, \Vec{e}_{\theta}^{\perp} \cdot \eta + \dfrac{\alpha}{\sqrt{r \varepsilon}} \Big)$. A detailed proof is provided in Appendix \ref{appendixB}, Lemma \ref{lem:LZproblem}. Using the appendices of \cite{fermaniangerard2002} and \cite{kammererlasser2008} to study the solution of the Landau--Zener model and the relation \eqref{eqLZv}, we obtain the following proposition. We adapt the proof from \cite{FGH2021} to the case $\alpha \neq 0$.

\begin{prop}
\label{prop:soleqmodel}
    Let $R$ be a positive real number such that $R \geqslant 1$. We assume that \[ \alpha \, \leqslant \dfrac{R}{2} \sqrt{\varepsilon} \quad \text{and} \quad \lvert \eta \lvert \, \leqslant \dfrac{R}{2}. \]    
    \noindent Then, there exist $\alpha_1^{\text{in}}, \alpha_2^{\text{in}}, \alpha_1^{\text{out}}, \alpha_2^{\text{out}} \in \schwartz{\R^2}$ such that
    \begin{align}
        w^{\varepsilon}(s, \eta) & = e^{i \Lambda^{\varepsilon,\alpha}(s,\eta)} \alpha_1^{\text{in}}(\eta) \Vec{Y}_{\flat}^{\perp} + e^{- i \Lambda^{\varepsilon,\alpha}(s,\eta)} \alpha_2^{\text{in}}(\eta) \Vec{Y}_{\flat} + \mathcal{O}\Big( R^3 \lvert s \lvert^{-1} \Big) \quad \text{when} \, \, \,  s \to - \infty, \label{LZminusin}\\
        w^{\varepsilon}(s, \eta) & = e^{i \Lambda^{\varepsilon,\alpha}(s,\eta)} \alpha_1^{\text{out}}(\eta) \Vec{Y}_{\flat}^{\perp} + e^{- i \Lambda^{\varepsilon,\alpha}(s,\eta)} \alpha_2^{\text{out}}(\eta) \Vec{Y}_{\flat} + \mathcal{O}\Big( R^3 \lvert s \lvert^{-1} \Big) \quad \text{when} \, \, \,  s \to + \infty, \label{LZplusin}
    \end{align}
where $\Lambda^{\varepsilon,\alpha}$ is defined in \eqref{phase}. Besides, we have the following relation
\begin{equation}
    \label{passage}
    \begin{pmatrix}
        \alpha_1^{\text{out}} \\
        \alpha_2^{\text{out}}
    \end{pmatrix} = \mathcal{S}\Big( \dfrac{\eta \cdot \Vec{e}_{\theta}^{\perp}}{\sqrt{r}} + \dfrac{\alpha}{\sqrt{r \varepsilon}}\Big) \begin{pmatrix}
        \alpha_1^{\text{in}} \\
        \alpha_2^{\text{in}}
    \end{pmatrix} \quad \text{with} \quad \mathcal{S}(z) = \begin{pmatrix}
        a(z) & -\overline{b}(z) \\
        b(z) & a(z)
    \end{pmatrix} \quad \text{for } z \in \C,
\end{equation}
\noindent where the coefficients $a$ and $b$ are given for all $z \in \R$ by \eqref{fonctionaetb}.
\end{prop}

\begin{proof}[Proof of Proposition ~\ref{prop:soleqmodel}]
    Since we can take $\eta + \dfrac{\alpha}{\sqrt{r \varepsilon}}$ in a fixed compact \review{set} of size $R$, using the appendices of \cite{fermaniangerard2002} and \cite{kammererlasser2008}, when $s \to \pm \infty$, we have

    \begin{equation*}
        u_{\text{LZ}}^{\varepsilon}(s, \eta) = e^{i \lambda^{\varepsilon,,\alpha}(s,\eta)} \begin{pmatrix}
            u_1^{\pm}(z_2(\eta)) \\ 
            0
        \end{pmatrix} + e^{- i \lambda^{\varepsilon,\alpha}(s,\eta)} \begin{pmatrix}
            0 \\ 
            u_2^{\pm}(z_2(\eta))
        \end{pmatrix} + \mathcal{O}\Big( R^2 \lvert s \lvert^{-1} \Big),
    \end{equation*}

    \noindent with $\lambda^{\varepsilon,\alpha}$ given for all $(s,\eta) \in \R \times \R^2$ by

    \begin{equation*}
        \lambda^{\varepsilon,\alpha}(s,\eta) = \dfrac{1}{2} \Bigg[ \Big( s + \eta \cdot \Vec{e}_{\theta} \Big)^2 + \Big( \eta \cdot \Vec{e}_{\theta}^{\perp} + \dfrac{\alpha}{\sqrt{r\varepsilon}} \Big)^2 \ln \lvert s + \eta \cdot \Vec{e}_{\theta} \lvert \Bigg]
    \end{equation*}

\noindent and with $(u_1^{\pm}, u_2^{\pm}) \in \review{\schwartz{\R^2}}$ that satisfy the following relation for all $\eta \in \R^2$
\begin{equation*}
    \begin{pmatrix}
        u_1^{+}(\eta) \\
        u_2^{+}(\eta)
    \end{pmatrix} = \begin{pmatrix}
        a(\eta) & -\overline{b}(\eta) \\
        b(\eta) & a(\eta)
    \end{pmatrix} \begin{pmatrix}
        u_1^{-}(\eta) \\
        u_2^{-}(\eta)
    \end{pmatrix} .
\end{equation*}
Since for all $(s,\eta) \in \R \times \R^2$, $w^{\varepsilon}(s, \eta) = \mathcal{R}(\phi) u_{\text{LZ}}^{\varepsilon} \Big(s\sqrt{r}, \dfrac{\eta}{\sqrt{r}} \Big)$ with $\mathcal{R}(\phi)$ given in \eqref{eqmatrot}, we can write

\begin{equation*}
        w^{\varepsilon}(s, \eta) = e^{i \widetilde{\lambda}^{\varepsilon}(s,\eta)} u_1^{\pm}\Big(z_2(\frac{\eta}{\sqrt{r}})\Big)  \mathcal{R}(\phi) \begin{pmatrix}
            1 \\ 
            0
        \end{pmatrix} + e^{- i \widetilde{\lambda}^{\varepsilon}(s,\eta)} u_2^{\pm}(z_2(\frac{\eta}{\sqrt{r}})) \mathcal{R}(\phi)\begin{pmatrix}
            0 \\ 
            1
        \end{pmatrix} + \mathcal{O}\Big( R^2 \lvert s \lvert^{-1} \Big),
    \end{equation*}
\noindent where   
    \begin{align*}
    \widetilde{\lambda}^{\varepsilon}(s,\eta) & = \lambda \Big(s\sqrt{r}, \dfrac{\eta}{\sqrt{r}} \Big)\\
    & = \dfrac{1}{2r} \Bigg[ \Big( sr + \eta \cdot \Vec{e}_{\theta} \Big)^2 + \Big( \eta \cdot \Vec{e}_{\theta}^{\perp} + \dfrac{\alpha}{\sqrt{\varepsilon}} \Big)^2 \ln \Big\lvert s\sqrt{r} + \dfrac{\eta \cdot \Vec{e}_{\theta}}{\sqrt{r}} \Big\lvert \Bigg].
    \end{align*}

\noindent As $\lvert \eta \lvert \, \leqslant \dfrac{R}{2}$, we have $\dfrac{\eta \cdot \Vec{e}_{\theta}}{sr} \underset{s \to \pm \infty}{\longrightarrow} 0$ and we have

\begin{equation*}
    \ln \Big\lvert s\sqrt{r} + \dfrac{\eta \cdot \Vec{e}_{\theta}}{\sqrt{r}} \Big\lvert = \ln \Big\lvert s\sqrt{r} \Big( 1 + \dfrac{\eta \cdot \Vec{e}_{\theta}}{sr} \Big) \Big\lvert = \ln \lvert s \sqrt{r} \lvert + \dfrac{\eta \cdot \Vec{e}_{\theta}}{sr} + \mathcal{O}\Big( R \lvert s \lvert^{-1} \Big) = \ln \lvert s \sqrt{r} \lvert + \mathcal{O}\Big( R \lvert s \lvert^{-1} \Big).
\end{equation*}

\noindent We deduce that

\begin{equation*}
    \widetilde{\lambda}^{\varepsilon}(s,\eta) = \dfrac{1}{2r} \Bigg[ \Big( sr + \eta \cdot \Vec{e}_{\theta} \Big)^2 + \Big( \eta \cdot \Vec{e}_{\theta}^{\perp} + \dfrac{\alpha}{\sqrt{\varepsilon}} \Big)^2 \ln \lvert s\sqrt{r} \lvert \Bigg] + \mathcal{O}\Big( R^3 \lvert s \lvert^{-1} \Big).
\end{equation*}

\noindent Thus, performing a Taylor expansion of the exponential, we deduce that when $s \to \pm \infty$
\vspace{-0.1cm}
\begin{equation*}
        w^{\varepsilon}(s, \eta) = e^{i \Lambda^{\varepsilon,\alpha}(s,\eta)} u_1^{\pm}(z_2(\frac{\eta}{\sqrt{r}}))  \mathcal{R}(\phi) \begin{pmatrix}
            1 \\ 
            0
        \end{pmatrix} + e^{- i \Lambda^{\varepsilon,\alpha}(s,\eta)} u_2^{\pm}(z_2(\frac{\eta}{\sqrt{r}})) \mathcal{R}(\phi)\begin{pmatrix}
            0 \\ 
            1
        \end{pmatrix} + \mathcal{O}\Big( R^3 \lvert s \lvert^{-1} \Big),
    \end{equation*}

\noindent with $\Lambda^{\varepsilon,\alpha}$ given in \eqref{phase}.

 \noindent According to the choice of $\phi$, we have the relation $\mathcal{R}(\phi)^{-1} A(\Vec{e}_{\theta}) \mathcal{R}(\phi) = \begin{pmatrix}
    -1 & 0 \\
    0 & 1
\end{pmatrix}$. Computations give the relation $$\begin{pmatrix}
    -1 & 0 \\
    0 & 1
\end{pmatrix} \mathcal{R}(\phi)^{-1} \Vec{Y}_{\flat} = \mathcal{R}(\phi)^{-1} \Vec{Y}_{\flat}.$$ 

\noindent Since $\Vec{Y}_{\flat}$ is normalized and $\mathcal{R}(\phi)^{-1}$ is orthogonal, we deduce that there exists $\zeta \in \lbrace -1 , 1 \rbrace$ such that $$\Vec{Y}_{\flat} = \zeta \mathcal{R}(\phi)\begin{pmatrix}
            0 \\ 
            1
        \end{pmatrix}.$$

\noindent Using the fact $\begin{pmatrix}
            0 \\ 
            1
        \end{pmatrix} = \mathcal{R}(\pi) \begin{pmatrix}
            1 \\ 
            0
        \end{pmatrix}$, the relation $\mathcal{R}(\phi) \mathcal{R}(\pi) = \mathcal{R}(\pi) \mathcal{R}(\phi)$ and the definition of $\Vec{Y}_{\flat}^{\perp}$ given in \eqref{eq:ICvecafter}, we deduce that \[ \Vec{Y}_{\flat}^{\perp} = \zeta \mathcal{R}(\phi)\begin{pmatrix}
            1 \\ 
            0
        \end{pmatrix}.\] The result of the proposition follows by setting $$ \alpha_j^{\text{in}}(\eta) = \zeta u_j^{-}\Big(z_2 \Big(\frac{\eta}{\sqrt{r}} \Big)\Big) \, \, \text{and} \, \, \alpha_j^{\text{out}}(\eta) = \zeta u_j^{+} \Big(z_2\Big(\frac{\eta}{\sqrt{r}}\Big)\Big)$$ for all $\eta \in \R^2$ and $j \in \lbrace 1, 2 \rbrace$.
\end{proof}

\subsection{Change of scale}
\label{subsec:newscale}

\noindent The goal of this subsection is to prove Lemma \ref{prop:changetime} and to analyze the phase factor as well as the directions $\vec{Y}_{\pm}$ and $\vec{\widetilde{Y}}_{+}$ using the change of variables \eqref{newinconnue}. To this end, we will use the Taylor expansions established in Section \ref{sec:analysisCQ}. We recall that since $\delta$ will be chosen as a suitable function of $\varepsilon$, we can always choose $\varepsilon$ small enough such that $\delta \leqslant \delta_0$, where $\delta_0$ is the minimal size for which the Taylor expansions in Propositions \ref{prop:DLpq}, \ref{prop:DLaction} and \ref{prop:DLPi} remain valid. Thus, we can use the results of Section \ref{sec:analysisCQ} for $t \in \mathrm{I}$ such that $\alpha \leqslant \lvert t - t^{\flat} \lvert \, \leqslant \delta$.

\begin{proof}[Proof of Lemma ~\ref{prop:changetime}] \noindent 
\begin{enumerate}[leftmargin=*, labelindent=0pt]
    \item The asymptotic expansions of the classical quantities, given in Propositions \ref{prop:DLpq} and \ref{prop:DLaction}, imply that for all $- s_0 \leqslant s \leqslant - \dfrac{\alpha}{\varepsilon}$, we have the following pointwise estimate for $t = t^{\flat} + \sqrt{\varepsilon} s$
\begin{multline*}
    \dfrac{i}{\varepsilon}  S_{-}(t;t_0,z_0) + \dfrac{i}{\varepsilon} p_{-}(t) \cdot (x - q_{-}(t)) = \dfrac{i}{\varepsilon} S_{-}^{\flat} + \dfrac{i}{\varepsilon}  S_{0}(t;t^{\flat},z^{\flat}) + \dfrac{i}{\sqrt{\varepsilon}} p_0(t) \cdot y + \dfrac{is}{2\sqrt{\varepsilon}} \sqrt{\alpha^2 + r^2 s^2 \varepsilon} \\ +  \dfrac{i \alpha^2}{2r \varepsilon} \Argsh{\dfrac{r}{\alpha}\sqrt{\varepsilon}s} + \dfrac{i}{r \sqrt{\varepsilon}} \eta \cdot \Vec{e}_{\theta} \Big( \sqrt{\alpha^2 + r^2 s^2 \varepsilon} - \alpha \Big) + \dfrac{i \alpha}{r \sqrt{\varepsilon}} \eta \cdot \Vec{e}_{\theta}^{\perp} \Argsh{\dfrac{r}{\alpha}\sqrt{\varepsilon}s} \, + \, \sigma_1(\varepsilon,s) 
\end{multline*}

\noindent with $\lvert \sigma_1(\varepsilon,s) \lvert \, \leqslant C_1 ( \sqrt{\varepsilon} s^3 + \sqrt{\varepsilon} s^2 \lvert y \lvert)$. Since $-s_0 \to - \infty$ when $\varepsilon \to 0$ and $\Big\lvert \dfrac{\alpha}{\sqrt{\varepsilon}} \Big\lvert \leqslant \dfrac{R}{2}$, we can use the asymptotic expansions of Appendix \ref{appendixB} to obtain the following Taylor expansion when $s \to - \infty$
\begin{multline*}
    \dfrac{i}{\varepsilon}  S_{-}(t;t_0,z_0) + \dfrac{i}{\varepsilon} p_{-}(t) \cdot (x - q_{-}(t)) = \dfrac{i}{\varepsilon} S_{-}^{\flat} + \dfrac{i}{\varepsilon}  S_{0}(t;t^{\flat},z^{\flat}) + \dfrac{i}{\sqrt{\varepsilon}} p_0(t) \cdot y - \dfrac{irs^2}{2} - \dfrac{i \alpha^2}{4r\varepsilon} \\ - \dfrac{i \alpha^2}{2r \varepsilon} \ln \Big( \dfrac{2r\sqrt{\varepsilon} \lvert s \lvert}{\alpha} \Big) - \dfrac{i\alpha}{r \sqrt{\varepsilon}} \eta \cdot \Vec{e}_{\theta} - is \eta \cdot \Vec{e}_{\theta}  - \dfrac{i \alpha}{r \sqrt{\varepsilon}} \eta \cdot \Vec{e}_{\theta}^{\perp} \ln \Big( \dfrac{2r\sqrt{\varepsilon} \lvert s \lvert}{\alpha} \Big) \, + \, \sigma_2(\varepsilon,s) 
\end{multline*}

\noindent with $\lvert \sigma_2(\varepsilon,s) \lvert \, \leqslant C_2 \Big( \dfrac{1}{s} + \sqrt{\varepsilon} s^3 + \sqrt{\varepsilon} s^2 \lvert y \lvert \Big)$. Using the definition \eqref{phase} of $\Lambda^{\varepsilon,\alpha}$ and \eqref{phase2} of $\Phi^{\varepsilon,\alpha}$ and noticing that $\Gamma_0 y \cdot y = \dfrac{1}{r}(\eta \cdot \Vec{e}_{\theta}^{\perp})^2$, we can identify some terms and we deduce that
\begin{multline*}
 \dfrac{i}{\varepsilon}  S_{-}(t;t_0,z_0) + \dfrac{i}{\varepsilon} p_{-}(t) \cdot (x - q_{-}(t)) = \dfrac{i}{\varepsilon} S_{-}^{\flat} + \dfrac{i}{\varepsilon} S_{0}(t;t^{\flat},z^{\flat}) + \dfrac{i}{\sqrt{\varepsilon}} p_0(t) \cdot y + \dfrac{i}{2} \Gamma_0 y \cdot y \ln(\sqrt{r} \lvert s \lvert) \\ - i \Lambda^{\varepsilon,\alpha}(s, \eta(y))  - i \Phi^{\varepsilon,\alpha}(\eta(y)) - \dfrac{i \alpha}{r \sqrt{\varepsilon}} \eta(y) \cdot \Vec{e}_{\theta} \, + \, \sigma_2(\varepsilon,s).
\end{multline*}

\noindent We now apply the asymptotic expansion of $G_{\alpha}$, provided in Lemma \ref{usefulTE} of Appendix \ref{appendixB}, which yields
\begin{equation*}
    G_{\alpha}(t^{\flat} + \sqrt{\varepsilon} s) = \Gamma_{0} \ln (\sqrt{r} \lvert s \lvert) + \Gamma_0 \ln(2\sqrt{r\varepsilon}) + \Gamma_1 + \, \sigma_3 \quad \text{with } \lvert \sigma_3 \lvert \, \leqslant \dfrac{C_3}{s}.
\end{equation*}
\noindent To conclude and obtain the desired result, we substitute in the last expression of the phase.
\item We recall the notations $S_{-}^{\flat} = S_{-}(t^{\flat};t_0,z_0)$ and $\widetilde{S}_{+}^{\flat} = \widetilde{S}_{+}(t^{\flat};t_0,z_0)$. The proof is exactly the same using Propositions \ref{prop:DLpq}, \ref{prop:DLtildez} and \ref{prop:DLaction} with $$S_{-}(t;t^{\flat},z^{\flat}) = S_{-}(t;t_0,z_0) - S_{-}^{\flat} \quad \text{and} \quad \widetilde{S}_{+}(t;t^{\flat},z^{\flat}) = \widetilde{S}_{+}(t;t_0,z_0) - \widetilde{S}_{+}^{\flat},$$ up to a sign change noting that the expansions of Appendix \ref{appendixB} are taken for $s \to + \infty$.
\end{enumerate}
\end{proof}

\noindent The following Proposition is obtained using Proposition~\ref{prop:DLPi} together with the asymptotic results presented in Appendix ~\ref{appendixA}. This is useful for studying explicitly the directions of the wave packets close to the crossing.

\begin{prop}
    \label{lem:veclimit}
    Under the same assumptions as in Lemma \ref{prop:changetime}, we have the following estimates  
    \begin{equation*}
        \Pi_{\pm,\alpha}(q_{\pm}(t^{\flat} + \sqrt{\varepsilon} s)) = \dfrac{1}{2} \Big( I_2 \pm \sgn(s) A(\vec{e}_{\theta})\Big) + \sigma(s) \quad \text{with } |\sigma(s)| \, \leqslant \frac{C}{s},
    \end{equation*}
    for $s := \dfrac{t - t^{\flat}}{\sqrt{\varepsilon}}$ large enough (i.e. $\varepsilon \to 0$ and $|t - t^{\flat}| \, \leqslant \delta$).
\end{prop}

\noindent A similar result holds for the ``drifted" classical trajectories $\widetilde{z}_{\pm}$. The previous proposition allows us to establish a connection, close to the crossing, between the direction of the wave packets that approach our solution and the vectors $\Vec{Y}^{\flat}$ and $\Vec{Y}^{\flat\perp}$, defined in Equation \eqref{eq:ICvecafter}.

\begin{corollary}
\label{coro:vec}
    Under the same assumptions as in Proposition \ref{lem:veclimit} and using the notations of Propositions \ref{prop:Vec} and \ref{prop:Vecaftercrossing}, we have the following expansions when $\varepsilon \to 0$ and $|t - t^{\flat}| \, \leqslant \delta$
    \begin{enumerate}[leftmargin=*, labelindent=0pt]
        \item If $t < t^{\flat}$, we have: $\vec{Y}_-(t) := \vec{\mathcal{V}}(q_-(t)) = \Vec{Y}^{\flat} \, + \, \mathcal{O}(\sqrt{\varepsilon} \delta^{-1} + \delta).$
        \item If $t > t^{\flat}$, we have: \[ \vec{Y}_-(t) := \vec{\mathcal{V}}(q_-(t)) = \Vec{Y}^{\flat\perp} \, + \, \mathcal{O}(\sqrt{\varepsilon} \delta^{-1} + \delta) \quad \text{and} \quad \vec{\widetilde{Y}_+}(t) := \vec{\mathcal{V}}(\widetilde{q}_+(t)) = \Vec{Y}^{\flat} \, + \, \mathcal{O}(\sqrt{\varepsilon} \delta^{-1} + \delta).\]
    \end{enumerate}
\end{corollary}

\begin{proof}[Proof of Corollary ~\ref{coro:vec}] In this proof, we use the notation introduced in Subsection \ref{subsec:studyVec}, just after Propositon \ref{prop:DLPi}, for the eigenprojectors $\Pi_{\pm,\alpha}^{\flat}$, $\Pi_{\pm,0}^{\flat,-}$ and $\Pi_{\pm,0}^{\flat,+}$.
\begin{enumerate}[leftmargin=*, labelindent=0pt]
\item For $|t - t^{\flat} | \leqslant \delta$, we have the following Taylor expansion
\[ \vec{\mathcal{V}}(q_-(t)) := \vec{Y}_-(t) = \vec{Y}_-(t^{\flat}) + \sigma_1 \quad \text{with} \quad |\sigma_1| \leqslant C \lvert t - t^{\flat}|.\]
\noindent According to Proposition \ref{prop:Vec}, \(\vec{Y}_-(t^{\flat})\) is an eigenvector of $V(q^{\flat})$ associated with the eigenvalue $\lambda_-(q^{\flat})$. Therefore, we obtain
    \begin{align*}
    \vec{\mathcal{V}}(q_-(t))& = \Pi_{-,\alpha}^{\flat} \vec{Y}_-(t^{\flat}) + \sigma_1 \\
    & = \Big( \Pi_{-,\alpha}^{\flat} - \Pi_{-,0}^{-,\flat} \Big) \vec{Y}_-(t^{\flat}) + \Pi_{-,0}^{-,\flat} \vec{Y}_-(t^{\flat}) + \sigma_1. 
    \end{align*}
    According to the definition of $\Vec{Y}_{\flat}$ given in Equation \eqref{eq:ICvecafter}, we have 
    \[ \vec{\mathcal{V}}(q_-(t)) = \Vec{Y}_{\flat} + \Big( \Pi_{-,\alpha}^{\flat} - \Pi_{-,0}^{-,\flat} \Big) \vec{Y}_-(t^{\flat}) + \sigma_1.\]
    The desired result is given by Proposition \ref{lem:veclimit} and the value of $\Pi_{-,0}^{\flat,-}$ defined in Subsection \ref{subsec:studyVec}.
\item The second point is proven in the same way, using the value at time $t^{\flat}$ of the vectors $\vec{Y}_-$ and $\vec{\widetilde{Y}_+}$ given in Proposition \ref{prop:Vecaftercrossing} and the relations
    \begin{equation*}
        \Pi_{\pm,\alpha}^{\flat} = \Big( \Pi_{\pm,\alpha}^{\flat} - \Pi_{\pm,0}^{+,\flat} \Big) + \Pi_{\pm,0}^{+,\flat}, \quad \Pi_{-,0}^{+,\flat} = \Pi_{+,0}^{-,\flat}, \quad \Pi_{+,0}^{+,\flat} = \Pi_{-,0}^{-,\flat} \quad \text{and} \quad \Pi_{+,0}^{-,\flat} \Vec{Y}^{\flat\perp} = \Vec{Y}^{\flat\perp}. 
    \end{equation*}
\end{enumerate}
\end{proof}

\subsection{Identification of the Landau--Zener data using the ingoing wave packet}
\label{subsec:ingoingWP}

\noindent Now, we prove Theorem \ref{theo:ingoingWP} which allows us to identify the incoming wave packet in the region $\lbrace \lvert w(x) \lvert \, \leqslant \delta \rbrace$ at time $t^{\flat} - \delta$. This enables us to obtain a Landau--Zener data at time $-s_0 = -\frac{\delta}{\sqrt{\varepsilon}}$. When \( t = t^{\flat} - \delta \), using the change of time \eqref{newtime} and the relation \( \sqrt{\varepsilon}\, \delta^{-1} \ll 1 \), we study the limit \( s \to -\infty \). This proof is based on the one in~\cite[Proposition $4.6$]{FGH2021}.

\begin{proof}[Proof of Theorem ~\ref{theo:ingoingWP}]
According to the adiabatic theorem, we know that for all $t \leqslant t^{\flat} - \delta$ and all $x \in \R^d$
    \begin{equation*}
    \psi^{\varepsilon}(t,x) = \varepsilon^{- \frac{d}{4}} e^{\frac{i}{\varepsilon} S_{-}(t,t_0,z_0) + \frac{i}{\varepsilon} p_{-}(t) \cdot (x - q_{-}(t))} \vec{\mathcal{V}}_{-}(q_{-}(t)) \times u_{-}\Big(t, \dfrac{x - q_{-}(t)}{\sqrt{\varepsilon}} \Big) + \mathcal{O} \Big( (\sqrt{\varepsilon} \delta^{-1} + \varepsilon^{\frac{3}{2}} \delta^{-4})(1 + \lvert \ln \delta \lvert) \Big)
    \end{equation*}
    
\noindent in $\Sigma_{\varepsilon}^k$. Thanks to the first point of Corollary \ref{coro:vec}, when $s \to - \infty$, we have a relation between $\vec{\mathcal{V}}_{-}(q_{-})$ and $\Vec{Y}_{\flat}$. Using this relation and the fact that $\mathcal{O} \Bigg( \dfrac{1}{s^2} \Bigg) = \mathcal{O} \Bigg( \dfrac{1}{s} \Bigg) = \mathcal{O} \Big(  \sqrt{\varepsilon} \delta^{-1} \Big)$, we obtain the following estimate in $\Sigma_{\varepsilon}^k$
\vspace{-0.3cm}
\begin{multline*}
    \psi^{\varepsilon}(t,x) = \varepsilon^{- \frac{d}{4}} \text{Exp} \Bigg( \dfrac{i}{\varepsilon} S_{-}^{\flat} + \dfrac{i}{\varepsilon}  S_{0}(t,t^{\flat},z^{\flat}) + \dfrac{i}{\sqrt{\varepsilon}} p_0(t) \cdot y\Bigg) \text{Exp} \Bigg(- i \Lambda^{\varepsilon,\alpha}(s, \eta) - \dfrac{i \alpha}{r \sqrt{\varepsilon}} \eta \cdot \Vec{e}_{\theta} \\ - \dfrac{i\alpha^2}{4r \varepsilon} + \dfrac{i\alpha^2}{2r \varepsilon} \ln \Big( \dfrac{\alpha}{2\sqrt{r \varepsilon}} \Big) + \dfrac{i}{2r} (\eta \cdot \Vec{e}_{\theta})^2 + \dfrac{i\alpha}{r \sqrt{\varepsilon}} \eta \cdot \Vec{e}_{\theta}^{\perp} \ln \Big( \dfrac{\alpha}{2\sqrt{r \varepsilon}} \Big) + \dfrac{i}{2r} (\eta \cdot \Vec{e}_{\theta}^{\perp})^2 \ln \Big( \dfrac{1}{2\sqrt{r \varepsilon}} \Big) - \dfrac{i}{2} \Gamma_1 y \cdot y \Bigg) \\ \times \text{Exp} \Bigg( \dfrac{i}{2} G_{\alpha}(t^{\flat} + \sqrt{\varepsilon} s) y \cdot y \Bigg) \times u_{-}\Big(t, \dfrac{x - q_{-}(t)}{\sqrt{\varepsilon}} \Big) \times \Vec{Y}_{\flat}  + \mathcal{O} \Big( (\sqrt{\varepsilon} \delta^{-1} + \varepsilon^{\frac{3}{2}} \delta^{-4} + \delta + \varepsilon^{-1} \delta^3)(1 + \lvert \ln \delta \lvert) \Big).
\end{multline*}

\noindent Using the relation between $u_{-}$ and $u_{-}^{\text{in}, \alpha}$ (Proposition \ref{prop:ingoingprofiles}), we obtain the following estimate in $\widetilde{\Sigma}_{\varepsilon}^k$
\begin{multline*}
    \psi^{\varepsilon}(t,x) = \varepsilon^{- \frac{d}{4}} \text{Exp} \Bigg( \dfrac{i}{\varepsilon} S_{-}^{\flat} + \dfrac{i}{\varepsilon}  S_{0}(t,t^{\flat},z^{\flat}) + \dfrac{i}{\sqrt{\varepsilon}} p_0(t) \cdot y\Bigg) \text{Exp} \Bigg(- i \Lambda^{\varepsilon,\alpha}(s, \eta) - \dfrac{i \alpha}{r \sqrt{\varepsilon}} \eta \cdot \Vec{e}_{\theta} \\ - \dfrac{i\alpha^2}{4r \varepsilon} + \dfrac{i\alpha^2}{2r \varepsilon} \ln \Big( \dfrac{\alpha}{2\sqrt{r \varepsilon}} \Big) + \dfrac{i}{2r} (\eta \cdot \Vec{e}_{\theta})^2 + \dfrac{i\alpha}{r \sqrt{\varepsilon}} \eta \cdot \Vec{e}_{\theta}^{\perp} \ln \Big( \dfrac{\alpha}{2\sqrt{r \varepsilon}} \Big) + \dfrac{i}{2r} (\eta \cdot \Vec{e}_{\theta}^{\perp})^2 \ln \Big( \dfrac{1}{2\sqrt{r \varepsilon}} \Big) - \dfrac{i}{2} \Gamma_1 y \cdot y \Bigg) \\ \times \Vec{Y}_{\flat} \times u_{-}^{\text{in}, \alpha}\Big(y + y^{\varepsilon}(s) \Big) + \mathcal{O} \Big( (\sqrt{\varepsilon} \delta^{-1} + \varepsilon^{\frac{3}{2}} \delta^{-4} + \delta + \varepsilon^{-1} \delta^3)(1 + \lvert \ln \delta \lvert) \Big)
\end{multline*}

\noindent where $y^{\varepsilon}(s) = \mathcal{O} \Big( \sqrt{\varepsilon} s^2 \langle y \rangle \Big)$. Thus, as $u_{-}^{\text{in}, \alpha} \in \schwartz{\R^d}$ and $\delta \ll \dfrac{\delta^2}{\sqrt{\varepsilon}} \ll \dfrac{\delta^3}{\varepsilon}$, we identify that
\begin{multline*}
    u^{\varepsilon}(-s_0,y) = \mathrm{Exp} (-i \Lambda^{\varepsilon, \alpha} (-s_0, \eta(y))) \alpha_2^{\text{in}}(\eta(y)) \Vec{Y}_{\flat} + \, \mathcal{O} \Big( (\sqrt{\varepsilon} \delta^{-1} + \varepsilon^{\frac{3}{2}} \delta^{-4} + \varepsilon^{-1} \delta^3)(1 + \lvert \ln \delta \lvert) \Big) \, \, \text{in} \, \, \widetilde{\Sigma}_{\varepsilon}^k,
\end{multline*}

\noindent where $\alpha_2^{\text{in}}(\eta(y))$ is explicitly stated in \eqref{alphain}. This concludes the proof.
\end{proof}

\subsection{Identification of the outgoing wave packets using the Landau--Zener solution}
\label{subsec:outgoingWP}

\noindent The goal of this subsection is to present the proof of Theorem~\ref{theo:outgoingWP}. To this end, we localize the solution of Equation~\eqref{eqmodel} using a suitably chosen initial data. Then, we can compare it with $u^{\varepsilon}$, the solution of Equation~\eqref{eqUeps}, whose the initial data at time $-s_{0}$ is known thanks to Theorem~\ref{theo:ingoingWP}. We will use Proposition~\ref{prop:soleqmodel}, which contains both the Landau--Zener formula~\eqref{passage} and Equation~\eqref{LZplusin}, to determine the outgoing wave packet at time $t^{\flat} + \delta$.\\

\subsubsection{Comparison of $u^{\varepsilon}$ with a solution to a Landau-Zener model} Thanks to Proposition~\ref{prop:soleqmodel}, we can describe the solution of the Landau--Zener Equation~\eqref{eqmodel} for $\eta(y) \, + \, \frac{\alpha}{\sqrt{\varepsilon}}$ in a compact set. Therefore, we consider $w^{\varepsilon}$ \review{a} solution to Equation \eqref{eqmodel} with for all $y \in \R^d$ \[ w^{\varepsilon}(-s_0,y) = \mathrm{Exp} (-i \Lambda^{\varepsilon, \alpha} (-s_0, \eta(y))) \alpha_2^{\text{in}}(\eta(y)) \Vec{Y}_{\flat}\] as initial condition (the phase $\Lambda^{\varepsilon, \alpha}$ and the function $\alpha_2^{\text{in}}$ are defined in Theorem \ref{theo:ingoingWP}) and a cut-off function $\chi_0 \in \mathcal{C}^{\infty}_{0}(\R^d, [0,1])$ such that \[ \lvert \eta (y) \lvert \, \leqslant \dfrac{R}{2} \text{ when } \dfrac{y}{R} \in \text{supp}(\chi_0).\] Thus, the function \begin{equation} \label{eq:solLZcompact}w_R^{\varepsilon} = \chi_0 \Big(\frac{y}{R} \Big) w^{\varepsilon} \end{equation} is \review{a} solution to Equation \eqref{eqmodel} and has a compact support. Then, we can apply Proposition \ref{prop:soleqmodel} to it. The aim of the following proposition is to compare the solution $u^{\varepsilon}$ to \eqref{eqUeps}, using the function $u^{\varepsilon}(-s_0)$ (given in~\eqref{Uepsminus} in Theorem~\ref{theo:ingoingWP}) as initial condition, with $w_{R}^{\varepsilon}$. This is done in~\cite[Lemma $4.8$]{FGH2021} and may be reused in our framework.

\begin{prop}
\label{prop:comparaison}
Let $k \in \N$. There exists a constant $C > 0$ such that the following holds. For all $\varepsilon \in (0,1)$, $\delta \in (0,1]$, $N_0 \in \N^{*}$, $s \in [- s_0, s_0]$ and $R \geqslant 1$ such that $R^2 \sqrt{\varepsilon} \ll 1$, $R \delta \ll 1$ and $R \varepsilon^2 \delta^{-4} \ll 1$, we have
\begin{equation*}
    u^{\varepsilon}(s) = w^{\varepsilon}_{R}(s) + r^{\varepsilon}(R,\delta)
\end{equation*}
where $r^{\varepsilon}(R,\delta)$ satisfies
\begin{equation}
\label{eq:reste1}
\| r^{\varepsilon}(R,\delta) \|_{\widetilde{\Sigma}_{\varepsilon}^k} \leqslant C( \sqrt{\varepsilon} \delta^{-1} + \varepsilon^{\frac{3}{2}} \delta^{-4} + \delta + R \varepsilon^{-1} \delta^3 + R^{-N_0})(1 + \lvert \ln \delta \lvert).
\end{equation}
\end{prop}

\subsubsection{The outgoing wave packets}

\noindent We now prove Theorems \ref{theo:outgoingWPfctu} and \ref{theo:outgoingWP}, that allows to identify the outgoing wave packets at time $t^{\flat} + \delta$. Here, we analyze the limit $s \to + \infty$.

\begin{proof}[Proof of Theorem ~\ref{theo:outgoingWPfctu}] Following Theorem \ref{theo:ingoingWP} and the construction of the function $w_R^{\varepsilon}$, given in Equation ~\eqref{eq:solLZcompact}, we can identify $\alpha_{1,R}^{\text{in}} = 0$ and $\alpha_{2,R}^{\text{in}} = \chi_0 ( \frac{y}{R} ) \alpha_{2}^{\text{in}}$, that is to say
\begin{multline*}
    \alpha_{2,R}^{\text{in}}(\eta(y)) = \text{Exp} \Bigg( \dfrac{i}{\varepsilon} S_{-}^{\flat} - \dfrac{i\alpha^2}{4r \varepsilon} + \dfrac{i \alpha^2}{2r \varepsilon} \ln\Big(\dfrac{\alpha}{2\sqrt{r\varepsilon}} \Big) + \dfrac{i}{2r} (\eta(y) \cdot \Vec{e}_{\theta})^2 + \dfrac{i \alpha}{r \sqrt{\varepsilon}} \eta(y) \cdot \Vec{e}_{\theta}^{\perp} \ln\Big(\dfrac{\alpha}{2\sqrt{r\varepsilon}} \Big)\\ + \dfrac{i}{2r} (\eta(y) \cdot \Vec{e}_{\theta}^{\perp})^2 \ln\Big(\dfrac{1}{2\sqrt{r\varepsilon}} \Big) - \dfrac{i}{2} \Gamma_1 y \cdot y \Bigg) \text{Exp} \Bigg( - \dfrac{i \alpha}{r \sqrt{\varepsilon}} \eta(y) \cdot \Vec{e}_{\theta} \Bigg) \chi_0 \Big( \frac{y}{R} \Big) \, u_{-}^{\text{in},\alpha}(y).
\end{multline*}
Then, using the transfer formula \eqref{passage}, we define the outgoing function $\begin{pmatrix}
    \alpha_{1,R}^{\text{out}} \\
    \alpha_{2,R}^{\text{out}}
\end{pmatrix}$ that appear\review{s} in Equation \eqref{LZplusin} for $w_R^{\varepsilon}$. Thus, we have an expression of the function $w_R^{\varepsilon}$ when $t = t^{\flat} + \delta = t^{\flat} + s_0 \sqrt{\varepsilon}$. According to Proposition ~\ref{prop:comparaison}, we deduce that $\psi^{\varepsilon}$ satisfies \eqref{newinconnue} at $s = s_0$ with
\begin{multline*}
    u^{\varepsilon}(s_0,y) = \chi_0 \Big( \frac{y}{R} \Big) \times \Big( e^{i \Lambda^{\varepsilon,\alpha}(s_0,\eta(y))} \alpha_{1}^{\text{out}}(\eta(y)) \Vec{Y}_{\flat}^{\perp} + e^{-i \Lambda^{\varepsilon,\alpha}(s_0,\eta(y))} \alpha_{2}^{\text{out}}(\eta(y)) \Vec{Y}_{\flat} \Big)\, \\ + \, r^{\varepsilon}(R,\delta) \, + \, \mathcal{O} \Big( R^3 \sqrt{\varepsilon} \delta^{-1} (1 + \lvert \ln \delta \lvert) \Big) \text{ in } \widetilde{\Sigma}_{\varepsilon}^k,
\end{multline*}
\noindent where the quantity $r^{\varepsilon}(R,\delta)$ is introduced in \eqref{eq:reste1} and the functions $\begin{pmatrix}
    \alpha_{1}^{\text{out}} \\
    \alpha_{2}^{\text{out}}
\end{pmatrix}$ are given in \eqref{eq:relationalphainetout}.
\end{proof}

\begin{proof}[Proof of Theorem ~\ref{theo:outgoingWP}]
\noindent For the remainder of this proof, $a = \mathcal{O}(b)$ means that $\lvert a \lvert \, \leqslant C \lvert b \lvert$ for some constant $C$ that only depends on the parameters $M, r_0$, $k$ and the cut-off $\chi_0$.\\

\noindent In view of the previous theorem, we note that the functions $\alpha_{1,R}^{\text{out}}(\eta(y)) : y \mapsto \chi_0 ( \frac{y}{R} ) \alpha_1^{\text{out}}(\eta(y))$ and $\alpha_{2,R}^{\text{out}}(\eta(y)) : y \mapsto \chi_0 ( \frac{y}{R} ) \alpha_2^{\text{out}}(\eta(y))$ are compactly supported. Thus, there exists a cut-off $\chi : \R^d \longrightarrow [0,1]$ such that $\chi = 1$ on $\text{supp}(\chi_0)$ and \[ \dfrac{y}{R} \in \text{supp}(\chi) \Longrightarrow  \lvert n(y) \lvert := \Big| \diff w(q^{\flat}) y\Big| \leqslant \dfrac{R}{2}.\]

\noindent By construction, for all $y \in \R^d$, $\alpha_{1,R}^{\text{out}}(\eta(y)) = \chi \, \alpha_{1,R}^{\text{out}}(\eta(y))$ and $\alpha_{2,R}^{\text{out}}(\eta(y)) =  \chi \, \alpha_{2,R}^{\text{out}}(\eta(y))$. Thus, Theorem ~\ref{theo:outgoingWPfctu} implies that for $t = t^{\flat} + \delta$ and for all $x \in \R^d$ $$\psi^{\varepsilon}(t,x) = \chi \Big(\psi^{\varepsilon}_{+}(t,x) + \psi^{\varepsilon}_{-}(t,x) \Big) + r^{\varepsilon}(R,\delta) + \mathcal{O} \Big( R^3 \sqrt{\varepsilon} \delta^{-1} (1 + \lvert \ln \delta \lvert) \Big)$$
\noindent in $\Lp{2}{\R^d}{\C^2}$, where the remainder $r^{\varepsilon}(R,\delta)$ is defined in \eqref{eq:reste1} and with
\begin{equation}
\label{psi+}
    \psi^{\varepsilon}_{+}(t,x) = e^{\frac{i}{\varepsilon} S_0(t,t^{\flat},z^{\flat}) + \frac{i}{\varepsilon} p_0(t) \cdot (x - q_0(t))} \Big( e^{-i \Lambda^{\varepsilon,\alpha}(s_0,\eta(y))} \alpha_{2,R}^{\text{out}}(\eta(y)) \Big) \Big\lvert_{y = \frac{x - q_0(t)}{\sqrt{\varepsilon}}} \Vec{Y}_{\flat}, 
\end{equation}

\begin{equation}
\label{psi-}
    \psi^{\varepsilon}_{-}(t,x) = e^{\frac{i}{\varepsilon} S_0(t,t^{\flat},z^{\flat}) + \frac{i}{\varepsilon} p_0(t) \cdot (x - q_0(t))} \Big( e^{i \Lambda^{\varepsilon,\alpha}(s_0,\eta(y))} \alpha_{1,R}^{\text{out}}(\eta(y)) \Big) \Big\lvert_{y = \frac{x - q_0(t)}{\sqrt{\varepsilon}}} \Vec{Y}_{\flat}^{\perp}. 
\end{equation}

\noindent  It remains to see why the functions $\psi^{\varepsilon}_{\pm}$ can be approximated by wave packets associated with the classical trajectories $(\widetilde{q}_{+}, \widetilde{p}_{+})$ and $(q_{-},p_{-})$, respectively. For this, let us study the \review{asymptotic behavior} of the phase and the profiles for $t > t^{\flat}$. According to  the second point of Lemma \ref{prop:changetime}, we have the following estimates
\begin{multline*}
\dfrac{i}{\varepsilon} S_{0}(t;t^{\flat},z^{\flat}) + \dfrac{i}{\sqrt{\varepsilon}} p_0(t) \cdot y = \dfrac{i}{\varepsilon}  S_{-}(t;t^{\flat},z^{\flat}) + \dfrac{i}{\varepsilon} p_{-}(t) \cdot (x - q_{-}(t)) - i \Lambda^{\varepsilon,\alpha}(s, \eta(y)) \\ - i \Phi^{\varepsilon,\alpha}(\eta(y)) + \dfrac{i}{2} G_{\alpha}(t^{\flat} + \sqrt{\varepsilon}s) y \cdot y + \dfrac{i \alpha}{r \sqrt{\varepsilon}} \eta(y) \cdot \Vec{e}_{\theta} + \mathcal{O} \Big( \dfrac{1}{s} \Big) + \mathcal{O} \Big( \sqrt{\varepsilon} s^3 \Big) + \mathcal{O} \Big( \sqrt{\varepsilon} s^2 \lvert y \lvert \Big),
\end{multline*}
and
\begin{multline*}
\dfrac{i}{\varepsilon} S_{0}(t;t^{\flat},z^{\flat}) + \dfrac{i}{\sqrt{\varepsilon}} p_0(t) \cdot y = \dfrac{i}{\varepsilon} \widetilde{S}_{+}(t;t^{\flat},\widetilde{z}^{\flat}) + \dfrac{i}{\varepsilon} \widetilde{p}_{+}(t) \cdot (x - \widetilde{q}_{+}(t)) - \dfrac{i}{\sqrt{\varepsilon}} \delta_{p^{\flat}} \cdot y + i \Lambda^{\varepsilon,\alpha}(s, \eta(y)) \\ + i \Phi^{\varepsilon,\alpha}(\eta(y)) - \dfrac{i}{2} G_{\alpha}(t^{\flat} + \sqrt{\varepsilon}s) y \cdot y - \dfrac{i \alpha}{r \sqrt{\varepsilon}} \eta(y) \cdot \Vec{e}_{\theta} + \mathcal{O} \Big( \dfrac{1}{s} \Big) + \mathcal{O} \Big( \sqrt{\varepsilon} s^3 \Big) + \mathcal{O} \Big( \sqrt{\varepsilon} s^2 \lvert y \lvert \Big),
\end{multline*}

\noindent for $s := \frac{t-t^{\flat}}{\sqrt{\varepsilon}}$ large and positive, where we recall that $y = \frac{x-q_0(t)}{\sqrt{\varepsilon}}$. Now, we are able to see the wave packet structure of the functions $\psi_{\pm}^{\varepsilon}$ at time $t^{\flat} + \delta$. Let us begin to study more precisely the function $\psi_{-}^{\varepsilon}$ defined in \eqref{psi-}. In view of the relation stated above, we have for $t = t^{\flat} + s_0 \sqrt{\varepsilon}$
\begin{multline*}
    \psi^{\varepsilon}_{-}(t,x) = \Vec{Y}_{\flat}^{\perp} \text{Exp} \Bigg( \dfrac{i}{\varepsilon}  S_{-}(t,t^{\flat},z^{\flat}) + \dfrac{i}{\varepsilon} p_{-}(t) \cdot (x - q_{-}(t)) \Bigg) \times \Bigg( \text{Exp} \Big( \dfrac{i}{2} G_{\alpha}(t) y \cdot y \Big) \\ \text{Exp} \Big( - \dfrac{i\alpha^2}{4r \varepsilon} + \dfrac{i \alpha^2}{2r \varepsilon} \ln \Big( \dfrac{\alpha}{2\sqrt{r \varepsilon}} \Big) + \dfrac{i}{2r} (\eta \cdot \Vec{e}_{\theta})^2 + \dfrac{i \alpha}{r \sqrt{\varepsilon}} \eta \cdot \Vec{e}_{\theta}^{\perp} \ln \Big( \dfrac{\alpha}{2\sqrt{r \varepsilon}} \Big) + \dfrac{i}{2r} (\eta \cdot \Vec{e}_{\theta}^{\perp})^2 \ln \Big( \dfrac{1}{2\sqrt{r \varepsilon}} \Big) \\ - \dfrac{i}{2} \Gamma_1 y \cdot y + \dfrac{i \alpha}{r \sqrt{\varepsilon}} \eta \cdot \Vec{e}_{\theta} \Big) \alpha_{1,R}^{\text{out}}(\eta(y)) \Bigg) \Bigg\lvert_{y = \frac{x - q_0(t)}{\sqrt{\varepsilon}}} + \mathcal{O} \Bigg( \dfrac{1}{s} + \sqrt{\varepsilon} s^3 + \sqrt{\varepsilon} s^2 \lvert y \lvert \Bigg)
\end{multline*}

\noindent in $\Lp{2}{\R^d}{\C^2}$. Thanks to the asymptotic expansion of $q_{-}$ close to $t^{\flat}$, we know $$y = \frac{x-q_{0}}{\sqrt{\varepsilon}} = \frac{x-q_{-}}{\sqrt{\varepsilon}} + \mathcal{O} \Big( \sqrt{\varepsilon} s^2 \Big).$$

\noindent Thus, using the regularity of $\alpha_{1,R}^{\text{out}}$ (given in Proposition \ref{prop:soleqmodel}), we deduce $$\alpha_{1,R}^{\text{out}}\Big( \frac{x - q_0(t)}{\sqrt{\varepsilon}} \Big) = \alpha_{1,R}^{\text{out}}\Big( \frac{x - q_-(t)}{\sqrt{\varepsilon}} \Big) + \mathcal{O} \Big( \sqrt{\varepsilon} s^2 \Big).$$

\noindent Moreover, using $\eta := \eta(y) = \eta \Big(\dfrac{x-q_{-}}{\sqrt{\varepsilon}} \Big) + \mathcal{O} \Big( \sqrt{\varepsilon} s^2 \Big)$ and the notation $y_{-} = \dfrac{x - q_{-}}{\sqrt{\varepsilon}}$, we have the following relations
\begin{align*}
    (\eta \cdot \Vec{e}_{\theta})^2 & = (\eta(y_{-}) \cdot \Vec{e}_{\theta})^2 + \mathcal{O} \Big( \sqrt{\varepsilon} s^2 \lvert y_{-} \lvert \Big) + \mathcal{O} \Big( \varepsilon s^4 \Big),\\
    \dfrac{i \alpha}{r \sqrt{\varepsilon}} \eta \cdot \Vec{e}_{\theta} & = \dfrac{i \alpha}{r \sqrt{\varepsilon}} \eta(y_-) \cdot \Vec{e}_{\theta} + \mathcal{O} \Big( \sqrt{\varepsilon} s^3 \Big),\\
    \dfrac{i \alpha}{r \sqrt{\varepsilon}} \eta \cdot \Vec{e}_{\theta}^{\perp} \ln \Big( \dfrac{\alpha}{2\sqrt{r \varepsilon}} \Big) & = \dfrac{i \alpha}{r \sqrt{\varepsilon}} \eta(y_-)  \cdot \Vec{e}_{\theta}^{\perp} \ln \Big( \dfrac{\alpha}{2\sqrt{r \varepsilon}} \Big) + \mathcal{O} \Big( \sqrt{\varepsilon} s^3 \Big).
\end{align*}

\noindent Thanks to the relation $y \cdot y = y_- \cdot y_- + \mathcal{O} \Big( \sqrt{\varepsilon} s^2 (1 + \lvert y_- \lvert) \Big) =  \mathcal{O} \Big(  \sqrt{\varepsilon} s^2 \langle y_{-} \rangle \Big)$ and the fact that $\sqrt{\varepsilon} s^2 \ll \sqrt{\varepsilon} s^3$ for $s \to + \infty$, we can identify a wave packet approximation in $\Lp{2}{\R^d}{\C^2}$ with 
\begin{equation*}
    \psi^{\varepsilon}_{-}(t,x) = e^{\frac{i}{\varepsilon}  S_{-}(t,t^{\flat},z^{\flat})} \mathrm{WP}_{\Phi_-^{t,t^{\flat}}(z^{\flat})}^{\varepsilon} \Big(f^{\varepsilon,\alpha}_{-}(x) \Big) \Vec{Y}_{\flat}^{\perp} + \mathcal{O} \Bigg( \dfrac{1}{s} + \sqrt{\varepsilon} s^3 \langle y_- \rangle \Bigg)
\end{equation*}
\noindent where for all $y \in \R^d$
\begin{equation}
\label{WPmode-}
    f^{\varepsilon,\alpha}_{-}(y) = \text{Exp} \Bigg(\dfrac{i}{2} G_{\alpha}(t^{\flat} + \sqrt{\varepsilon} s_0) y \cdot y \Bigg) \times u_{-,R}^{\text{out}}(y),
\end{equation}
\noindent with
\begin{multline*}
\quad u_{-,R}^{\text{out}}(y) = \text{Exp} \Bigg(-\dfrac{i\alpha^2}{4r \varepsilon} + \dfrac{i \alpha^2}{2r \varepsilon} \ln \Big( \dfrac{\alpha}{2\sqrt{r \varepsilon}} \Big) + \dfrac{i}{2r} (\eta(y) \cdot \Vec{e}_{\theta})^2 + \dfrac{i \alpha}{r \sqrt{\varepsilon}} \eta(y) \cdot \Vec{e}_{\theta}^{\perp} \ln \Big( \dfrac{\alpha}{2\sqrt{r \varepsilon}} \Big)\\ + \dfrac{i}{2r} (\eta(y) \cdot \Vec{e}_{\theta}^{\perp})^2 \ln \Big( \dfrac{1}{2\sqrt{r \varepsilon}} \Big) - \dfrac{i}{2} \Gamma_1 y \cdot y  + \dfrac{i \alpha}{r \sqrt{\varepsilon}} \eta(y) \cdot \Vec{e}_{\theta} \Bigg) \alpha_{1,R}^{\text{out}}(\eta(y)) \in \schwartz{\R^d}.
\end{multline*}

\noindent Since $\alpha_{1,R}^{\text{out}} = \chi_0 (\frac{y}{R}) \alpha_1^{\text{out}}$, using the relation between $\alpha_1^{\text{out}}$ and $\alpha_2^{\text{in}}$ and the explicit expression of $\alpha_2^{\text{in}}$ given in \eqref{alphain}, we obtain
\begin{multline}
\label{eq:lienu_inu_outmoins}
    u_{-,R}^{\text{out}}(y) = - \text{Exp} \Bigg( \dfrac{i}{\varepsilon} S_{-}^{\flat} \Bigg) \times \,  \text{Exp} \Bigg( -\dfrac{i\alpha^2}{2r \varepsilon} + \dfrac{i \alpha^2}{r \varepsilon} \ln \Big( \dfrac{\alpha}{2\sqrt{r \varepsilon}} \Big) + \dfrac{i}{r} (\eta(y) \cdot \Vec{e}_{\theta})^2 \\ + \dfrac{2i \alpha}{r \sqrt{\varepsilon}} \eta(y) \cdot \Vec{e}_{\theta}^{\perp} \ln \Big( \dfrac{\alpha}{2\sqrt{r \varepsilon}} \Big) + \dfrac{i}{r} (\eta(y) \cdot \Vec{e}_{\theta}^{\perp})^2 \ln \Big( \dfrac{1}{2\sqrt{r \varepsilon}} \Big) - i \Gamma_1 y \cdot y  \Bigg) \overline{b} \Big( \review{\dfrac{\eta \cdot \Vec{e}_{\theta}^{\perp}}{\sqrt{r}} + \dfrac{\alpha}{\sqrt{r\varepsilon}}} \Big) \chi_0 \Big(\dfrac{y}{R} \Big) \, u_{-}^{\text{in},\alpha}(y).
\end{multline}

\noindent Since the cut-off $\chi$ has a compact support, we can write that for $t \in [t^{\flat} + \frac{\delta}{2}, t^{\flat} + \delta]$ $$\chi \times \mathcal{O} \Big( \frac{1}{s} + \sqrt{\varepsilon} s^3 \langle y_- \rangle \Big) = \mathcal{O} \Big( \sqrt{\varepsilon} \delta^{-1} + R \varepsilon^{-1} \delta^3 \Big).$$

\noindent Thus, at time $t = t^{\flat} + \sqrt{\varepsilon} s_0$, $\chi \, \psi^{\varepsilon}_{-}$ can be approximated by a wave packet in $\Lp{2}{\R^d}{\C^2}$ as follows $$\chi \, \psi^{\varepsilon}_{-}(t,x) = e^{\frac{i}{\varepsilon} S_{-}(t,t^{\flat},z^{\flat})} \mathrm{WP}_{\Phi_-^{t,t^{\flat}}(z^{\flat})}^{\varepsilon} \Big(e^{\frac{i}{2} G_{\alpha}(t^{\flat} + \sqrt{\varepsilon} s_0) x \cdot x} u_{-,R}^{\text{out}}(x) \Big) \Vec{Y}_{\flat}^{\perp} + \mathcal{O} \Big( \sqrt{\varepsilon} \delta^{-1} + \varepsilon^{-1} \delta^3 \Big).$$

\noindent With regard to the plus-mode, using the expression \eqref{psi+} stated above and the same reasoning, we have for $t = t^{\flat} + s_0 \sqrt{\varepsilon}$
\begin{multline*}
    \psi^{\varepsilon}_{+}(t,x) = \Vec{Y}_{\flat} \text{Exp} \Bigg( \dfrac{i}{\varepsilon}  \widetilde{S}_{+}(t,t^{\flat},z^{\flat}) + \dfrac{i}{\varepsilon} \widetilde{p}_{+}(t) \cdot (x - \widetilde{q}_{+}(t)) \Bigg) \times \Bigg( \text{Exp} \Big( - \dfrac{i}{2} G_{\alpha}(t) y \cdot y \Big) \times \, \text{Exp} \Big(  -\dfrac{i}{\sqrt{\varepsilon}}\delta_{p^{\flat}} \cdot y + \dfrac{i\alpha^2}{4r \varepsilon} \\ - \dfrac{i \alpha^2}{2r \varepsilon} \ln \Big( \dfrac{\alpha}{2\sqrt{r \varepsilon}} \Big) - \dfrac{i}{2r} (\eta \cdot \Vec{e}_{\theta})^2 - \dfrac{i \alpha}{r \sqrt{\varepsilon}} \eta \cdot \Vec{e}_{\theta}^{\perp} \ln \Big( \dfrac{\alpha}{2\sqrt{r \varepsilon}} \Big) - \dfrac{i}{2r} (\eta \cdot \Vec{e}_{\theta}^{\perp})^2 \ln \Big( \dfrac{1}{2\sqrt{r \varepsilon}} \Big) \\ + \dfrac{i}{2} \Gamma_1 y \cdot y - \dfrac{i \alpha}{r \sqrt{\varepsilon}} \eta \cdot \Vec{e}_{\theta} \Big) \alpha_{2,R}^{\text{out}}(\eta(y)) \Bigg) \Bigg\lvert_{y = \frac{x - q_0(t)}{\sqrt{\varepsilon}}} + \mathcal{O} \Bigg( \dfrac{1}{s} + \sqrt{\varepsilon} s^3 + \sqrt{\varepsilon} s^2 \lvert y \lvert \Bigg).
\end{multline*}
\noindent We set for all $y \in \R^d$
\begin{multline*}
    u_{+,R}^{\text{out}}(y) = \text{Exp} \Big(  -\dfrac{i}{\sqrt{\varepsilon}}\delta_{p^{\flat}} \cdot y + \dfrac{i\alpha^2}{4r \varepsilon}  - \dfrac{i \alpha^2}{2r \varepsilon} \ln \Big( \dfrac{\alpha}{2\sqrt{r \varepsilon}} \Big) - \dfrac{i}{2r} (\eta \cdot \Vec{e}_{\theta})^2 \\ - \dfrac{i \alpha}{r \sqrt{\varepsilon}} \eta \cdot \Vec{e}_{\theta}^{\perp} \ln \Big( \dfrac{\alpha}{2\sqrt{r \varepsilon}} \Big) - \dfrac{i}{2r} (\eta \cdot \Vec{e}_{\theta}^{\perp})^2 \ln \Big( \dfrac{1}{2\sqrt{r \varepsilon}} \Big) + \dfrac{i}{2} \Gamma_1 y \cdot y - \dfrac{i \alpha}{r \sqrt{\varepsilon}} \eta \cdot \Vec{e}_{\theta} \Big) \alpha_{2,R}^{\text{out}}(\eta(y)),
\end{multline*}
\noindent that is to say
\begin{equation*}
    u_{+,R}^{\text{out}}(y) = \text{Exp} \Big( \dfrac{i}{\varepsilon} S_{-}^{\flat} \Big) \mathrm{Exp} \Big( -\dfrac{i}{\sqrt{\varepsilon}}\delta_{p^{\flat}} \cdot y - \dfrac{2i\alpha}{r\sqrt{\varepsilon}} \eta(y) \cdot \Vec{e}_{\theta} \Big) a \Big( \review{\dfrac{\eta \cdot \Vec{e}_{\theta}^{\perp}}{\sqrt{r}} + \dfrac{\alpha}{\sqrt{r\varepsilon}}} \Big) \chi_0 \Big( \dfrac{y}{R} \Big) \, u_{-}^{\text{in},\alpha}(y).
\end{equation*}
\noindent To ensure sufficient regularity of $u_{+,R}^{\text{out}}$ in the spaces $\Sigma^k$ (we want that the derivatives and momenta are uniformly bounded), the condition \eqref{eq:ourdrift} on the drift $\delta_{p^{\flat}}$ has to be satisfied (it is at this stage that the drift $\delta_{p^{\flat}}$ is determined through our computations). Using $\delta_{p^{\flat}} = - \dfrac{2 \alpha}{r} \diff \transpose{w(q^{\flat})} \Vec{e}_{\theta}$ , we have 
\begin{equation}
\label{eq:lienu_inu_outplus}
     u_{+,R}^{\text{out}}(y) = \text{Exp} \Big( \dfrac{i}{\varepsilon} S_{-}^{\flat} \Big) a \Big( \review{\dfrac{\eta \cdot \Vec{e}_{\theta}^{\perp}}{\sqrt{r}} + \dfrac{\alpha}{\sqrt{r\varepsilon}}} \Big) \chi_0 \Big( \dfrac{y}{R} \Big) \, u_{-}^{\text{in},\alpha}(y).
\end{equation}
\noindent Therefore, at time $t = t^{\flat} + \sqrt{\varepsilon} s_0$, for the same reasons as before, $\chi \, \psi^{\varepsilon}_{+}$ can be approximated by a wave packet in $\Lp{2}{\R^d}{\C^2}$ as follows $$\chi \, \psi^{\varepsilon}_{+}(t,x) = e^{\frac{i}{\varepsilon}  \widetilde{S}_{+}(t,t^{\flat},z^{\flat})} \mathrm{WP}_{\widetilde{\Phi}_+^{t,t^{\flat}}(\widetilde{z}^{\flat})}^{\varepsilon} \Big(f^{\varepsilon,\alpha}_{+}(x) \Big) \Vec{Y}_{\flat} + \mathcal{O} \Bigg( \sqrt{\varepsilon} \delta^{-1} + R \varepsilon^{-1} \delta^3 \Bigg)$$

\noindent with for all $y \in \R^d$
\begin{equation}
\label{WPmode+}
    f^{\varepsilon,\alpha}_{+}(y) = \text{Exp} \Bigg(- \dfrac{i}{2} G_{\alpha}(t^{\flat} + \sqrt{\varepsilon} s_0) y \cdot y \Bigg) u_{+,R}^{\text{out}}(y).
\end{equation}

\noindent According to the explicit formulas \eqref{WPmode-} and \eqref{WPmode+}, we have that
\( f^{\varepsilon,\alpha}_{\pm} = \chi_0 \, \times \, \varphi_{\pm}^{\varepsilon,\alpha} \) with
\begin{equation}
\label{eq:phiplusmoins}
\varphi_{-}^{\varepsilon,\alpha} = \text{Exp} \Bigg( \dfrac{i}{2} G_{\alpha}(t^{\flat} + \sqrt{\varepsilon} s_0) y \cdot y \Bigg) u_{-}^{\text{out},\alpha} , \, \, \varphi_{+}^{\varepsilon,\alpha} = \text{Exp} \Bigg( - \dfrac{i}{2} G_{\alpha}(t^{\flat} + \sqrt{\varepsilon} s_0) y \cdot y \Bigg) \widetilde{u}_{+}^{\text{out},\alpha}
\end{equation}
\noindent where the functions $(\widetilde{u}_{+}^{\text{out},\alpha}, u_{-}^{\text{out},\alpha})$ introduced just above are related to $(\widetilde{u}_{+,R}^{\text{out},\alpha}, u_{-,R}^{\text{out},\alpha})$ as follows
\begin{equation}
    \label{eq:lienuinoutavecsuppcompact}
    (\widetilde{u}_{+,R}^{\text{out},\alpha}, u_{-,R}^{\text{out},\alpha}) = \chi_0 \Big(\dfrac{y}{R} \Big)(\widetilde{u}_{+}^{\text{out},\alpha}, u_{-}^{\text{out},\alpha}).
\end{equation}
Since the function $\chi_0$ has a compact support and $\varphi_{\pm}^{\varepsilon,\alpha} \in \schwartz{\R^d}$, we deduce that \[ f^{\varepsilon,\alpha}_{\pm} = \varphi_{\pm}^{\varepsilon,\alpha} + \mathcal{O} ( R^{-N_0} ).\] Finally, using the fact that $\sqrt{\varepsilon} \delta^{-1} + R \varepsilon^{-1} \delta^3 = \mathcal{O} \Big( (\sqrt{\varepsilon} \delta^{-1} +R \varepsilon^{-1} \delta^3)(1 + \lvert \ln \delta \lvert)\Big)$ and the definition of $r^{\varepsilon}(R,\delta)$ given in \eqref{eq:reste1}, we obtain 
\begin{multline*}
\psi^{\varepsilon}(t,x) = e^{\frac{i}{\varepsilon}  S_{-}(t,t^{\flat},z^{\flat})} \mathrm{WP}_{\Phi_-^{t,t^{\flat}}(z^{\flat})}^{\varepsilon} \Big(\varphi^{\varepsilon,\alpha}_{-}(x) \Big) \Vec{Y}_{\flat}^{\perp} + e^{\frac{i}{\varepsilon}  \widetilde{S}_{+}(t,t^{\flat},z^{\flat})} \mathrm{WP}_{\widetilde{\Phi}_+^{t,t^{\flat}}(\widetilde{z}^{\flat})}^{\varepsilon} \Big(\varphi^{\varepsilon,\alpha}_{+}(x) \Big) \Vec{Y}_{\flat} \\ + \mathcal{O} \Big( (\sqrt{\varepsilon} \delta^{-1} + \varepsilon^{\frac{3}{2}} \delta^{-4} + \delta + R \varepsilon^{-1} \delta^3 + R^3 \sqrt{\varepsilon} \delta^{-1} + \varepsilon^{-1} \delta^3 + R^{-N_0})(1 + \lvert \ln \delta \lvert) \Big)
\end{multline*}
where the functions $\varphi^{\varepsilon,\alpha}_{\pm}$ are given in \eqref{eq:phiplusmoins}. The transition relation given in \eqref{coro:transitioninandout} is obtained directly by expressing the relations \eqref{eq:lienu_inu_outmoins} and \eqref{eq:lienu_inu_outplus} in matrix form, using \eqref{eq:lienuinoutavecsuppcompact}.
\end{proof}

\noindent It remains to prove Theorem \ref{theo:main}.

\begin{proof}[Proof of Theorem \ref{theo:main}] The adiabatic theorem and the construction of the parameters of the wave packets for $t > t^{\flat}$ (Proposition \ref{prop:Vecaftercrossing} for the vectors, Proposition \ref{prop:outgoingprofiles} for the profiles) allows us to propagate the solution for $t \geqslant t^{\flat} + \delta$ with the initial condition at time $t^{\flat} + \delta$ given by \eqref{outgoingWP}. Then
    \begin{equation*}
         \psi^{\varepsilon}(t) = \Vec{Y}_{-}(t) v_{-}^{\varepsilon}(t) + \Vec{\widetilde{Y}_{+}}(t) \widetilde{v}_{+}^{\varepsilon}(t) \, + \, \mathcal{O} \Big( (\sqrt{\varepsilon} \delta^{-1} + \varepsilon^{\frac{3}{2}} \delta^{-4} + \delta + R \varepsilon^{-1} \delta^3 + R^3 \sqrt{\varepsilon} \delta^{-1} + \varepsilon^{-1} \delta^3 + R^{-N_0})(1 + \lvert \ln \delta \lvert) \Big) 
    \end{equation*}
where all the parameters of waves packets are defined as in Theorem \ref{theo:main}. To conclude and obtain the remaining part stated in Theorem \ref{theo:outgoingWP}, it suffices to choose $N_0$ and the parameters $\delta$ and $R$ as suitable functions of $\varepsilon$. First, we write $\delta = \varepsilon^{\beta_1}$ with $\beta_1 > 0$. In order to have $\varepsilon^{\frac{3}{2}} \delta^{-4} \ll 1$ and $\varepsilon^{-1} \delta^3 \ll 1$, we need $\beta_1 \in ( \frac{1}{3}, \frac{3}{8} )$. An optimal choice, that is, having $\varepsilon^{\frac{3}{2}} \delta^{-4} = \varepsilon^{-1} \delta^3$, is to take $\beta_1 = \dfrac{5}{14}$. We must then choose $R$ so that the assumptions of Proposition \ref{prop:comparaison} are satisfied, and such that the $R$-dependent remainder terms are bounded by $\varepsilon^{\frac{1}{14}}$. We write $R = \varepsilon^{-\beta_2}$ with $\beta_2 > 0$. Since $\varepsilon^{-1} \delta^3 = \varepsilon^{\frac{1}{14}}$, we need $\beta_2 < \dfrac{1}{14}$. Then, we require that $N_0$ be chosen sufficiently large such that $R^{-N_0} = \mathcal{O} \big( \varepsilon^{\frac{1}{14}} \big)$, $R^3 \sqrt{\varepsilon} \delta^{-1} = \mathcal{O} \big( \varepsilon^{\frac{1}{14}} \big)$. This can be achieved if $N_0 \geqslant \dfrac{1}{14\beta_2}$ and $\beta_2 < \dfrac{1}{42}$. We are thus left with an approximation of order $\mathcal{O} \big( \varepsilon^{\frac{1}{14} - \beta} \big)$ where $\beta \in (0, \frac{1}{42} )$. Moreover, we have $\alpha_0 \leqslant C \varepsilon^{\frac{1}{2} - \beta}.$
\end{proof}

%%%%%%%%%%%%%%%%%%%%%%%%%%%%%%%%%%%%%%%%%%%%%%%%%%%%%%%%%%%%
\appendix

\section{Notations and computations}

\subsection{Notations}
\label{appendixA}

\noindent This appendix summarizes all the notations used throughout the document.\\

\noindent $\bullet$ For $\alpha = (\alpha_1, \cdots, \alpha_d) \in \N^d$ a multi-index, $\lvert \alpha \lvert$ is its length, defined by: $\lvert \alpha \lvert \, = \, \alpha_1 + \cdots + \alpha_d$. For $x = (x_1, \cdots, x_d) \in \R^d$, we write $x^{\alpha}$ for $x_1^{\alpha_1} \cdots x_d^{\alpha_d}$.

\noindent $\bullet$ $\partial_t$ is the partial derivative with the respect to time, $\partial_x^{\alpha}$ the partial derivative with the respect to space of order $\alpha$. We denote $\partial_i$ or $\partial_{x_i}$ the partial derivative with the respect to the $i$-th space coordinate. $\Delta$ \review{corresponds to the Laplacian}, $\nabla$ to the gradient and $\text{Hess}$ to the Hessian matrix. We will use the notation $D = \dfrac{1}{i} \nabla$. 

\noindent $\bullet$ For $x \in \R^d$, $\langle x \rangle = \Big( 1 + \lvert x \lvert^2 \Big)^{\frac{1}{2}}$.

\noindent $\bullet$ We use the notation $a\ll b$ to mean that $a = f(\varepsilon)b$ and $\lim\limits_{\varepsilon \to 0} f(\varepsilon) = 0.$

\noindent $\bullet$ We denote $\schwartz{\R^d} := \mathscr{S}(\R^d,\C)$ the Schwartz space which is the functional space of all smooth functions whose derivatives are rapidly deacreasing. 

\noindent $\bullet$ For $E$ a finite-dimensional normed vector space, we denote $\norm{\cdot}{E}$ any norm in the space $E$.

\noindent $\bullet$ For $\xi \in \R^d$, $\xi \cdot \nabla$ denotes the operator $\xi \cdot \nabla = \sum\limits_{i=1}^{d} \xi_{i} \partial_{i}$.

\noindent $\bullet$ For a vector $V = \begin{pmatrix}
    v_1 \\
    v_2
\end{pmatrix} \in \R^2$, we associate the vector $V^{\bot} = \begin{pmatrix}
    - v_2 \\
    v_1
\end{pmatrix}$. If $U = \begin{pmatrix}
    u_1 \\
    u_2
\end{pmatrix} \in \R^2$, we set $U \wedge V = U^{\bot} \cdot V = u_1 v_2 - u_2 v_1$ and $U \otimes V = \begin{pmatrix}
    u_1 v_1 & u_1 v_2 \\
    u_2 v_1 & u_2 v_2
\end{pmatrix} \in \R^{2 \times 2}$.

\noindent $\bullet$ For $\phi \in \R$, we denote $\Vec{e}_{\phi}$ the vector of $\R^2$ defined by $\Vec{e}_{\phi} = \begin{pmatrix}
    \cos(\phi) \\
    \sin(\phi)
\end{pmatrix}.$

\noindent $\bullet$ We associate to the vector $w = \begin{pmatrix}
    w_1 \\
    w_2
\end{pmatrix} \in \R^2$ the matrix $A(w) = \begin{pmatrix}
    w_1 & w_2 \\
    w_2 & - w_1
\end{pmatrix} \in \R^{2 \times 2}$.

\noindent $\bullet$ Let $(n,m)$ be in $\N^2$. For any matrix $M$ belonging to $\R^{n \times m}$, the matrix $\transpose{M}$ is the transpose of the matrix $M$ and belongs to $\R^{m \times n}$. For all $n \in \N$, we denote $I_n$ the identity matrix of $\R^{n \times n}$.

\noindent $\bullet$ We recall the Gamma function is defined for $z \in \C$ by
\begin{equation}
\label{eq:Gammafunction}
    \Gamma(z) = \ds\int_0^{+ \infty} t^{z-1} e^{-t} \diff t. \tag{A1}
\end{equation}

\subsection{Computations}
\label{appendixB}

\subsubsection{Antiderivatives}

\noindent In the construction of the profiles of the wave packets, we need to control certain integrals, independently of the parameter $\alpha$. To this end, several integrals are computed in the proofs presented in Section \ref{sec:profiles}. Here is a summary of all the useful antiderivatives.

\begin{lemma} Let $C$ be a real number. We have the following antiderivatives
\begin{enumerate}
\item $\ds\int \dfrac{\diff u}{\sqrt{1 + u^2}} = \argsh{x} + C$,
\item $\ds\int \dfrac{u}{\sqrt{1 + u^2}} \, \diff u = \sqrt{1 + x^2} - 1 + C$,
\item $\ds\int \dfrac{ \diff u}{(1 + u^2)^{\frac{3}{2}}} =  \dfrac{x}{\sqrt{1 + x^2}} + C$,
\item $\ds\int \dfrac{u}{(1 + u^2)^{\frac{3}{2}}} \, \diff u = - \dfrac{1}{\sqrt{1 + x^2}} + C$,
\item $\ds\int \sqrt{1 + u^2} \, \diff u = \dfrac{1}{2} \Big( x \sqrt{1 + x^2} + \argsh{x} \Big) + C$,
\item $\ds\int \argsh{u} \, \diff u = x \argsh{x} - \sqrt{1 + x^2} - 1 + C$,
\item $\ds\int \ln \Big( u + \sqrt{1 + u^2} \Big) \diff u = x \ln \Big( x + \sqrt{1 + x^2} \Big) - \sqrt{1 + x^2} + C.$

\end{enumerate}
\end{lemma}

\subsubsection{Asymptotic expansions}

\noindent First, we recall the Taylor expansion of the function $u \mapsto \dfrac{1}{\sqrt{1+u}}$.

\begin{lemma}
\label{lem:estim1}
    There exists $C > 0$ such that for all $u \in \, \Bigg[-\dfrac{1}{2},\dfrac{1}{2} \Bigg]$, \[ \dfrac{1}{\sqrt{1+u}} = 1 + r \quad \text{with } \lvert r \lvert \, \leqslant C \lvert u \lvert .\]
\end{lemma}

\noindent When we study the passage of wave packets in the gap region (Section \ref{sec:passage}), a time rescaling is introduced (see \eqref{newtime}). Therefore, essentially in Subsection \ref{subsec:newscale}, we require the use of asymptotic expansions when $s \to \pm \infty$. The following lemma summarizes those that have been used.

\begin{lemma}
\label{usefulTE}
Let $\varepsilon \in (0,1)$, $\alpha \geqslant 0$ and $R \geqslant 1$. We assume that $\alpha \leqslant \dfrac{R}{2} \sqrt{\varepsilon}$.

\noindent We have the following Taylor expansion when $s \longrightarrow \pm \infty$
    \begin{enumerate}
    \item $\dfrac{1}{\sqrt{\alpha^2 + r^2 s^2\varepsilon }} = \dfrac{1}{r \sqrt{\varepsilon} \lvert s \lvert} + \mathcal{O} \Bigg( \dfrac{R^2}{s^3} \Bigg),$
    \item $\dfrac{is}{2\sqrt{\varepsilon}} \sqrt{\alpha^2 + r^2 s^2 \varepsilon} = i \, \sgn (s) \Bigg( \dfrac{s^2 r}{2} + \dfrac{\alpha^2}{4r \varepsilon} \Bigg) + \mathcal{O} \Bigg( \dfrac{R^4}{s^2} \Bigg),$
    \item $\dfrac{i}{r \sqrt{\varepsilon}} \eta \cdot \Vec{e}_{\theta} \sqrt{\alpha^2 + r^2 s^2 \varepsilon} = \sgn (s) \times i \eta \cdot \Vec{e}_{\theta} s + \mathcal{O} \Bigg( \dfrac{R^2}{s} \Bigg),$
    \item $\dfrac{i \alpha^2}{2r \varepsilon} \Argsh{\dfrac{r}{\alpha}\sqrt{\varepsilon}s} = \sgn (s) \dfrac{i \alpha^2}{2r \varepsilon} \ln \Big( \dfrac{2r \sqrt{\varepsilon} \lvert s \lvert}{\alpha} \Big) + \mathcal{O} \Bigg( \dfrac{R^4}{s^2} \Bigg),$
    \item $\dfrac{i \alpha}{r \sqrt{\varepsilon}} \eta \cdot \Vec{e}_{\theta}^{\perp} \Argsh{\dfrac{r}{\alpha}\sqrt{\varepsilon}s} = \sgn (s) \dfrac{i \alpha}{r \sqrt{\varepsilon}} \eta \cdot \Vec{e}_{\theta}^{\perp} \ln \Big( \dfrac{2r \sqrt{\varepsilon} \lvert s \lvert}{\alpha} \Big) + \mathcal{O} \Bigg( \dfrac{R^3}{s^2} \Bigg),$
    \item $h_{\alpha}(t^{\flat} + \sqrt{\varepsilon} s) = \ln (2 r \sqrt{\varepsilon} \lvert s \lvert) + \mathcal{O} \Bigg( \dfrac{R^2}{s} \Bigg)$ where $h_{\alpha}$ is defined in \eqref{eq:h_alpha},  \item $G_{\alpha}(t^{\flat} + \sqrt{\varepsilon} s) = \Gamma_0 \ln\Big( 2r\sqrt{\varepsilon} \lvert s \lvert \Big) + \Gamma_1 + \mathcal{O} \Bigg( \dfrac{R^2}{s} \Bigg)$ where the matrix $\Gamma_1$ is given in \eqref{matGamma1} and the function $G_{\alpha}$ is defined in \eqref{def:phase}.
\end{enumerate} 
\end{lemma}

\subsection{The Landau--Zener problem} In this subsection, we provide some computational details to explain how the system \eqref{eqmodel} can be turn into a Landau--Zener problem of the form \eqref{eqLZ}. This is the goal of the following lemma.

\begin{lemma}
\label{lem:LZproblem}
    The problem \eqref{eqmodel} \review{can be turn} into the Landau--Zener problem \eqref{eqLZ} with $z := z(\eta)$ and $\eta$ a parameter of $\C^2$
    \begin{equation*}
    \left\{
        \begin{array}{r c l l}
        \dfrac{1}{i} \partial_s u_{\text{LZ}}^{\varepsilon}(s,\eta) & = & \begin{pmatrix}
            s + z_1(\eta) & z_2(\eta) \\
            z_2(\eta) & - s - z_1(\eta)
        \end{pmatrix} u_{\text{LZ}}^{\varepsilon}(s,\eta) & \text{for all} \, \, \, (s,\eta) \in \R \times \R^2\\
        u_{\text{LZ}}^{\varepsilon}(0,\eta) & = & \mathcal{R}(\phi)^{-1} v_0^{\varepsilon}(\sqrt{r} \eta) &  \text{for all} \, \, \, \eta \in \R^2,
        \end{array}
    \right.
\end{equation*}

\noindent with $\phi \in \R$ such that $\Vec{e}_{\phi} = - \Vec{e}_{\theta}$, $\mathcal{R}(\phi)$ is the rotation matrix defined in \eqref{eqmatrot} and $v_0^{\varepsilon}$ the initial condition of the system \eqref{eqmodel}.
\end{lemma}

\begin{proof}[Proof of Lemma ~\ref{lem:LZproblem}]
\noindent Computations give
    \begin{align*}
        \dfrac{1}{i} \partial_s u_{\text{LZ}}^{\varepsilon}(s, \eta) & = \dfrac{1}{i \sqrt{r}} \mathcal{R}(\phi)^{-1} \partial_s v^\varepsilon \Big( \frac{s}{\sqrt{r}} , \sqrt{r} \eta \Big)\\
        & = - \dfrac{1}{\sqrt{r}} \mathcal{R}(\phi)^{-1} A \Big( s\sqrt{r} \Vec{e}_{\theta} + \eta \sqrt{r}  + \dfrac{\alpha \Vec{e}_{\theta}^{\perp}}{\sqrt{\varepsilon}} \Big) v^\varepsilon \Big( \frac{s}{\sqrt{r}} , \sqrt{r} \eta \Big)\\
        & = \begin{pmatrix}
            s + \Vec{e}_{\theta} \cdot \eta & \Vec{e}_{\theta}^{\perp} \cdot \eta + \dfrac{\alpha}{\sqrt{r \varepsilon}} \\
            \Vec{e}_{\theta}^{\perp} \cdot \eta + \dfrac{\alpha}{\sqrt{r \varepsilon}} & - s - \Vec{e}_{\theta} \cdot \eta
        \end{pmatrix} \mathcal{R}(\phi)^{-1} v^\varepsilon \Big( \frac{s}{\sqrt{r}} , \sqrt{r} \eta \Big)\\
        & = \begin{pmatrix}
            s + z_1(\eta) & z_2(\eta) \\
            z_2(\eta) & -s - z_1(\eta) 
        \end{pmatrix} u_{\text{LZ}}^{\varepsilon}(s, \eta)
        \end{align*}
        
\noindent with $z_1(\eta) = \Vec{e}_{\theta} \cdot \eta$ and $z_2(\eta) = \Vec{e}_{\theta}^{\perp} \cdot \eta + \dfrac{\alpha}{\sqrt{r \varepsilon}}$ and the lemma is established.
\end{proof}

\section{Transfer formula for the other mode}
\label{appendix:autremode}

\noindent In this appendix, we outline the necessary adjustments to certain steps of the proof that would be required if the initial wave packet were assumed to be on the plus-mode. Then, in the following, we consider \( z_0 \in \mathcal{A}_{+}(M,\alpha_0,r_0,\delta,\mathrm{I})\) and we assume that the initial data $\psi_0^{\varepsilon}$ of the system \eqref{equation} is given for all $x \in \R^d$ by
\begin{equation}
\label{IDannexe}
    \psi_{0}^{\varepsilon}(x) = \Vec{Y}_{1} \, \mathrm{WP}_{z_0}^{\varepsilon} \varphi(x)
    \tag{B1}
\end{equation}
where $\mathrm{WP}_{z_0}^{\varepsilon}$ is a wave packet associated with the point $z_0 \in \R^{2d} \setminus \Upsilon$, a function $\varphi \in \schwartz{\R^d}$ and $\Vec{Y}_{1} \in \R^2$ a normalized eigenvector of $V(q_0)$ for the plus-mode ($V(q_0) \Vec{Y}_{1} = \lambda_{+}(q_0) \Vec{Y}_{1}$).\\

\noindent According to Definition \ref{defi:initialdata}, there exist a time \( t^{\flat} \in \mathrm{I} \) and a point $z^{\flat} \in \Sigma_{\text{nd}}$ such that $[t^{\flat} - \delta,\, t^{\flat} + \delta] \subset \mathrm{I}$ and
\[ z^{\flat} := \Phi_{+}^{t^{\flat}, t_0}(z_0) \in \Sigma_{\mathrm{nd}}, \quad \Phi_{+}^{t,t_0}(z_0) \notin \Sigma \text{ for } t \in \mathrm{I} \setminus \lbrace t^{\flat}\rbrace, \quad \alpha := |w(q^{\flat})| \leqslant \alpha_0 \quad \text{and} \quad r := \left| \mathrm{d}w(q^{\flat})\, p^{\flat} \right| > r_0.\]

\subsection{Drift}

\noindent When we start from the plus-mode, the conservation of energy can be expressed as follows
\begin{equation*}
    h_{-}(\widetilde{z}^{\flat}) = h_{+}(z^{\flat}) + \mathcal{O}(\varepsilon),
\end{equation*}
\noindent where scalar Hamiltonians $h_{\pm}$ are defined in \eqref{eq:scalarhamiltonien} and $\widetilde{z}^{\flat} = (q^{\flat}, p^{\flat} + \delta_{p^{\flat}})$. This leads to the following conditions on $\delta_{p^{\flat}}$
\begin{equation*}
    p^{\flat} \cdot \delta_{p^{\flat}} = 2\alpha \quad \text{and} \quad \delta_{p^{\flat}} = \mathcal{O}(\sqrt{\varepsilon}).
\end{equation*}

\noindent At this stage, a change of sign can already be observed with respect to the conditions specified in Equation \eqref{eq:defdrift}. We have previously established that our drift term is justified by Lemma \ref{prop:changetime}. In order to restate this lemma for the case where the transition occurs from the plus-mode to the minus-mode, we need to consider
\begin{itemize}[leftmargin=*, labelindent=0pt]
    \item the classical quantities $z_{+}$ and $S_+$ defined respectively in \eqref{eqclassicaltraj} and \eqref{Actions},
    \item the ``drifted" classical trajectory $\widetilde{z}_{-}$ solution to the following system
    \begin{equation}
    \label{eq:drifttrajannexe}
    \left\lbrace 
    \begin{array}{cll}
        \dot{\widetilde{q}}_{-} & = & \widetilde{p}_{-} \\
        \dot{\widetilde{p}}_{-} & = & -\nabla \lambda_{-}(\widetilde{q}_{-})\\
        (\widetilde{q}_{-}(t^{\flat}),\widetilde{p}_{-}(t^{\flat})) & = & \widetilde{z}_{\flat}
    \end{array}
    \right.
    \tag{B2}
    \end{equation}
    \item the associated ``drifted" action $\widetilde{S}_-$ defined by \begin{equation}
        \label{Actionstildeannexe}
        \widetilde{S}_{-}(t;t^{\flat},\widetilde{z}_{\flat}) = \ds\int_{t_0}^t (\lvert \widetilde{p}_{-}(s) \lvert^2 \, - \, h_{-}(\widetilde{z}_{-}(s))) \, \diff s.
        \tag{B3}
    \end{equation}
\end{itemize}
\noindent Therefore, we have the following lemma.

\begin{lemma} We consider \( M \) and \( r_0 \) two positive real numbers. There exists a constant $C > 0$ such that the following holds.

\smallskip

\noindent For all $(\varepsilon,\delta, \alpha_0,R) \in (0,1) \times (0,1] \times \R_{+} \times [1, + \infty [ \) such that $\sqrt{\varepsilon} \delta^{-1} \leqslant 1$ and $\alpha_0 \leqslant \dfrac{R}{2} \sqrt{\varepsilon}$, for all compact interval\review{s} $\mathrm{I} \subset \R$ with $t_0 = \text{min}(\mathrm{I})$, for all $z_0 \in \mathcal{A}_{+}(M,\alpha_0,r_0,\delta,\mathrm{I})$ \review{, let \[ t = t^{\flat} + s \sqrt{\varepsilon}, \quad  y = \frac{x - q_0(t)}{\sqrt{\varepsilon}} \quad \text{and} \quad \eta(y) = \diff w(q^{\flat})y{\sqrt{\varepsilon}}. \] Then, the following pointwise estimates hold.}
\begin{enumerate}[leftmargin=*, labelindent=0pt]
    \item If $s \to - \infty$ (i.e. $\varepsilon \to 0$, $\lvert t - t^{\flat} \lvert \, \leqslant \delta$ and $t < t^{\flat}$)\review{, we have}
\begin{multline*}
\dfrac{i}{\varepsilon}  S_{+}(t;t_0,z_0) + \dfrac{i}{\varepsilon} p_{+}(t) \cdot (x - q_{+}(t)) = \dfrac{i}{\varepsilon} S_{+}(t^{\flat};t_0,z_0) + \dfrac{i}{\varepsilon} S_{0}(t;t^{\flat},z^{\flat}) + \dfrac{i}{\sqrt{\varepsilon}} p_0(t) \cdot y + i \Lambda^{\varepsilon,\alpha}(s, \eta(y)) \\ + i \Phi^{\varepsilon,\alpha}(\eta(y)) - \dfrac{i}{2} G_{\alpha}(t^{\flat} + \sqrt{\varepsilon}s) y \cdot y + \dfrac{i \alpha}{r \sqrt{\varepsilon}} \eta(y) \cdot \Vec{e}_{\theta} + \sigma.
\end{multline*}
    \item If $s \to + \infty$ (i.e. $\varepsilon \to 0$, $\lvert t - t^{\flat} \lvert \, \leqslant \delta$ and $t > t^{\flat}$)\review{, we have}
\begin{multline*}
\dfrac{i}{\varepsilon} S_{0}(t;t^{\flat},z^{\flat}) + \dfrac{i}{\sqrt{\varepsilon}} p_0(t) \cdot y = \dfrac{i}{\varepsilon}  S_{+}(t;t^{\flat},z^{\flat}) + \dfrac{i}{\varepsilon} p_{+}(t) \cdot (x - q_{+}(t)) + i \Lambda^{\varepsilon,\alpha}(s, \eta(y)) \\ + i \Phi^{\varepsilon,\alpha}(\eta(y)) - \dfrac{i}{2} G_{\alpha}(t^{\flat} + \sqrt{\varepsilon}s) y \cdot y - \dfrac{i \alpha}{r \sqrt{\varepsilon}} \eta(y) \cdot \Vec{e}_{\theta} +  \sigma \text{ and }
\end{multline*}
\begin{multline*}
\dfrac{i}{\varepsilon} S_{0}(t;t^{\flat},z^{\flat}) + \dfrac{i}{\sqrt{\varepsilon}} p_0(t) \cdot y = \dfrac{i}{\varepsilon} \widetilde{S}_{-}(t;t^{\flat},\widetilde{z}^{\flat}) + \dfrac{i}{\varepsilon} \widetilde{p}_{-}(t) \cdot (x - \widetilde{q}_{-}(t)) - \dfrac{i}{\sqrt{\varepsilon}} \delta_{p^{\flat}} \cdot y - i \Lambda^{\varepsilon,\alpha}(s, \eta(y)) \\ - i \Phi^{\varepsilon,\alpha}(\eta(y)) + \dfrac{i}{2} G_{\alpha}(t^{\flat} + \sqrt{\varepsilon}s) y \cdot y + \dfrac{i \alpha}{r \sqrt{\varepsilon}} \eta(y) \cdot \Vec{e}_{\theta} + \sigma.
\end{multline*}
\end{enumerate}

\noindent \review{In the above asymptotics, we recall that \begin{enumerate}[leftmargin=*, labelindent=0pt]
    \item[$\bullet$] the phaes $\Lambda^{\varepsilon,\alpha}$ and $\Phi^{\varepsilon,\alpha}$ are respectively defined in Equations \eqref{phase} and \eqref{phase2},
    \item[$\bullet$] the matrix-valued function $G_{\alpha}$ is introduced in \eqref{def:phase},
    \item[$\bullet$] the matrix $\Gamma_1$ is explicitly given in ~\eqref{matGamma1}.
\end{enumerate}}

\noindent \review{Moreover, the remainder term $\sigma$ appearing in the three previous estimates satisfies \[ \lvert \sigma \lvert \, \leqslant C \Big( \dfrac{1}{|s|} + \sqrt{\varepsilon} |s|^3 + \sqrt{\varepsilon} s^2 \lvert y \lvert \Big). \]}
\end{lemma}

\noindent For the same reasons explained after Lemma \ref{prop:changetime} and an analysis of the sign changes between the plus-phase for $t < t^{\flat}$ and the minus-phase for $t > t^{\flat}$, we deduce that the appropriate drift in the case of a transition from the plus-mode to the minus-mode is $$\delta_{p^{\flat}} = \dfrac{2\alpha}{r} \transpose{\dint w(q^{\flat})} \Vec{e}_{\theta}.$$

\subsection{Adiabatic basis}

\noindent If the initial data of Equation \eqref{equation} satisfies \eqref{IDannexe}, we construct $\vec{Y}_+$ for $t < t^{\flat}$, solution to \eqref{eqVec}, with $\vec{Y}_+(t_0) = \vec{Y}_1$. Proposition \ref{prop:Vec} holds and according to the notations introduced in Subsection \ref{subsec:studyVec} for the eigenprojectors, we can set
\begin{equation}
\label{eq:ICvecafterannexe}
    \Vec{Y}^{\flat\perp} = \Pi_{+,0}^{-,\flat} \Bigg( \lim\limits_{t \to t^{\flat}, \, t < t^{\flat}} \vec{Y}_{+}(t) \Bigg) \quad \text{and} \quad \Vec{Y}^{\flat} = \begin{pmatrix}
            0 & -1 \\
            1 & 0
        \end{pmatrix} \Vec{Y}^{\flat\perp}.
        \tag{B4}
\end{equation}
\noindent In the same way as in Proposition \ref{prop:Vecaftercrossing}, for $t > t^{\flat}$, we construct $\vec{Y}_+$ and $\vec{\widetilde{Y}_-}$, solution to Equation \eqref{eqVec} with 
\begin{equation}
\label{yflatannexe}
    \vec{Y}_+(t^{\flat}) = \Pi_{+,\alpha}^{\flat} \Vec{Y}^{\flat\perp} \quad \text{and} \quad \vec{\widetilde{Y}_-}(t^\flat) = \Pi_{-,\alpha}^{\flat} \Vec{Y}^{\flat}
    \tag{B5}
\end{equation}
\noindent as initial contion at time $t^{\flat}$.

\subsection{Main results}

\noindent The following theorem adapts the proof of Theorem~\ref{theo:ingoingWP} and provides a description of the ingoing wave packet at time $t^{\flat} - \delta$.

\begin{theorem}[The ingoing wave packet] \label{ingoingbis} Let Assumptions \ref{hypsurV2} and \ref{hypcrossing} hold. We consider \( M, \, r_0 \) two positive real numbers. There exists a constant $C > 0$ such that the following holds.

\smallskip

\noindent For all $(\varepsilon,\delta,\alpha_0,R,k) \in (0,1) \times (0,1] \times \R_{+} \times [1, + \infty [ \, \times \, \N \) such that $\sqrt{\varepsilon} \delta^{-1} \leqslant 1$ and $\alpha_0 \leqslant \dfrac{R}{2} \sqrt{\varepsilon}$, for all compact interval\review{s} $\mathrm{I} \subset \R$ with $t_0 = \text{min}(\mathrm{I})$ and $z_0 \in \mathcal{A}_{+}(M,\alpha_0,r_0,\delta,\mathrm{I})$, the solution $\psi^{\varepsilon}$ of \review{Equation} ~\eqref{equation}, \review{corresponding to the initial data \(\psi_0^\varepsilon\) given in Assumption~\eqref{IDannexe}}, satisfies \review{Equation} \eqref{newinconnue} at time $t = t^{\flat} - \delta$ \Big( i.e. at time $- s_0 = - \dfrac{\delta}{\sqrt{\varepsilon}}$ \Big) with 
\begin{equation*}
    u^{\varepsilon}(-s_0,y) = e^{i \Lambda^{\varepsilon,\alpha}(-s_0,\eta(y))} \alpha_1^{\text{in}}(\eta(y)) \Vec{Y}_{\flat}^{\perp} \, +  \, \sigma(\varepsilon, \delta)  \quad \text{where } \| \sigma(\varepsilon, \delta) \|_{_{\widetilde{\Sigma}_{\varepsilon}^{k}}} \leqslant C (\sqrt{\varepsilon} \delta^{-1} + \varepsilon^{\frac{3}{2}} \delta^{-4} + \varepsilon^{-1} \delta^3)(1 + \lvert \ln \delta \lvert).
\end{equation*}
\noindent The vector $\Vec{Y}_{\flat}^{\perp}$ is introduced in Equation \eqref{yflatannexe}, $\eta(y) = \diff w(q^\flat)y$, the phase $\Lambda^{\varepsilon,\alpha}$ is defined by \review{Equation}~\eqref{phase} and 
\begin{multline*}
    \alpha_1^{\text{in}}(\eta(y)) = \text{Exp} \Bigg( \dfrac{i}{\varepsilon} S_{+}^{\flat} + \dfrac{i\alpha^2}{4r \varepsilon} - \dfrac{i \alpha^2}{2r \varepsilon} \ln\Big(\dfrac{\alpha}{2\sqrt{r\varepsilon}} \Big) - \dfrac{i}{2r} (\eta(y) \cdot \Vec{e}_{\theta})^2 - \dfrac{i \alpha}{r \sqrt{\varepsilon}} \eta(y) \cdot \Vec{e}_{\theta}^{\perp} \ln\Big(\dfrac{\alpha}{2\sqrt{r\varepsilon}} \Big)\\ - \dfrac{i}{2r} (\eta(y) \cdot \Vec{e}_{\theta}^{\perp})^2 \ln\Big(\dfrac{1}{2\sqrt{r\varepsilon}} \Big) + \dfrac{i}{2} \Gamma_1 y \cdot y \Bigg) \text{Exp} \Bigg(\dfrac{i \alpha}{r \sqrt{\varepsilon}} \eta(y) \cdot \Vec{e}_{\theta} \Bigg) u_{+}^{\text{in},\alpha}(y)\review{,}
\end{multline*}
\noindent \review{where} $u_{+}^{\text{in},\alpha}$ \review{is} a Schwartz function constructed in Proposition \ref{prop:ingoingprofiles} \review{from} the solution $u_{+}$ of Equation~\eqref{eqProfil} with the initial condition $u_{+}(t_0) = \varphi$.
\end{theorem}

\noindent The final step is to describe the outgoing wave packet as in Theorem~\ref{theo:outgoingWP}. This is the purpose of the following theorem.

\begin{theorem}[The outgoing wave packet] Let Assumptions \ref{hypsurV2} and \ref{hypcrossing} hold. We consider \( M, \, r_0 \) two positive real numbers and $N_0 \in \N^{*}$. There exists a constant $C > 0$ such that the following holds.

\smallskip

\noindent For all $(\varepsilon,\delta,\alpha_0,R,k) \in (0,1) \times (0,1] \times \, \R_{+} \times \, [1, + \infty[ \, \times \, \N \) satisfying conditions \eqref{eq:conditiontheo}, for all compact interval\review{s} $\mathrm{I} \subset \R$ with $t_0 = \text{min}(\mathrm{I})$ and for all $z_0 \in \mathcal{A}_{+}(M,\alpha_0,r_0,\delta,\mathrm{I})$, the solution $\psi^{\varepsilon}$ of Equation~\eqref{equation}, \review{corresponding to the initial data \(\psi_0^\varepsilon\) given in Assumption~\eqref{IDannexe}}, satisfies at time $t = t^{\flat} + \delta$
\begin{equation}
\label{eq:phiannexeaftercrossing}
\psi^{\varepsilon} = \psi^{\varepsilon}_{+} + \psi^{\varepsilon}_{-} + \sigma(\varepsilon, \delta, R, N_0)\review{,} 
\tag{B6}
\end{equation}
\noindent \review{where} for all $x \in \R^d$
\begin{equation*}
    \psi^{\varepsilon}_{-}(t^{\flat} + \delta,x) = e^{\frac{i}{\varepsilon} \widetilde{S}_{-}(t^{\flat} + \delta;t^{\flat},\widetilde{z}^{\flat})} \mathrm{WP}_{\widetilde{\Phi}_-^{t,t^{\flat}}(\widetilde{z}^{\flat})}^{\varepsilon} \Big(\mathrm{Exp} \Big( \frac{i}{2} G_{\alpha}(t^{\flat} + \delta) x \cdot x \Big) \widetilde{u}_{-}^{\text{out},\alpha}(x) \Big) \Vec{Y}_{\flat}^{\perp} + \mathcal{O} \Big( \varepsilon^{\frac{1}{14} - \beta} \Big)
\end{equation*}
\noindent and
\begin{equation*}
    \psi^{\varepsilon}_{+}(t^{\flat} + \delta,x) = e^{\frac{i}{\varepsilon} S_{+}(t^{\flat} + \delta;t^{\flat},z^{\flat})} \mathrm{WP}_{\Phi_+^{t,t^{\flat}}(z^{\flat})}^{\varepsilon} \Big(\mathrm{Exp} \Big( - \frac{i}{2} G_{\alpha}(t^{\flat} + \delta) x \cdot x \Big) u_{+}^{\text{out},\alpha}(x) \Big) \Vec{Y}_{\flat} + \mathcal{O} \Big( \varepsilon^{\frac{1}{14} - \beta} \Big)\review{.}
\end{equation*}
\noindent \review{The quantities involved in the two previous equations are defined as follows}
\begin{itemize}[leftmargin=*, labelindent=0pt]
    \item the flow maps $\Phi_+^{t,t^{\flat}}(z^{\flat}) = (p_{+}, q_{+})$ and $\widetilde{\Phi}_-^{t,t^{\flat}}(\widetilde{z}^{\flat}) = (\widetilde{p}_{-}, \widetilde{q}_{-})$ are defined in \eqref{eqclassicaltraj} and \eqref{eq:drifttrajannexe},
    \item the actions $S_{+}(t;t^{\flat},z^{\flat})$ and $\widetilde{S}_{-}(t;t^{\flat},\widetilde{z}^{\flat})$ are introduced in \eqref{Actions} and \eqref{Actionstildeannexe},
    \item the vectors $\Vec{Y}_{\flat}$ and $\Vec{Y}_{\flat}^{\perp}$ are given in Equation \eqref{yflatannexe},
    \item the symmetric matrix-valued function $G_{\alpha}$ is explicitly given in \eqref{def:phase},
    \item the profiles $(u_{+}^{\text{out},\alpha},\widetilde{u}_{-}^{\text{out},\alpha})$ are Schwartz functions constructed using the ingoing profile $u_{+}^{in,\alpha}$ (given in Theorem \ref{ingoingbis}) \review{and} defined for all $y \in \R^d$ by
    \begin{equation*}
    \begin{pmatrix}
        u_{+}^{\mathrm{out},\alpha}(y) \\
        \widetilde{u}_{-}^{\mathrm{out},\alpha}(y)
    \end{pmatrix}
    =
    \begin{pmatrix}
        e^{i\widetilde{\Lambda}^{\varepsilon,\alpha}(\mathrm{H}(y))} b\left( \frac{\mathrm{H}_2(y)}{\sqrt{r}} + + \frac{\alpha}{\sqrt{r\varepsilon}}\right)
        & 
        a\left( \frac{\mathrm{H}_2(y)}{\sqrt{r}} + \frac{\alpha}{\sqrt{r\varepsilon}} \right) \\
        a\left( \frac{\mathrm{H}_2(y)}{\sqrt{r}} + \frac{\alpha}{\sqrt{r\varepsilon}} \right)
        &
        -e^{-i\widetilde{\Lambda}^{\varepsilon,\alpha}(\mathrm{H}(y))} \overline{b}\left( \frac{\mathrm{H}_2(y)}{\sqrt{r}} + \frac{\alpha}{\sqrt{r\varepsilon}} \right)
    \end{pmatrix}
    \begin{pmatrix}
    e^{\frac{i}{\varepsilon} S_{+}(t^{\flat};t_0,z_0)} u_{+}^{\mathrm{in},\alpha}(y)\\
    0
    \end{pmatrix},
\end{equation*}
with $\mathrm{H}(y) = \Big( \diff w(q^{\flat}) y \cdot \vec{e}_{\theta} , \diff w(q^{\flat}) y \cdot \vec{e}_{\theta}^{\perp} \Big)$, the phase $\widetilde{\Lambda}^{\varepsilon,\alpha}$ defined in \review{Equation}~\eqref{phase} and the functions $a$ and $b$ given in \review{Equation}~\eqref{fonctionaetb}.
\end{itemize}    

\smallskip

\noindent \review{Moreover, the remainder term $\sigma(\varepsilon, \delta, R, N_0)$ appearing in Equation~\eqref{eq:phiannexeaftercrossing} satisfies the following estimate
\[ \| \sigma(\varepsilon, \delta, R, N_0)  \|_{_{\Sigma_{\varepsilon}^{k}}} \leqslant C \Bigg( (\sqrt{\varepsilon} \delta^{-1} + \varepsilon^{\frac{3}{2}} \delta^{-4} + \delta + R \varepsilon^{-1} \delta^3 + R^3 \sqrt{\varepsilon} \delta^{-1} + \varepsilon^{-1} \delta^3 + R^{-N_0})(1 + \lvert \ln \delta \lvert) \Bigg).\]}
\end{theorem}

\section{Additional results on classical trajectories}
\label{appendix:classicalq}

\noindent We use the notations introduced throughout the paper. This appendix provides some relevant properties of the classical trajectories $z_{\pm}$ defined in \review{Equation}~\eqref{eqclassicaltraj}. All of these properties also hold for the trajectories $\widetilde{z}_{\pm}$, which are solutions of \review{System}~\eqref{eqclassicaltrajdrift}.

\begin{prop}
\label{lem:controlCQ}
Let \( M, r_0 \) \review{be two positive real numbers}, \( \delta \in \, (0,1] \) and \( \mathrm{I} \subset \R \) \review{be} a compact interval such that \(t_0 = \text{min}(\mathrm{I}) \). For all \( z_0 \in \mathcal{A}_{\pm}(M,r_0,\delta, \mathrm{I}) \), there exists a constant \( C > 0 \) such that the function \( t \in \mathrm{I} \mapsto \Phi_{\pm}^{t,t_0}(z_0) \) satisfies the following estimate
\[ |q_{\pm}(t) | + |p_{\pm}(t)| \leqslant C, \quad \forall t \in \mathrm{I}.\]
\end{prop}

\begin{proof}[Proof of Proposition \ref{lem:controlCQ}]
    Since $z_0 \in \mathcal{A}_{\pm}(M,r_0,\delta, \mathrm{I})$, the functions $(q_{\pm}, p_{\pm})$ belong to $\mathcal{C}^{\infty}(\mathrm{I}, \mathbb{R}^{2d})$ and satisfy $\dot{q}_{\pm} = p_{\pm}$. Therefore, we have
\[\ddot{q}_{\pm} = \dot{p}_{\pm} = -\nabla \lambda_{\pm}(q_{\pm}).\]

\noindent Multiplying this equation by $\dot{q}_{\pm}$ and integrating from $t_0$ to any $t \in \mathrm{I}$, we obtain the conservation of energy
\[\frac{1}{2} |\dot{q}_{\pm}(t)|^2 + \lambda_{\pm}(q_{\pm}(t)) = \frac{1}{2} |\review{p_0}|^2 + \lambda_{\pm}(q_0).\]

\noindent Using the reverse triangle inequality, we deduce that for all $t \in \mathrm{I}$
\[\frac{1}{2} |\dot{q}_{\pm}(t)|^2 - |\lambda_{\pm}(q_{\pm}(t))| \leq \left| \frac{1}{2} |\review{p_0}|^2 + \lambda_{\pm}(q_0) \right|.\]

\noindent Thus, we have
\[\frac{1}{2} |\dot{q}_{\pm}(t)|^2 \leq |\lambda_{\pm}(q_{\pm}(t))| + |C_{1}|, \quad \text{with } C_{1} = \left| \frac{1}{2} |\review{p_0}|^2 + \lambda_{\pm}(q_0) \right|.\]

\noindent  According to Assumption\review{s} \ref{hypsurV2}, the eigenvalues \(\lambda_{\pm}\) grow at most quadratically. It follows that there exists a constant $C_{2} > 0$ such that
\[|\lambda_{\pm}(q_{\pm})| \leq C_{2} (1 + |q_{\pm}|^2).\]

\noindent Therefore, for all $t \in \mathrm{I}$,
\begin{equation}
\label{majqpoint}
    |\dot{q}_{\pm}(t)|^2 \leqslant C_{3} + C_{2} |q_{\pm}(t)|^2, \quad \text{with } C_{3} = |C_{1}| + C_{2} > 0.
\end{equation}
 Using Gronwall's lemma, it follows that for all $t \in \mathrm{I}$,
\[|q_{\pm}(t)|^2 \leqslant C_{4} (1 + e^{C_{2} |t|}) \leqslant C_5 , \quad \text{with } C_4, C_5 > 0.\]

\noindent Since $\dot{q}_{\pm}$ satisfies \eqref{majqpoint} and $\dot{q}_{\pm} = p_{\pm}$, it follows that
\[ |p_{\pm}(t)| \leqslant C_{6}, \quad \text{with } C_6 > 0 \text{ and for all } t \in \mathrm{I}\]
\noindent and the proposition holds.
\end{proof}

\begin{corollary}
    \label{coro:deriveeCQ} Under the same assumptions as in Proposition \ref{lem:controlCQ}, we consider the following function $f_{\pm} : t \in \mathrm{I} \mapsto w(q_{\pm}(t))$. For $n \in [\![ 0, 2 ]\!]$, there exists $C> 0$ such that \[ \| f_{\pm}^{(n)} \|_{\infty, \mathrm{I}} \leqslant C. \]
\end{corollary}

\begin{proof}[Proof of Corollary \ref{coro:deriveeCQ}]
    This is obtained by applying the previous proposition, computing the first two derivatives of $f_{\pm}$, and using the fact that the function $t \in \mathrm{I} \mapsto \nabla \lambda_{\pm}(q_{\pm}(t))$ is bounded.
\end{proof}

%%%%%%%%%%%%%%%%%%%%%%%%%%%%%%%%%%%%%%%%%%%%%%%%%%%%%%%%%%%%
\newpage
\bibliography{bibliographie}
\bibliographystyle{alpha}

\end{document}